\documentclass[10pt,reqno,b5paper]{amsart}%
\usepackage{anysize}
\marginsize{15mm}{15mm}{10mm}{10mm}
\usepackage[noadjust]{cite}
\usepackage[T1]{fontenc}
\usepackage[colorlinks,urlcolor=black,citecolor=blue,linkcolor=blue]{hyperref}
\usepackage{bm}
\usepackage{amsmath}
\usepackage{mathrsfs}
\usepackage{amsfonts}
\usepackage{amssymb}
\usepackage{amsthm}
\usepackage{graphicx}
\usepackage{cancel}
\usepackage{esint}
\DeclareSymbolFont{CM}{OMX}{cmex}{m}{n}
\DeclareMathSymbol{\sumop}{\mathop}{CM}{"50}
\renewcommand{\sum}{\sumop}

\allowdisplaybreaks[4]

\newtheorem{thm}{\sc Theorem}[section]

\newtheorem{lem}{\sc Lemma}[section]
\newtheorem{pro}{\sc Proposition}[section]
\theoremstyle{definition}
\newtheorem{defn}[thm]{\sc Definition}
\theoremstyle{remark}
\newtheorem{rem}{Remark}[section]
\newtheorem*{clm}{Claim}
\newtheorem*{pf}{Proof}
\newcommand{\Cref}[1]{\eqref{#1}}
\newcommand{\cref}[1]{(\ref{#1})}

\numberwithin{equation}{section}

\begin{document}
\title[ Non-uniqueness for the compressible Oldroyd-type models]
{Sharp Non-uniqueness for the Hyper-viscous compressible 3D Oldroyd-type models: Beyond the Lions exponent}
\author[H. B. Li \quad\&\quad P. Qu\quad\& Z. R. Zeng\quad\&\quad M. X. Zhang]
{Haobin Li $^A$, \quad Peng Qu $^{A,B}$, \quad Zirong Zeng $^C$,\quad Mingxin Zhang $^A$}
\thanks{ E-mails: hbli25@m.fudan.edu.cn  (H. B. Li),\quad pqu@fudan.edu.cn (P. Qu),\quad beckzzr@nuaa.edu.cn  (Z. R. Zeng),\quad mxzhang24@m.fudan.edu.cn  (M. X. Zhang)}
\dedicatory{
\small
$^{A}$ School of Mathematical Sciences, \\Fudan University, Shanghai 200433, P. R. China \\
$^B$ Shanghai Key Laboratory for Contemporary Applied Mathematics, \\Fudan University, Shanghai 200433, P. R. China\\
$^C$ School of Mathematics and Key Laboratory of MIIT,\\Nanjing University of Aeronautics and Astronautics, Nanjing 211106, P. R. China
}
\vspace{-5mm}
\begin{abstract}
We prove the non-uniqueness of weak solutions to the three-dimensional hyper-viscous compressible Oldroyd-type model with general pressure laws and Newtonian viscous. Using an intermittent convex integration scheme, we construct non-unique weak solutions in $L_t^\gamma W_x^{s,p}$ with $(s,p,\gamma)$ lies in the supercritical regime relative to the Ladyzhenskaya-Prodi-Serrin criteria. In addition, the Newtonian viscosity we considered contains the fractional  viscosity $(-\Delta)^\alpha$, where the viscous exponent $\alpha$ can be larger than  the Lions exponent $5/4$. In the classical viscous case $\alpha=1$, our construction also yields sharp non-uniqueness in $L_t^pL_x^\infty$ for $1<p<2$.

Furthermore, we obtain two singular-limit results. First, we prove the strong vanishing viscosity limit for any weak solutions in $H_{t,x}^{\widetilde{\beta}}$ to the compressible fundamental elastodynamic system. Second, $H_{t,x}^{\widetilde{\beta}}$ weak
solutions to the incompressible Oldroyd-type model can be obtain as a low Mach number limit (incompressible limit) of a sequence of weak solutions to the compressible Oldroyd-type model.

The main ingredient of our result is a new cancellation mechanism that overcomes the difficulty caused by the relative rigidity of the pressure. To the best of our knowledge, this is the first non-uniqueness results of weak solutions to the hyper-viscous compressible hydrodynamics equations.
\vskip 2mm
\noindent{Keywords.}  Compressible Oldroyd-type model; Non-uniqueness; Convex integration; Intermittency; Hyper-viscosity. 
\end{abstract}
\maketitle
\tableofcontents
\section{Introduction and main results}
\subsection{Background}
In this paper, we are concerned with the three-dimensional compressible Oldroyd-type model on the torus $\mathbb{T}^3:=[-\pi,\pi]^3$, which is a coupled model from the kinetic theory of dilute polymer solutions. Mathematically, this system reads:
\begin{equation}\label{eq1.1}
     \left\{
    \begin{aligned}
        &\partial_t\rho+{\rm div}(\rho u)=0,\\
        &\partial_t(\rho u)+{\rm div}(\rho u\otimes u)+\nabla P(\rho)={\rm div}\mathbb{S}(u)+{\rm div}(\rho \mathbf{F}\mathbf{F}^{{\rm T}}),\\
        &\partial_t\mathbf{F}+u\cdot\nabla \mathbf{F}-\left(\nabla u\right) \mathbf{F}=0,\\
        &{\rm div}(\rho \mathbf{F}^{{\rm T}})=0,
    \end{aligned}
    \right.
\end{equation}
where $\rho(t,x):[0,T]\times\mathbb{T}^3\rightarrow(0,+\infty)$, $u(t,x):[0,T]\times\mathbb{T}^3\rightarrow\mathbb{R}^3$ and $\mathbf{F}(t,x):[0,T]\times\mathbb{T}^3\rightarrow\mathbb{R}^{3\times3}$ are the mass density, velocity field and deformation gradient tensor, respectively. ${\rm div}(\rho \mathbf{F}\mathbf{F}^{{\rm T}})$ stands for the non-Newtonian elastic force and $\mathbb{S}(u)$ is the Newtonian viscous stress which reads as follows:
\begin{equation}\label{eq1.2}
    \mathbb{S}(u)=2\nu^s\left({\rm D}(u)-\frac{1}{3}{\rm div}u\mathbb{I}\right)+\nu^b{\rm div}u\mathbb{I},
\end{equation}
where ${\rm D}(u)=\frac{1}{2}(\nabla u+\nabla^Tu)$ is the deformation tensor. The coefficients $\nu^s$ and $\nu^b$ are both constants. $\nu^s$ is the shear viscosity coefficient and $\nu^b$ is the bulk viscosity coefficient satisfying the physical assumptions
\begin{align*}
    \nu^s>0,\quad\nu^b+\frac{1}{3}\nu^s\geq0.
\end{align*}
The pressure $P(\rho)$ is a function of the density. As usual, it is assumed to be $C^2$-regular with respect to $\rho$ and satisfy a monotonicity condition $P(\rho)>0$, $P^{\prime}(\rho)>0$ and $P^{\prime\prime}(\rho)>0$.  The classical case of polytropic gases $P(\rho)=A\rho^{\gamma}$ with $A>0$ and $\gamma>1$ naturally satisfies our requirements on the domain without vacuum.

It is noteworthy that the key constraint ${\rm div}(\rho\mathbf{F}^{{\rm T}})=0$ which plays an important role in defining the weak solution is called as the Piola condition. More precisely, according to \cite{LeiZhou2005,QianZhang2010}, if the initial datum satisfy ${\rm div}\left(\rho_0\mathbf{F}^{\rm T}_0\right)=0$, then  ${\rm div}(\rho\mathbf{F}^{{\rm T}})=0$ is preserved for all subsequent times for which the solution exists. To relate this constraint to the underlying continuum-mechanical description, let $\eta(t,X)$ denote the flow map, which maps the reference configuration $\Omega_{{\rm ref}}$ to the current configuration $\Omega_t$, associated with the velocity field $u$, defined by
\begin{align*}
    \left\{
    \begin{aligned}
        &\frac{{\rm d}}{{\rm d}t}\eta(t,X)=u(t,\eta(t,X)),\\
        &\eta(0,X)=\eta_0(X).
    \end{aligned}
    \right.
\end{align*}
The deformation gradient is defined in Lagrangian coordinates by
\begin{align*}
    \mathbf{F}(t,\eta(t,X)):=\nabla_X\eta(t,X).
\end{align*}
Then, According to the conservation law of mass, it holds that
\begin{align*}
    \rho(t,\eta(t,X)){\rm det}\mathbf{F}(t,\eta(t,X))=\rho_{{\rm ref}}(X),
\end{align*}
where $\rho_{{\rm ref}}(X)$ is the reference density. From a physical point of view, ${\rm det}\mathbf{F}(t,\eta(t,X))$ measures the local change of volume induced by the deformation. In particular, the condition ${\rm det}\mathbf{F}(t,\eta(t,X))\equiv1$ corresponds to a volume-preserving deformation and is characteristic of incompressible flows, whereas in the compressible setting ${\rm det}\mathbf{F}(t,\eta(t,X))$ is generally not constrained to be unity. Moreover, a spatially homogeneous reference density, $\rho_{{\rm ref}}(X)\equiv {\rm Const}$, implies
\begin{align*}
  \rho(t,\eta(t,X)){\rm det}\mathbf{F}(t,\eta(t,X))={\rm Const}.
\end{align*}
 For the system \eqref{eq1.1} considered here, the deformation is in general not volume preserving, that is, ${\rm det}\mathbf{F}(t,\eta(t,X))\not\equiv1$, while the mass constraint above remains valid. In addition, provided that the initial data satisfy the corresponding compatibility condition, the Piola condition ${\rm div}(\rho\mathbf{F}^{{\rm T}})=0$ is propagated by the evolution .

Furthermore, to investigate a broader class of dissipative effects in the Oldroyd-type model, we shall consider the fractional Laplacian operator $(-\Delta)^{\alpha}$, $\alpha\in[1,\frac{3}{2})$, which is defined by the Fourier  transform
\begin{align*}
    \mathcal{F}[(-\Delta)^{\alpha}u](\xi)=|\xi|^{2\alpha}\mathcal{F}[u](\xi),\quad\xi\in\mathbb{Z}^3,
\end{align*}
which plays an important role in applied research and has given rise to numerous models.

Then, based on equations \eqref{eq1.1}, the Newtonian stress \eqref{eq1.2} and the fractional Laplacian operator, we can rewrite the equations in the conservation-law form in terms of density, momentum, and the deformation gradient tensor with density (see Appendix \ref{Appendix_A} for details):
\begin{equation}\label{1.3s}
    \left\{
    \begin{aligned}
        &\partial_t\rho+{\rm div}m=0,\\
        &\partial_tm+\nu^s(-\Delta)^{\alpha}(\rho^{-1}m)-(\nu^b+\frac{1}{3}\nu^s)\nabla{\rm div}(\rho^{-1}m)+\nabla P(\rho)\\
        &\qquad\quad\qquad\qquad\qquad+{\rm div}\left(\rho^{-1}\left(m\otimes m-\mathcal{N}\mathcal{N}^{{\rm T}}\right)\right)=0,\\
        &\partial_t\mathcal{N}^i+{\rm div}\left(\rho^{-1}\left(\mathcal{N}^i\otimes m-m\otimes \mathcal{N}^i\right)\right)=0,\quad i=1,2,3,\\
        &{\rm div}\mathcal{N}^i=0,\quad i=1,2,3,
    \end{aligned}
    \right.
\end{equation}
where $m:=\rho u$, $\mathcal{N}:=(\mathcal{N}^1,\mathcal{N}^2,\mathcal{N}^3):=(\rho \mathbf{F}^1,\rho \mathbf{F}^2,\rho \mathbf{F}^3)=\rho \mathbf{F}$.

In addition, we note that when the viscosity vanish $(\nu^b=\nu^s=0)$, \eqref{1.3s} reduce to the following compressible fundamental elastodynamic system (see for \cite{HuTuWangWen-2025-arxiv}),
\begin{equation}\label{1.4ss}
    \left\{
    \begin{aligned}
        &\partial_t\rho+{\rm div}m=0,\\
        &\partial_tm+{\rm div}\left(\rho^{-1}\left(m\otimes m-\mathcal{N}\mathcal{N}^{{\rm T}}\right)\right)+\nabla P(\rho)=0,\\
        &\partial_t\mathcal{N}^i+{\rm div}\left(\rho^{-1}\left(\mathcal{N}^i\otimes m-m\otimes \mathcal{N}^i\right)\right)=0,\quad i=1,2,3,\\
        &{\rm div}\mathcal{N}^i=0,\quad i=1,2,3.
    \end{aligned}
    \right.
\end{equation}

Moreover, If we consider the dimensionless form of \eqref{1.3s} (we only focus on the Mach number $\mathrm{Ma}:=\sqrt{(\varrho_0U^2)/K_0}$ (see for \cite{kinsler2000fundamentals}), where $\varrho_0$, $U$ and $K_0$ are the characteristic density, characteristic velocity and bulk modulus respectively, and, set other dimensionless numbers to $1$, see Appendix \ref{Appendix_B} for more details), then the system \eqref{1.3s} with dimensionless pressure $\frac{1}{\mathrm{Ma}^2}\nabla P(\rho)$ will reduce to the incompressible Oldroyd-type model (see for \cite{ChenLiu2022}) at low Mach number, that is, ${\rm Ma}\rightarrow0$,
\begin{equation}\label{1.5-s}
    \left\{
    \begin{aligned}
        &\partial_tv+(-\Delta)^{\alpha}v+\nabla\Pi+{\rm div}\left(v\otimes v-\mathbf{H}\mathbf{H}^{{\rm T}}\right)=0,\\
        &\partial_t\mathbf{H}^i+{\rm div}\left(\mathbf{H}^i\otimes v-v\otimes  \mathbf{H}^i\right)=0,\quad i=1,2,3,\\
        &{\rm div}v={\rm div}\mathbf{H}^i=0,\quad i=1,2,3,
    \end{aligned}
    \right.
\end{equation}
where $v$, $\mathbf{H}$ and $\Pi$ are the incompressible velocity, elastic stress and pressure respectively.

The Oldroyd-type model was first proposed by Oldroyd in his seminal paper \cite{Oldroyd1950}. Interestingly, the Oldroyd-type model can also be derived from a coarse-grained, microscopic model by modeling dilute polymer molecule as a dilute suspension of non-interacting Hookean 'dumbbells' in a Newtonian solvent, each polymer molecule chain being idealized as a dumbbell comprising an infinitely extensible Hookean spring connecting two beads, and defining elastic stress $\tau$ via a weighted average of $\psi$ (the probability density function of the (random) conformation vector of the polymer molecules), see \cite{BarrettSuli2018,ConstantinKliegl2012,Pravia2002}. For more physical and mechanical background of \eqref{eq1.1}, we refer to \cite{BirdArmstrongHassager1987,Larson1988}.

In the pioneering paper by Guillop\'e--Saut \cite{GuillopeSaut1990}, the unique local strong solution to an initial-boundary value problem with large data was established in the $H^s$ Sobolev space. Moreover, the global strong solution with small initial data and coupling parameter was also obtained in \cite{GuillopeSaut1990}. Later, Molinet--Talhouk \cite{MolinetTalhouk2004} improved their results by removing the small restriction on the coupling parameter. More extending results on this direction can be found in \cite{CheminMasmoudi2001,LeiZhou2005,FangHieberZi2013,HieberNaitoShibata2012,Zhu2018,ZiFangZhang2014} and among others. In particular, Chemin--Masmoudi \cite{CheminMasmoudi2001} gave a BKM-type blow-up criterion of regular solutions, see also \cite{LeiMasmoudiZhou2010} for a refined version. 

For the compressible case, Bollada--Phillips \cite{BolladaPhillips2012} developed a mathematical framework for compressible viscoelastic fluids. And then Hu--Wang \cite{HuWang2011,HuWang2012} established global solutions for multi-dimensional compressible viscoelastic flows near equilibrium in critical Besov spaces and proved the local existence and uniqueness of strong solutions in three dimensions. Hu--Wu \cite{HuWu2013} further obtained global strong solutions near equilibrium and established their optimal time-decay rates.  Later, Zhou--Zhu--Zi \cite{ZhouZhuZi2018} proved that the three dimensional compressible Oldroyd-type model could admit a unique global strong solution provided the initial data are close to the constant equilibrium state in $H^2$-framework. Moreover, Cai--Lei--Lin--Masmoudi \cite{CaiLeiLinMasmoudi2019} proved global existence which can justify in particular the vanishing viscosity limit for all time. For more results, see also \cite{BarrettLuSuli2017, LuPokorny2020,ZhaiLi2021,LiuLuWen2021,ZhaiChen2024}.

Compared with the well-posedness theory, even less is know for the flexibility and non-uniqueness for the Oldroyd-type model. Since the seminal works of De Lellis--Sz\'ekelyhidi \cite{DS-2009-ann,DS-2010-arma}, which introduced the convex integration framework for constructing infinitely many weak solutions to the incompressible Euler equations, substantial progress has been made on non-uniqueness and flexibility phenomena for a broad class of fluid equations. A major milestone in this development was the resolution of the flexible side of Onsager's conjecture, see for \cite{Isett-2018-ann,BDS-2019-cpam}. For the viscous fluids, Buckmaster--Vicol \cite{BV-2019-annmath} established, in a breakthrough work, the non-uniqueness of weak solutions to the three-dimensional incompressible Navier-Stokes equations. More recently, Li--Qu--Zeng--Zhang \cite{LQZZ-2024-jmpa} obtained sharp non-uniqueness results for the three-dimensional incompressible hyper-dissipative (beyond the Lions exponent) Navier-Stokes equations in the Ladyzhenskaya--Prodi--Serrin supercritical spaces $L_t^\gamma W_x^{s,p}$. Cheskidov--Dai--Palasek \cite{CDP-2026-arxiv} further demonstrated instantaneous Type-I blow-up and non-uniqueness for smooth solutions to the Navier-Stokes equations through an inverse energy cascade mechanism. These ideas have since been extended to other equations, including hyper-viscous three-dimensional Navier-Stokes equations \cite{LQZZ-2024-jmpa,lt20}, two-dimensional incompressible models in the low-viscosity regime \cite{lq20}, transport equations \cite{CL-2021-annpde,cl22,ms18}, stationary Navier-Stokes equations \cite{luo19} and MHD equations \cite{LZZ-2022-jmpa,MiaoYe2024,LZZ-2024-jfa,NY-2025-jns,MNY-2025-annpde }. However, for the Oldroyd-type models, as we know, only Chen--Liu \cite{ChenLiu2022} proved that there exist non-trivial weak solutions of the three-dimensional hypo-viscous elastodynamics with finite kinetic energy by using convex integration.

Even more scarce is the non-uniqueness for the compressible viscous fluids due to the relative rigidity of pressure $\nabla P(\rho)$ must be taken into account. Currently, only Li--Qu--Zeng-Zhang \cite{LQZZ-2022-arxiv} and Li--Qu \cite{LiQu2026} provided the non-uniqueness via $L^p_tL^{\infty}_x$ with $p<2$ to the hypo-dissipative compressible Navier-Stokes equations and MHD equations respectively. However, to the best of our knowledge, there remains no non-uniqueness results for compressible hyper-viscous fluids.

\subsection{Main results}
To begin with, let us formulate precisely the definition of weak solutions in the distributional sense to the equations \eqref{1.3s}.
\begin{defn}\label{def-weak-solutions}(Weak solutions). Let $T\in(0,+\infty)$. Given any initial datum $\rho_0\in L^{\infty}(\mathbb{T}^3)$, $\rho_0>0$, and $u_0,\mathbf{F}_0\in L^2(\mathbb{T}^3)$ with ${\rm div}\left(\rho_0\mathbf{F}_0\right)=0$, we say that $(\rho,m,\mathcal{N})\in L^{\infty}([0,T]\times\mathbb{T}^3)\times \left(L^{2}([0,T]\times\mathbb{T}^3)\right)^2$ is a weak solution for hyper-viscous compressible Oldroyd-type model \eqref{1.3s} if
\begin{itemize}
    \item $\rho\geq0$ a.e. and 
    \begin{align*}
        \int_{\mathbb{T}^3}\rho_0(x)\psi(0,x){\rm d}x=-\int_0^T\int_{\mathbb{T}^3}\rho\partial_t\psi+(m\cdot\nabla)\psi{\rm d}x{\rm d}t,
    \end{align*}
    for any test function $\psi\in C^{\infty}_0([0,T)\times\mathbb{T}^3;\mathbb{R})$.\\
    \item $m=0$ and $\mathcal{N}=0$ whenever $\rho=0$, and
    \begin{align*}
        \int_{\mathbb{T}^3}m_0\cdot\psi(0,x){\rm d}x&
        =-\int_0^T\int_{\mathbb{T}^3}m\cdot\partial_t\psi-\rho^{-1}m\cdot\left(\nu^s(-\Delta)^{\alpha}\psi-(\nu^b+\frac{1}{3}\nu^s)\nabla{\rm div}\psi\right)+P{\rm div}\psi{\rm d}x{\rm d}t\\
        &\quad-\int_0^T\int_{\mathbb{T}^3}\left(\rho^{-1}m\otimes m-\mathcal{N}\mathcal{N}^{\rm T}\right):\nabla\psi{\rm d}x{\rm d}t,
    \end{align*}
    for any test function $\psi\in C^{\infty}_0([0,T)\times\mathbb{T}^3;\mathbb{R}^3)$, where $m_0=\rho_0u_0$ and $\mathcal{N}=\rho\mathbf{F}$.\\
    \item $\mathcal{N}^i$ are all divergence-free and 
    \begin{align*}
        \int_{\mathbb{T}^3}\mathcal{N}^i_0\cdot\psi(0,x){\rm d}x=-\int_0^T\int_{\mathbb{T}^3}\mathcal{N}^i\cdot\partial_t\psi+\rho^{-1}\left(\mathcal{N}^i\otimes m-m\otimes \mathcal{N}^i\right):\nabla\psi{\rm d}x{\rm d}t,
    \end{align*}
    for any test function $\psi\in C^{\infty}_0([0,T)\times\mathbb{T}^3;\mathbb{R}^3)$, where $\mathcal{N}=\left(\mathcal{N}^1,\mathcal{N}^2,\mathcal{N}^3\right)$.
\end{itemize}
\end{defn}

The non-uniqueness of weak solutions for hyper-dissipative compressible Oldroyd-type model \eqref{1.3s} is summarized in the following Theorem \ref{thm-nonuniquenes-1}.
\begin{thm}\label{thm-nonuniquenes-1}
    Let $\alpha\in[1,\frac{3}{2})$. Let $\widetilde{\mathcal{N}}$ be any divergence-free tensor field and $(\widetilde{\rho},\widetilde{m})$ be any smooth solution to the transport equation
    \begin{equation}\label{1.6-s}
        \partial_t\widetilde{\rho}+{\rm div}\widetilde{m}=0,    
    \end{equation}
  on $\mathbb{T}^3$, such that $0<c_*\leq\widetilde{\rho}(t,x)\leq C_*$ for all $(t,x)\in[0,T]\times\mathbb{T}^3$, where $c_*$ and $C_*$ are both universal constants. Then, there exists $\beta^{\prime}\in(0,1)$, such that for any $\varepsilon_*>0$ and any exponents $(s,p,\gamma)$ satisfying 
  \begin{equation}\label{1.7-s}
      (s,p,\gamma)\in\mathcal{A}:=\left\{(s,p,\gamma)\in[0,\frac{3}{2})\times[1,+\infty)\times[1,+\infty):0\leq s<\frac{2\alpha}{\gamma}+\frac{2\alpha-2}{p}+1-2\alpha\right\},
  \end{equation}
  there exist $(\rho,m,\mathcal{N})$ to \eqref{1.3s} such that the following holds:
 
(i) Weak solutions: $(\rho, m,N)$ is the weak solution to \eqref{1.3s} in the sense of Definition \ref{def-weak-solutions}. 

(ii) Regularity:
$$
\rho \in C_t C_x^1, \quad m,\mathcal{N} \in H_{t,x}^{\beta^{\prime}}  \cap L^{\gamma}_tW_{x}^{s,p}\cap C_tH^{-\frac{3}{2}}_x.
$$

(iii) Mass preservation:
$$
\int_{\mathbb{T}^3} \rho(t, x) \mathrm{d} x=\int_{\mathbb{T}^3} \widetilde{\rho}(t, x) \mathrm{d} x,\quad \forall t \in[0, T] .
$$

(iv) Small deviation of norms:
\begin{align*}
\|\rho-\widetilde{\rho}\|_{C_t C_x^1} \leq \varepsilon_*,\quad\|m-\widetilde{m}\|_{L_t^1 L^2_x} \leq \varepsilon_*,\quad\|\mathcal{N}-\widetilde{\mathcal{N}}\|_{L_t^1L^2_x} \leq \varepsilon_*.
\end{align*}

(v) Small deviation of temporal support: if $\widetilde{T}:=\inf _{t \in[0, T]}\{t \mid \nabla \widetilde{\rho}(t, \cdot) \neq 0, \widetilde{m}(t, \cdot)\neq 0, \widetilde{\mathcal{N}}(t,\cdot)-\overline{\mathcal{N}} \neq 0\}>0$, then 
$$
\operatorname{supp}_t(\nabla \rho, m,\mathcal{N}-\overline{\mathcal{N}}) \subseteq \mathcal{N}_{\varepsilon_*}([\widetilde{T}, T]),
$$
where 
$$
\overline{\mathcal{N}}:=\fint_{\mathbb{T}^3}\mathcal{N}(0,x){\rm d}x.
$$
Note that, for any $A \subseteq[0, T]$ and $\varepsilon_*>0, N_{\varepsilon_*}(A)$ denotes the $\varepsilon_*$-neighborhood of $A$ in $[0, T]$, namely,
$$
N_{\varepsilon_*}(A):=\left\{t \in[0, T]: \exists s \in A \text {, s.t. }|t-s| \leq \varepsilon_*\right\} .
$$
\end{thm}

As a consequence, Theorem \ref{thm-nonuniquenes-1} gives the following result concerning the non-uniqueness of weak solutions.
\begin{thm}\label{thm-nonuniqueness-2}
    There exist smooth initial datum $\left(\rho_0,m_0,\mathcal{N}_0\right)$, $\rho_0>0$ such that for any exponents $(s,p,\gamma)$ satisfying \eqref{1.7-s}, there exist infinitely many weak solutions $(\rho,m,\mathcal{N})\in C_tC_x^1\times\left(H^{\beta^{\prime}}_{t,x}\cap L^{\gamma}_tW^{s,p}_x\right)^2$ to the hyper-viscous compressible Oldroyd-type model \eqref{1.3s} with the same initial datum $(\rho_0,m_0,\mathcal{N}_0)$ for some $\beta^{\prime}>0$.
\end{thm}

Furthermore, we also consider the strong vanishing viscosity  limit, which relates the hyper-viscous compressible Oldroyd-type model and  compressible fundamental elastodynamic system.
\begin{thm}\label{thm-vanishing limit}(strong vanishing viscosity and resistivity limit). Let $\alpha\in[1,\frac{3}{2})$ and $(\rho,m,\mathcal{N})\in C^{\widetilde{\beta}}_{t,x}\times\left(H^{\widetilde{\beta}}_{t,x}\right)^2$ where $\widetilde{\beta}>0$, be any weak solution to the compressible fundamental elastodynamic system \eqref{1.4ss}, such that $c_1\leq\rho\leq c_2$ for some constants $c_1,c_2>0$. Then there exist $\beta^{\prime}\in(0,\widetilde{\beta})$ and a sequence of weak solutions $(\rho^{(n)},m^{(n)},\mathcal{N}^{(n)})\in C_{t,x}\times \left(H^{\beta^{\prime}}_{t,x}\right)^2$ to the hyper-viscous compressible Oldroyd-type model \eqref{1.3s} with viscous coefficients $\kappa_n\nu^s$ and $\kappa_n\nu^b$ such that
\begin{align}\label{1.8-s}
    \rho^{(n)}\rightarrow\rho\text{ strongly in }C_{t,x},\text{ and }\left(m^{(n)},\mathcal{N}^{(n)}\right)\rightarrow \left(m,\mathcal{N}\right)\text{ strongly in }\left(H^{\beta^{\prime}}_{t,x}\right)^2\text{ as }\kappa_n\rightarrow0.
\end{align}
\end{thm}

Finally, we interpret the low Mach number limit in the sense of the bulk modulus $K_0\rightarrow\infty$, which precisely corresponds to the transition of the fluid from a compressible to and incompressible regime, that is, the incompressible limit below. 
\begin{thm}\label{thm-low mach limit}(incompressible limit). 
Let $\alpha\in[1,\frac{3}{2})$ and $(v,\mathbf{H})\in \left(H^{\widetilde{\beta}}_{t,x}\right)^2$ where $\widetilde{\beta}>0$, be any weak solution to the incompressible Oldroyd-type model \eqref{1.5-s}. Then there exist $\beta^{\prime}\in(0,\widetilde{\beta})$ and a sequence of weak solutions $(\rho^{(n)},m^{(n)},\mathcal{N}^{(n)})\in C_{t,x}\times \left(H^{\beta^{\prime}}_{t,x}\right)^2$ to the hyper-viscous compressible Oldroyd-type model \eqref{1.3s} as dimensionless form with Mach numbers $\varkappa_n{\rm Ma}$ such that
\begin{align}\label{1.9-s}
    \rho^{(n)}\rightarrow1\text{ strongly in }C_{t,x},\text{ and }\left(m^{(n)},\mathcal{N}^{(n)}\right)\rightarrow \left(v,\mathbf{H}\right)\text{ strongly in }\left(H^{\beta^{\prime}}_{t,x}\right)^2\text{ as }\varkappa_n\rightarrow0.
\end{align}
\end{thm}

\noindent\textbf{Comments on main results.} Let us present some comments on the main results below.

(i)\textbf{Non-uniqueness of weak solutions for hyper-viscous compressible fluids.} To the best of our knowledge, Theorem \ref{thm-nonuniqueness-2} (or the more general Theorem \ref{thm-nonuniquenes-1}) gives the first result for the non-uniqueness of weak solutions to the compressible hyper-viscous fluids.

This also shows that hyper-viscosity (even beyond the Lions exponent) cannot rule out the non-uniqueness mechanism, which is in fact consistent with the Ladyzhenskaya-Prodi-Serrin criterion. 

 On the other hand, in the hypo-viscosity regime ($\alpha\in(0,1)$), one may still employ the methods developed in \cite{LiQu2026}, together with a similar convex integration scheme, to obtain non-uniqueness of weak solutions to the 3D compressible Oldroyd-type model \eqref{1.3s} in the class $L_t^2L_x^\infty$.

(ii)\textbf{Sharp viscosity threshold for $L^2_tC_x$ well-posedness.} It usually believed that, stronger dissipation generally tends to promote well-posedness. In view of the endpoint $(s,p,\gamma)=(0,\infty,2)$ of the Ladyzhenskaya-Prodi-Serrin criterion, Theorem \ref{thm-nonuniqueness-2} also gives the non-uniqueness of weak solutions in the space $L^p_tL^{\infty}_x$ ($1<p<2$) for the compressible Oldroyd-type model with the classical viscosity $(\alpha=1)$. 

In contrast to the non-uniqueness results for the compressible Navier-Stokes and MHD systems, our result applies to the classical viscosity regime rather than being restricted to the hypo-viscosity regime, which suggests that the compressible Oldroyd-type model exhibits weaker rigidity. This phenomenon also provides, at least to some extent, a mathematical explanation for the physical mechanism by which the effect of viscosity is weakened in hyper-elastic fluids.

(iii)\textbf{Singular limits for hyper-viscous compressible Oldroyd-type model.} Singular limits are fundamental problem in the study of PDEs in fluid mechanics, as they provide a mathematical framework for understanding the connections between different fluid models. Typical examples include the vanishing viscosity limit, which relates viscous and inviscid fluids; the Newtonian limit, which connects non-Newtonian and Newtonian fluids; and the incompressible limit, which links compressible and incompressible flows.

 In particular, Theorem \ref{thm-vanishing limit} shows that a $C^{\widetilde{\beta}}_{t,x}\times\left(H^{\widetilde{\beta}}_{t,x}\right)^2$ ($\widetilde{\beta}>0$) weak solutions to the inviscid compressible Oldroyd-type model \eqref{1.4ss} can be obtained as vanishing viscosity limits of weak solutions to \eqref{1.3s}.

More interestingly, Theorem \ref{thm-low mach limit} provides, to the best of our knowledge, the first low Mach number limit for the hyper-viscous compressible Oldroyd-type model \eqref{1.3s} in a low regularity regime. By interpreting the limit ${\rm Ma}\rightarrow0$ as the bulk modulus $K_0\rightarrow\infty$, this result can equivalently be viewed as an incompressible limit. See also \cite{chen2026universalitylowmachnumber} for result from compressible to incompressible Euler equations.
\\ \hspace*{\fill}\\
\noindent\textbf{New methods for hyper-viscous compressible fluids.} A primary difficulty in the compressible fluids, whether in the hyper-viscous or hypo-viscosity regime, is that the pressure law in the compressible case depends on the density, which propagates along the transport equation. In the convex integration scheme, this reflects the relative rigidity for the choice of the pressure in the compressible case, while it is rather
flexible in the incompressible case (see \cite{LQZZ-2022-arxiv} for more details). 

Furthermore, compared with the hypo-viscosity regime, the hyper-viscous fluids requires a stronger intermittency mechanism in order to control the hyper-dissipative term $(-\Delta)^{\alpha}$ where $\alpha\in[1,\frac{3}{2})$. This, in turn, makes the comparatively rigid structure of the pressure term a more serious obstruction.

In this paper, we do face this kind of difficulties. We introduce a new cancellation mechanism and combine it with Mikado flows exhibiting both temporal and spatial intermittency. By dimensional analysis (refer to Appendix \ref{Appendix_B} for details), we observe that \eqref{1.3s} has no Deborah number or Weissenberg number which used to characterize the competition between elastic and viscous effects. This absence arises from the fact that \eqref{1.3s} intrinsically possesses no relaxation time, that is, the relaxation time effectively tend to infinity. Within the framework of Maxwell and Oldroyd-B models, this scenario corresponds precisely to the asymptotic limit in which the Deborah number (or Weissenberg number) approaches infinity. Consequently, this further substantiates that elastic force are predominant in \eqref{1.3s}, that is, the viscous mechanism is weakened, which may lead to stronger non-uniqueness (or ill-posedness). More precisely, the new geometric Lemma \ref{lem-2nd-geometric} provided a new canceled mechanism which allows us to construct suitable perturbations
    \begin{align*}
        w_{q+1}\simeq m_{q+1}-m_q,\quad f_{q+1}\simeq\mathcal{N}_{q+1}-\mathcal{N}_q,\quad z_{q+1}\simeq\rho_{q+1}-\rho_q,
    \end{align*}
such that the nonlinear effects decrease the amplitudes of the momentum Reynolds stress without damaging the relative rigidity of pressure. By using Lemma \ref{lem-2nd-geometric},  we may expect that
\begin{align*}
    w_{q+1}\otimes w_{q+1}-\sum_{i=1}^3f^i_{q+1}\otimes f^i_{q+1}\simeq R_q
\end{align*}
rather than
\begin{align*}
    w_{q+1}\otimes w_{q+1}\simeq R_q+\varrho{\rm Id}
\end{align*}
in the compressible Navier-Stokes equations \cite{LQZZ-2022-arxiv} and MHD equations \cite{LiQu2026}. 
\\ \hspace*{\fill}\\
\noindent\textbf{Organization.} Now we introduce the organization of the paper.  In Section \ref{sec2}, we present the main iteration estimates, which play a crucial role in the subsequent proof of Theorem \ref{thm-nonuniquenes-1}. The mollification procedure is also employed to avoid the loss of derivatives. In Section \ref{sec3}, we construct the key building blocks and apply them to the construction of the momentum, elastic field, and density perturbations. Corresponding estimates and certain cancellation identities are provided. In Section \ref{sec4}, we focus on the treatment of the Reynolds stress and elastic stress fluctuation, and prove the relevant inductive estimates. Finally, in Section \ref{sec5}, we establish the main results.
\\ \hspace*{\fill}\\
\noindent\textbf{Notations.}
We denote for $ p \in[1, \infty], s \in \mathbb{R}, N \in \mathbb{N}$ and $\eta \in(0,1)$,
$$
L_t^p:=L^p(0, T), \quad L_x^p:=L^p(\mathbb{T}^3), \quad C_x^N:=C^N(\mathbb{T}^3), \quad C_x^{N, \eta}:=C^{N, \eta}(\mathbb{T}^3), \quad W_x^{s, p}:=W^{s, p}(\mathbb{T}^3),
$$
where $W_x^{s, p}$ is the usual Sobolev space, $C_x^{N, \eta}$ is the H\"older space equipped with the norm
$$
\|u\|_{C_x^{N, \eta}}:=\sum_{0 \leq|\zeta| \leq N}\left\|\nabla^\zeta u\right\|_{C_x}+\max _{|\zeta|=N} \sup _{x \neq y \in \mathbb{T}^3} \frac{\left|\nabla^\zeta u(x)-\nabla^\zeta u(y)\right|}{|x-y|^\eta},
$$
and $\zeta=\left(\zeta_1, \zeta_2,\zeta_3\right)$ is the multi-index with $\nabla^\zeta:=\partial_{x_1}^{\zeta_1} \partial_{x_2}^{\zeta_2}\partial_{x_3}^{\zeta_3}$. When $N=0$, we denote $C_x^\eta:=C^{0, \eta}\left(\mathbb{T}^3\right)$ for brevity. We also use the shorthand notation $L_t^\gamma L_x^p$ to denote $L^\gamma\left(0, T ; L^p\left(\mathbb{T}^3\right)\right)$, where $p, \gamma \in[1, \infty]$. In particular, we write $L_{t, x}^p:=L_t^p L_x^p$ for short. Moreover, let
$$
\|u\|_{W_{t, x}^{N, p}}:=\sum_{0 \leq m+|\zeta| \leq N}\left\|\partial_t^m \nabla^\zeta u\right\|_{L_{t, x}^p}, \quad\|u\|_{C_{t, x}^N}:=\sum_{0 \leq m+|\zeta| \leq N}\left\|\partial_t^m \nabla^\zeta u\right\|_{C_{t, x}},
$$
Given any Banach space $X, C([0, T] ; X)$ denotes the space of continuous functions from $[0, T]$ to $X$, equipped with the norm $\|u\|_{C_t X}:=\sup _{t \in[0, T]}\|u(t)\|_X$.

We would also write $A \lesssim B$ to imply that $A \leq C B$ for some constant $C>0$ independent of the parameters $q$, $a$, $b$ and $\beta$.
\section{Main iteration and mollification}\label{sec2}
In this section, we will introduce the main iteration of density, momentum, elastic field, Reynolds stress and elastic stress fluctuation, which is the heart of the proof of main results.

\subsection{Main iteration}
For each integer $q\in\mathbb{N}$, we consider the following relaxation system
\begin{equation}\label{eq2.1}
    \left\{
    \begin{aligned}
        &\partial_t\rho_q+{\rm div}m_q=0,\\
        &\partial_tm_q+\nu^s(-\Delta)^{\alpha}(\rho^{-1}_qm_q)-(\nu^b+\frac{1}{3}\nu^s)\nabla{\rm div}(\rho^{-1}_qm_q)+\nabla P(\rho_q)\\
        &\qquad+{\rm div}\left(\rho_q^{-1}\left(m_q\otimes m_q-\mathcal{N}_q\mathcal{N}_q^{\rm T}\right)\right)={\rm div}R^m_q,\\
        &\partial_t\mathcal{N}^i_q+{\rm div}\left(\rho_q^{-1}\left(\mathcal{N}^i_q\otimes m_q-m_q\otimes \mathcal{N}^i_q\right)\right)={\rm div}M^i_q,\quad i=1,2,3,\\
        &{\rm div}\mathcal{N}^i_q=0,\quad i=1,2,3,\\
        &\mathcal{N}_q=\left(\mathcal{N}_q^1,\mathcal{N}^2_q,\mathcal{N}^3_q\right),\quad\left(\mathscr{M}_q\right)_{ijk}:=\left(M^i_q\right)_{jk},
    \end{aligned}
    \right.
\end{equation}
where $R^m_q$ and $\mathscr{M}_q$ are the $3\times3$ symmetric Reynolds stress and $3\times3\times3$ elastic stress fluctuation respectively, and $M^i_q$ (for simplicity, we also refer to it as the elastic stress fluctuation) is $3\times3$ skew-symmetric matrix.

In order to measure the size of relaxed solutions $(\rho_q,m_q,\mathcal{N}_q,R^m_q,\mathscr{M}_q)$ accurately, we need to carefully select the frequency parameter $\lambda_q$ and amplitude parameter $\delta_q$. More precisely, let $a\in\mathbb{N}$ is a sufficiently large integer to be determined later, $\varepsilon\in\mathbb{Q}$ is sufficiently small such that
\begin{equation}\label{eq2.2s}
    \varepsilon\leq\frac{1}{100}{\rm min}\left\{\frac{3}{2}-\alpha,\frac{2\alpha}{\gamma}+\frac{2\alpha-2}{p}-(2\alpha-1)-s\right\},\quad  b(2-\alpha-8\varepsilon)\in\mathbb{N},
\end{equation}
with $b\in2\mathbb{N}$ is a large even number and $\beta>0$ is the regularity parameter such that
\begin{equation}\label{eq2.3s}
    b>\frac{3000}{\varepsilon},\quad0<\beta<\frac{1}{100b^2}.
\end{equation}
For $q\in\mathbb{N}$, the frequency parameter $\lambda_q$ and amplitude parameter $\delta_q$ are defined by
\begin{equation}\label{eq2.4}
    \lambda_q:=a^{(b^q)},\quad \delta_q:=\lambda_q^{-2\beta}\lambda_1^{3\beta},
\end{equation}
where the constant $a>0$ will be chosen later. Furthermore, we assume that the following crucial inductive estimates hold for the relaxed solutions to \eqref{eq2.1} at level $q\in\mathbb{N}$:
\begin{align}
    &C_1+\lambda_q^{-\beta}\leq\rho_q\leq C_2-\lambda_q^{-\beta},\label{eq2.5}\\
    &\|\partial_t^M\rho_q\|_{C_tC^N_x}\lesssim\lambda_q^{\frac{N\varepsilon}{4}},\label{eq2.6}\\
    &\|m_q\|_{C^N_{t,x}}+\|\mathcal{N}_q\|_{C^N_{t,x}}\lesssim\lambda_q^{4N+3},\label{eq2.7}\\
    &\|R^m_q\|_{C^1_{t,x}}+\sum_{i=1}^3\|M^i_q\|_{C^1_{t,x}}\lesssim\lambda_q^{20},\label{eq2.8}\\
    &\|R^m_q\|_{L^1_{t,x}}+\sum_{i=1}^3\|M^i_q\|_{L^1_{t,x}}\lesssim\delta_{q+1},\label{eq2.9}
\end{align}
where $0\leq M\leq1$, $0\leq N\leq4$, $C_1$ and $C_2$ are positive universal constants, and the implicit constants are independent of $q$.

The main iteration result is contained in the following theorem.
\begin{thm}\label{thm-main-iteration} (Main iteration). 
    Let $(s,p,\gamma)$ satisfy \eqref{1.7-s}. Then, there exist $\beta\in(0,1)$ and $a_0$ large enough, such that for any integer $a\geq a_0$, the following holds:

    Suppose that $(\rho_q,m_q,\mathcal{N}_q,R^m_q,\mathscr{M}_q)$ solves \eqref{eq2.1} and satisfies \eqref{eq2.5}--\eqref{eq2.9}. Then, there exists a new relaxation solution $(\rho_{q+1},m_{q+1},\mathcal{N}_{q+1},R^m_{q+1},\mathscr{M}_{q+1})$ to \eqref{eq2.1} which satisfies \eqref{eq2.5}--\eqref{eq2.9} with $q+1$ replacing $q$ and the following estimates:
    \begin{align}
        &\int_{\mathbb{T}^3}\rho_{q+1}(t,x){\rm d}x=\int_{\mathbb{T}^3}\rho_{q}(t,x){\rm d}x,\quad\forall t\in[0,T],\label{eq2.10}\\
        &\|\rho_{q+1}-\rho_q\|_{C_tC_x^1}\lesssim\delta_{q+2}^{\frac{1}{2}},\label{eq2.11}\\
        &\|m_{q+1}-m_q\|_{L^2_{t,x}}+\|\mathcal{N}_{q+1}-\mathcal{N}_q\|_{L^2_{t,x}}\lesssim\delta_{q+1}^{\frac{1}{2}},\label{eq2.12}\\
        &\|m_{q+1}-m_q\|_{L^1_tL^2_x}+\|\mathcal{N}_{q+1}-\mathcal{N}_q\|_{L^1_tL^2_x}\lesssim\delta_{q+2}^{\frac{1}{2}},\label{eq2.13}
\end{align}
where the implicit constants are independent of $q$. In addition, concerning the temporal support of relaxed solutions, we denote
\begin{equation}\label{eq2.14}
    T_q:=\inf\left\{t\in[0,T]\mid \left(\nabla\rho_q,m_q,\mathcal{N}_q-\overline{\mathcal{N}},R^m_q,\mathscr{M}_q\right)(t,\cdot)\neq0 \right\}.
\end{equation}
for all $q\in\mathbb{N}$ where $\overline{\mathcal{N}}=\fint_{\mathbb{T}^3}\mathcal{N}(0,x){\rm d}x$. Then, if $T_0>0$, it holds that
\begin{equation}\label{eq2.15}
    T_{q+1}\geq T_q-\delta_{q+2}^{\frac{1}{2}}>0.
\end{equation}
\end{thm}

The heart of the proof of Theorem \ref{thm-main-iteration} is to construct suitable perturbations
\begin{align*}
    w_{q+1}\simeq m_{q+1}-m_q,\quad f^i_{q+1}\simeq \mathcal{N}^i_{q+1}-\mathcal{N}^i_q,\quad z_{q+1}\simeq\rho_{q+1}-\rho_q.
\end{align*}
In order to avoid the loss of derivatives in the convex integration scheme, we need to perform the mollification procedure to the relaxed system \eqref{eq2.1}.
\subsection{Mollification procedure}
Let $\phi_{\epsilon}$ and $\varphi_{\epsilon}$ be the standard mollifier on $\mathbb{R}^3$ and $\mathbb{R}$ with ${\rm supp}\phi_{\epsilon}\subseteq B(0,\epsilon)$ and ${\rm supp}\varphi_{\epsilon}\subseteq(-\epsilon,\epsilon)$ respectively. Then, the mollification of $(\rho_q,m_q,\mathcal{N}_q,R^m_q,\mathscr{M}_q)$ in space and time are defined by
\begin{equation}\label{eq2.16}
    \begin{aligned}
        &\rho_{l}:=(\rho_q*_x\phi_l)*_t\varphi_l,\quad P_l:=(P(\rho_q)*_x\phi_l)*_t\varphi_l,\\
        &m_l:=(m_q*_x\phi_l)*_t\varphi_l,\quad \mathcal{N}^i_l:=(\mathcal{N}^i_q*_x\phi_l)*_t\varphi_l,\\
        &R^m_l:=(R^m_q*_x\phi_l)*_t\varphi_l,\quad M^i_l:=(M^i_q*_x\phi_l)*_t\varphi_l,\\
        &\mathcal{N}_l:=\left(\mathcal{N}_l^1,\mathcal{N}_l^2,\mathcal{N}_l^3\right),\quad\left(\mathscr{M}_l\right)_{ijk}:=\left(M^i_l\right)_{jk},
    \end{aligned}
\end{equation}
where the scale of mollification is given by
\begin{align}\label{eq2.17}
    l:=\lambda_q^{-30}.
\end{align}
Then, by equations \eqref{eq2.1}, $(\rho_l,m_l,\mathcal{N}^i_l,R^m_l,M^i_l)$ satisfies
\begin{equation}\label{eq2.18}
    \left\{
    \begin{aligned}
        &\partial_t\rho_l+{\rm div}m_l=0,\\
        &\partial_tm_l+\nu^s(-\Delta)^{\alpha}(\rho^{-1}_lm_l)-(\nu^b+\frac{1}{3}\nu^s)\nabla{\rm div}(\rho^{-1}_lm_l)+\nabla P(\rho_l)\\
        &\qquad+{\rm div}\left(\rho_l^{-1}\left(m_l\otimes m_l-\sum_{i=1}^3\left(\mathcal{N}_l^i\otimes\mathcal{N}_l^i\right)\right)\right)={\rm div}\left(R^m_l+R^m_{{\rm com}}\right),\\
        &\partial_t\mathcal{N}^i_l+{\rm div}\left(\rho_l^{-1}\left(\mathcal{N}^i_l\otimes m_l-m_l\otimes \mathcal{N}^i_l\right)\right)={\rm div}\left(M^i_l+M^i_{{\rm com}}\right),\quad i=1,2,3,\\
        &{\rm div}\mathcal{N}^i_l=0,\quad i=1,2,3,
    \end{aligned}
    \right.
\end{equation}
where the symmetric commutator stress $R^m_{{\rm com}}$ and the  skew-symmetric commutator stress $M^i_{{\rm com}}$ are of form
\begin{align}
    &\begin{aligned}
        R^m_{{\rm com}}:=
        &\nu^s\mathcal{R}^m(-\Delta)^{\alpha}\left(\rho_l^{-1}m_l-((\rho_q^{-1}m_q)*_x\phi_l)*_t\varphi_l\right)+\mathcal{R}^m\nabla\left(P(\rho_l)-P_l\right)\\
        &-(\nu^b+\frac{1}{3}\nu^s)\mathcal{R}^m\nabla{\rm div}\left(\rho_l^{-1}m_l-((\rho_q^{-1}m_q)*_x\phi_l)*_t\varphi_l\right)\\
        &+\mathcal{R}^m{\rm div}\left(\rho_l^{-1}\left(m_l\otimes m_l-\sum_{i=1}^3\left(\mathcal{N}_l^i\otimes\mathcal{N}_l^i\right)\right)\right)\\
        &-\mathcal{R}^m{\rm div}\left(\left(\rho_q^{-1}\left(m_q\otimes m_q-\sum_{i=1}^3\left(\mathcal{N}_q^i\otimes\mathcal{N}_q^i\right)\right)\right)*_x\phi_l\right)*_t\varphi_l,
    \end{aligned}\label{eq2.19}\\
    & M^i_{{\rm com}}:=\mathcal{R}^n{\rm div}\left(\rho_l^{-1}(\mathcal{N}^i_l\otimes m_l-m_l\otimes \mathcal{N}^i_l)-\left(\left(\rho_q^{-1}(\mathcal{N}^i_q\otimes m_q-m_q\otimes \mathcal{N}^i_q)\right)*_x\phi_l\right)*_t\varphi_l\right),\label{eq2.20}
\end{align}
with $\mathcal{R}^m$ and $\mathcal{R}^n$ being the symmetric and skew-symmetric inverse-divergence operator respectively, given by \eqref{eq4.1} and \eqref{eq4.2} below.

By the standard mollification estimates and inductive estimates \eqref{eq2.5}--\eqref{eq2.9}, for any integers $1\leq N\leq4$ and $0\leq M\leq1$,
\begin{align}
    &C_1+\frac{1}{2}\lambda_q^{-\beta}\leq\rho_l\leq C_2-\frac{1}{2}\lambda_q^{-\beta},\label{eq2.21}\\    &\|\partial^M_t\rho_l\|_{C_tC^N_x}\lesssim\lambda_q^{\frac{N\varepsilon}{4}}, \label{eq2.22}\\
    &\|\rho_l-\rho_q\|_{C_tC^{N-1}_x}\lesssim l\lambda_q^{\frac{N\varepsilon}{4}}, \label{eq2.23}\\
    &\|m_l\|_{C^N_{t,x}}+\sum_{i=1}^3\|\mathcal{N}^i_l\|_{C^N_{t,x}}\lesssim l^{-N+1}\lambda_q^7,\label{eq2.24}\\
    &\|m_l-m_q\|_{C_tC^{N-1}_x}+\sum_{i=1}^3\|\mathcal{N}^i_l-\mathcal{N}^i_q\|_{C_tC^{N-1}_x}\lesssim l\lambda_q^{4N+3}, \label{eq2.25}\\
    &\|R^m_l\|_{C^N_{t,x}}+\sum_{i=1}^3\|M^i_l\|_{C^N_{t,x}}\lesssim l^{-N},\label{eq2.26}\\
    &\|R^m_l\|_{L^1_{t,x}}+\sum_{i=1}^3\|M^i_l\|_{L^1_{t,x}}\lesssim\delta_{q+1},\label{eq2.27}
\end{align}
where the implicit constants are independent of $q$. Moreover, concerning the temporal support of $\nabla\rho_l$, $m_l$, $\mathcal{N}_l-\overline{\mathcal{N}}$, $R^m_l$ and $\mathscr{M}_l$, we have
\begin{align}\label{eq2.28}
    {\rm supp}_t(\nabla\rho_l,m_l,\mathcal{N}_l-\overline{\mathcal{N}},R^m_l,\mathscr{M}_l)\subseteq N_l\left({\rm supp}_t(\nabla\rho_q,m_q,\mathcal{N}_q-\overline{\mathcal{N}},R^m_q,\mathscr{M}_q)\right).
\end{align}
\section{Perturbations}\label{sec3}
The aim of this section is to construct appropriate momentum, magnetic and density perturbations, such that the corresponding inductive estimates in Theorem \ref{thm-main-iteration} propagate through in the convex integration scheme.

The fundamental building blocks in the convex integration will be indexed by the following parameters
\begin{align}\label{eq3.1}
    \lambda:=\lambda_{q+1},\quad\tau:=\lambda_{q+1}^{2\alpha},\quad\sigma:=\lambda_{q+1}^{5\varepsilon},\quad r_{\perp}:=\lambda_{q+1}^{2-2\alpha-15\varepsilon},
\end{align}
where $\varepsilon$ is the small constant satisfying \eqref{eq2.2s}.
\subsection{Spatial building blocks}
To begin with, let us recall the geometric lemma in \cite{BBV-2020-annpde} and \cite{LZZ-2022-jmpa} which will be used to construct the basic spatial building blocks, namely, the Mikado flows.

\begin{lem}\label{lem-1st-geometric}(First Geometric Lemma). There exists a finite set $\Lambda_{\mathcal{N}^i}\subseteq\mathbb{S}^2\cap\mathbb{Q}^3$ consisting of vectors $k$ with orthonormal bases $(k,k_1,k_2)\in \mathbb{S}^2\cap\mathbb{Q}^3$ and smooth positive functions $\gamma_{(k)}:B_{{\rm skew},3}(0,\varepsilon_{\mathcal{N}})\rightarrow\mathbb{R}$, where $B_{{\rm skew},3}(0,\varepsilon_{\mathcal{N}})$ is the ball of radius  $\varepsilon_{\mathcal{N}}>0$ centered at $0$ in the space of $3\times3$ skew-symmetric matrices, for any $A\in B_{{\rm skew},3}(0,\varepsilon_{\mathcal{N}})$ we have the following identity:
\begin{align}\label{eq3.2}
    A^i=\sum_{k\in\Lambda_{\mathcal{N}^i}}\gamma^2_{(k)}(A^i)(k_2\otimes k_1-k_1\otimes k_2).
\end{align}
Moreover, we can choose $\Lambda_{\mathcal{N}}=\Lambda_{\mathcal{N}^1}\cup\Lambda_{\mathcal{N}^2}\cup\Lambda_{\mathcal{N}^3}$ with $\Lambda_{\mathcal{N}^i}\cap\Lambda_{\mathcal{N}^j}=\varnothing$ while $i\neq j$.
\end{lem}

In order to overcome the difficulty caused by the relative rigidity of pressure $\nabla P(\rho)$, we need the new geometric lemma below (for a detailed proof, see Appendix \ref{Appendix_C}).
\begin{lem}\label{lem-2nd-geometric}(Second Geometric Lemma). There exists a finite set $\Lambda_{m}\subseteq\mathbb{S}^2\cap\mathbb{Q}^3$ consisting of vectors $k$ with orthonormal bases $(k,k_1,k_2)\in \mathbb{S}^2\cap\mathbb{Q}^3$ and smooth positive functions $\gamma_{(k)}:B_{{\rm sym},3}(0,\varepsilon_m)\rightarrow\mathbb{R}$, where $B_{{\rm sym},3}(0,\varepsilon_m)$ is the ball of radius $\varepsilon_m>0$ centered at $0$ in the space of $3\times3$ symmetric matrices, such that for any $S\in B_{{\rm sym},3}(0,\varepsilon_m)$ we have the following identity:
\begin{align}\label{eq3.3}
    S=\sum_{k\in\Lambda_{m^1}}\gamma^2_{(k)}(S)(k_1\otimes k_1)-\sum_{i=1}^3\sum_{k\in\Lambda_{m^2_i}}\gamma^2_{(k)}(S)(k_2\otimes k_2),
\end{align}
where $\Lambda_m=\Lambda_{m^1}\cup\Lambda_{m^2}$ and $\Lambda_{m^2}=\Lambda_{m^2_1}\cup\Lambda_{m^2_2}\cup\Lambda_{m^2_3}$. Furthermore, we may choose $\Lambda_m$, $\Lambda_{m^1}$ and $\Lambda_{m^2}$ such that $\Lambda_\mathcal{N}\cap\Lambda_m=\varnothing$ and $\Lambda_{m^1}\cap\Lambda_{m^2}=\varnothing$.
\end{lem}

As pointed out in \cite{BBV-2020-annpde}, there exists $N_{\Lambda}\in\mathbb{N}$ such that
\begin{align*}
    \{N_{\Lambda}k,N_{\Lambda}k_1,N_{\Lambda}k_2\}\subseteq N_{\Lambda}\mathbb{S}^2\cap\mathbb{Z}^3,\quad \forall k\in\Lambda_m\cup\Lambda_{\mathcal{N}}.
\end{align*}
Furthermore, we denote by $M_*$ the geometric constant such that
\begin{align*}
    \sum_{k\in\Lambda_m}\|\gamma_{(k)}\|_{C^4(B_{{\rm sym},3}(0,\varepsilon_m))}+\sum_{k\in\Lambda_b}\|\gamma_{(k)}\|_{C^4(B_{{\rm skew},3}(0,\varepsilon_{\mathcal{N}}))}\leq M_*.
\end{align*}
This parameter is universal and will be used later in the estimates of the size of perturbations.

Then, let $\Phi:\mathbb{R}\rightarrow\mathbb{R}$ be a smooth cut-off function supported on $[-1,1]$, and then normalize $\Phi$ such that $\phi:=-\frac{{\rm d}^2}{{\rm d}x^2}\Phi$ satisfies
\begin{align*}
    \frac{1}{2\pi}\int_{\mathbb{R}}\phi^2(x){\rm d}x=1.
\end{align*}
The corresponding rescaled cut-off functions are defined by
\begin{align*}
    \phi_{r_{\perp}}(x):=r_{\perp}^{-\frac{1}{2}}\phi\left(\frac{x}{r_{\perp}}\right),\quad\Phi_{r_{\perp}}(x):=r_{\perp}^{-\frac{1}{2}}\Phi\left(\frac{x}{r_{\perp}}\right).
\end{align*}
By abuse of notation, we periodize $\phi_{r_{\perp}}$ and $\Phi_{r_{\perp}}$ so that they are treated as periodic functions defined on $\mathbb{T}^3$.

Inspired by \cite{BV-2019-emssms,BBV-2020-annpde,LQZZ-2022-arxiv,LZZ-2022-jmpa,LZZ-2024-jfa,ChenLiu2022}, we choose the intermittent shear flows as the basic spatial building blocks, defined by
\begin{align*}
        &W_{(k)}:=\phi_{r_{\perp}}(\lambda r_{\perp} N_{\Lambda}k\cdot x)k_1,\quad k\in\Lambda_{m^1}\cup \Lambda_{\mathcal{N}},\\
        &D^i_{(k)}:=\phi_{r_{\perp}}(\lambda r_{\perp} N_{\Lambda}k\cdot x)k_2,\quad k\in\Lambda_{m^2_i}\cup\Lambda_{\mathcal{N}^i},\quad i=1,2,3,
\end{align*}
where $W_{(k)}$ and $D^i_{(k)}$ are the intermittent momentum and elastic shear flows respectively. For ease of notation, we set
    \begin{align*}
        &\phi_{(k)}(x):=\phi_{r_{\perp}}(\lambda r_{\perp} N_{\Lambda}k\cdot x),\\
        &\Phi_{(k)}(x):=\Phi_{r_{\perp}}(\lambda r_{\perp} N_{\Lambda}k\cdot x),
    \end{align*}
and thus, the intermittent shear flows can be reformulated as 
\begin{equation}\label{3.4s}
    \begin{aligned}
        &W_{(k)}=\phi_{(k)}k_1,\quad k\in\Lambda_{m^1}\cup \Lambda_{\mathcal{N}},\\
        &D^i_{(k)}=\phi_{(k)}k_2,\quad k\in\Lambda_{m^2_i}\cup\Lambda_{\mathcal{N}^i}, \quad i=1,2,3.
    \end{aligned}
\end{equation}

Lemma \cref{lem-mikado-flows} below contains their crucial analytic estimates, which provide $2D$ spatial intermittency.
\begin{lem}\label{lem-mikado-flows}(Estimates of Intermittent shear flows). For any $p\in[1,+\infty)$ and $N\in\mathbb{N}$, we have
\begin{align*}
\|\nabla^N\phi_{(k)}\|_{L^p_x}+\|\nabla^N\Phi_{(k)}\|_{L^p_x}\lesssim r_{\perp}^{\frac{1}{p}-\frac{1}{2}}\lambda^N,
\end{align*}
where the implicit constants are independent of $\lambda$ but may depend on $N$.
\end{lem}
\subsection{Temporal building blocks}
Since intermittent Mikado flows chosen as in \eqref{3.4s} provide no enough intermittent effect, and thus cannot control the dissipation caused by $(-\Delta)^{\alpha}$ in \eqref{eq2.1}. The key idea here is to exploit more intermittency from the temporal oscillations. This technique was used in \cite{CL-2021-annpde,CL-2022-invent,CL-2023-annpde,LQZZ-2024-jmpa,LZZ-2024-jfa} which obtain sharp non-uniqueness results. However, unlike those works, because of spatial interactions of Mikado flows, inspired by \cite{LQZZ-2022-arxiv,DL-2025-JLMS}, suitable shifts should also be taken into account in the temporal building blocks, so that the supports of different temporal building blocks are disjoint.

More precisely, let $\left\{g_k\right\}_{k \in \Lambda_m\cup\Lambda_{\mathcal{N}}}\subseteq C_c^{\infty}([0, T])$ be cut-off functions such that ${\rm supp}(g_k)\cap{\rm supp}(g_k^{\prime})=\emptyset$ if $k \neq k^{\prime}$, and
\begin{align*}
    \fint_0^Tg^2_k(t){\rm d}t=1
\end{align*}
for all $k \in \Lambda_m\cup\Lambda_{\mathcal{N}}$. The existence of such a family $\left\{g_k\right\}_{k \in \Lambda_m\cup\Lambda_{\mathcal{N}}}$ is guaranteed by the fact that there are finitely many wavevectors in $k \in \Lambda_m\cup\Lambda_{\Lambda_{\mathcal{N}}}$, e.g., $g_k=g\left(t-\alpha_k\right)$, where $g \in C_c^{\infty}([0, T])$ with  small support and $\left\{\alpha_k\right\}_{k \in \Lambda_m\cup\Lambda_{\mathcal{N}}}$ are the temporal shifts such that the supports of $\left\{g_k\right\}$ are disjoint.

Then, for each $k \in \Lambda_m\cup\Lambda_{\mathcal{N}}$, we rescale the cut-off function $g_k$ by
\begin{align*}
    g_{k, \tau}(t)=\tau^{\frac{1}{2}} g_k(\tau t),
\end{align*}
where the concentration parameter $\tau$ is given by \eqref{eq3.1}. Then, we periodize $g_{k,\tau}$ such that the resulting functions (by an abuse of notation, still denoted by $g_{k,\tau}$) are periodic functions defined on $[0,T]$.

In order to balance the high temporal oscillations arising from the concentration functions $g_{k,\tau}$, we need the functions $h_{k,\tau}:[0,T]\rightarrow\mathbb{R}$, defined by
\begin{align*}
    h_{k, \tau}(t):=\int_0^t\left(g_{k, \tau}^2(s)-1\right) d s, \quad t \in[0, T].
\end{align*}

Set
\begin{align}\label{3.5s}
    g_{(k)}(t):=g_{k, \tau}(\sigma t), \quad h_{(k)}(t):=h_{k, \tau}(\sigma t) .
\end{align}
Then, we obtain
\begin{align}\label{3.6s}
    \partial_t\left(\sigma^{-1} h_{(k)}\right)=g_{(k)}^2-1=g_{(k)}^2-\fint_0^T g_{(k)}^2(s) \mathrm{d} s,
\end{align}
where $\sigma$ is given by \eqref{eq3.1}.

Now we introduce the crucial temporal intermittent estimates of $g_{(k)}$ and $h_{(k)}$ as in \cite{LQZZ-2022-arxiv} which are summarized in the following lemma.
\begin{lem}\label{lem-temporal-intermittency}(Estimates of temporal intermittency). For any $\gamma\in[1,+\infty]$, $M\in\mathbb{N}$, we have
\begin{align}\label{3.7s}  \|\partial_t^Mg_{(k)}\|_{L^{\gamma}_t}\lesssim\sigma^M\tau^{M+\frac{1}{2}-\frac{1}{\gamma}},
\end{align}
where the implicit constants are independent of $\sigma$ and $\tau$ but may depend on $M$. Moreover, we have the uniform bound
\begin{align}\label{3.8s}
    \|h_{(k)}\|_{L^{\infty}_t}\leq1.
\end{align}
\end{lem}

\subsection{Elastic and momentum perturbations}
This section is devoted to the construction of the key elastic and momentum perturbations, including the
principal parts, the almost incompressibility correctors and the temporal correctors. 

To begin with, let us first define suitable amplitudes of perturbations as similar fashion in \cite{BBV-2020-annpde}, which play the key role in the cancellation between the low-frequency part of the nonlinearity and the old Reynolds stress and elastic stress fluctuation.
\subsubsection{Amplitudes}
Let $\chi:[0,+\infty)\rightarrow\mathbb{R}$ be a smooth spatial cut-off function satisfying
\begin{align}\label{3.9s}
    \chi(z)=\left\{
    \begin{aligned}
     & 1,\quad0\leq z\leq1,\\  
     &z,\quad z\geq2,
    \end{aligned}
    \right.
\end{align}
and
\begin{align}\label{3.10s}
    \frac{1}{2}z\leq\chi(z)\leq2z,\quad 1<z<2.
\end{align}

Define
\begin{align}\label{3.11s}
\varrho_{\mathcal{N}^i}(t,x):=2\varepsilon_{\mathcal{N}}^{-1}\delta_{q+1}\chi\left(\frac{|M^i_l(t,x)|}{\delta_{q+1}}\right),\quad i=1,2,3,
\end{align}
where $\varepsilon_{\mathcal{N}}>0$ is the small constant in the first geometric lemma, Lemma \ref{lem-1st-geometric}. Then, by \eqref{eq2.8}, \eqref{eq2.9}, \eqref{eq2.26}, \eqref{3.9s}--\eqref{3.11s} and standard H\"older estimates, for $1\leq N\leq4$ and $i=1,2,3$, we have
\begin{align}
    &\left|\frac{M^i_l(t,x)}{\varrho_{\mathcal{N}^i}(t,x)}\right|\leq\varepsilon_,\quad\varrho_{\mathcal{N}^i}(t)\geq\varepsilon_{\mathcal{N}}^{-1}\delta_{q+1},\quad\forall t\in[0,T],x\in\mathbb{T}^3,\label{3.12s}\\   &\|\varrho_{\mathcal{N}^i}\|_{L^p_{t,x}}\leq8\varepsilon_{\mathcal{N}}^{-1}\left((16T\pi)^{\frac{1}{p}}\delta_{q+1}+\|M^i_l\|_{L^p_{t,x}}\right),\quad\forall p\in[1,+\infty),\label{3.13s}\\
    &\|\varrho_{\mathcal{N}^i}\|_{C_{t,x}}\lesssim l^{-1},\quad\|\varrho_{\mathcal{N}^i}\|_{C^N_{t,x}}\lesssim l^{-N}\delta_{q+1}^{-N+1},\label{3.14s}\\
    &\|\varrho^{\frac{1}{2}}_{\mathcal{N}^i}\|_{C_{t,x}}\lesssim l^{-1},\quad\|\varrho^{\frac{1}{2}}_{\mathcal{N}^i}\|_{C^N_{t,x}}\lesssim l^{-N}\delta_{q+1}^{-2N},\label{3.15s}\\
    &\|\varrho^{-1}_{\mathcal{N}^i}\|_{C_{t,x}}\lesssim\delta^{-1}_{q+1},\quad\|\varrho_{\mathcal{N}^i}^{-1}\|_{C^N_{t,x}}\lesssim l^{-N}\delta_{q+1}^{-2N},\label{3.16s}
\end{align}
where the implicit constants are independent of $q$.

Then, for all $i=1,2,3$, the smooth temporal cut-off function $f_{\mathcal{N}^i}(t)$ is defined by
\begin{itemize}
    \item $0\leq f_{\mathcal{N}^i}\leq1\text{ and } f_{\mathcal{N}^i}\equiv1\text{ on }{{\rm supp}_t(M^i_l)}$;
     \item ${\rm supp}_t(f_{\mathcal{N}^i})\subseteq N_l({\rm supp}_t(M^i_l))$;
     \item $\|f_{\mathcal{N}^i}\|_{C^N_t}\lesssim l^{-N},1\leq N\leq5$.
\end{itemize}

The amplitudes of the elastic perturbations are defined by
\begin{align}\label{3.17s}
    a_{(k)}(t,x):=f_{\mathcal{N}^i}(t)\varrho^{\frac{1}{2}}_{\mathcal{N}^i}(t,x)\rho^{\frac{1}{2}}_l(t,x)\gamma_{(k)}\left(\frac{-M^i_l(t,x)}{\varrho_{\mathcal{N}^i}(t,x)}\right),\quad k\in\Lambda_{\mathcal{N}^i},
\end{align}
where $\gamma_{(k)}$ and $\Lambda_{\mathcal{N}^i}$ are given in Lemma \ref{lem-1st-geometric}, and $\rho_l$ is the mollified density given by \eqref{eq2.16}.

Note that, by virtue of Lemma \ref{lem-1st-geometric}, the identity \eqref{eq3.2} and the expression \eqref{3.17s} of $a_{(k)}$, the following identity holds for all $i=1,2,3$:
\begin{equation}\label{3.18s}
    \begin{aligned} &\quad\sum_{k\in\Lambda_{\mathcal{N}^i}}\rho_l^{-1}a^2_{(k)}g^2_{(k)}\left(D^i_{(k)}\otimes W_{(k)}-W_{(k)}\otimes D^i_{(k)}\right)\\
  &=-M^i_l+\sum_{k\in\Lambda_{\mathcal{N}^i}}\rho_l^{-1}a^2_{(k)}g^2_{(k)}\mathbb{P}_{\neq0}\left(D^i_{(k)}\otimes W_{(k)}-W_{(k)}\otimes D^i_{(k)}\right)\\
  &\quad+\sum_{k\in\Lambda_{\mathcal{N}^i}}\rho_l^{-1}a^2_{(k)}\left(g^2_{(k)}-1\right)\fint_{\mathbb{T}^3}\left(D^i_{(k)}\otimes W_{(k)}-W_{(k)}\otimes D^i_{(k)}\right){\rm d}x,
    \end{aligned}
\end{equation}
where $\mathbb{P}_{\neq0}$ denotes the spatial projection onto nonzero Fourier modes.

Moreover, the analytic estimates of elastic amplitudes are included in the following lemma, whose proof is similar to \cite[Lemma 4.1]{LZZ-2022-jmpa}, and so the proof is omitted here.
\begin{lem}\label{lem-magnetic-amplitude-estimate}(Estimates of magnetic amplitude). For any $1\leq N\leq4$, $k\in\Lambda_{\mathcal{N}^i}$, we have
\begin{align}
    &\|a_{(k)}\|_{L^2_{t,x}}\lesssim\delta_{q+1}^{\frac{1}{2}},\label{3.19s}\\
    &\|a_{(k)}\|_{C_{t,x}}\lesssim l^{-2},\quad\|a_{(k)}\|_{C^N_{t,x}}\lesssim l^{-4N-1},\label{3.20s}
\end{align}
where the implicit constants are independent of $q$.
\end{lem}

Next we will focus on the momentum amplitudes. Unlike the previous elastic case, because of the strong coupling between the momentum and elastic stress, that is, sparked by \cite{BBV-2020-annpde,LZZ-2022-jmpa,LZZ-2024-jfa,ChenLiu2022},  a matrix $G^{\mathcal{N}}$ is indispensable in order to maintain the cancellation between the momentum perturbations and old Reynolds stress, where
\begin{align}\label{3.21s}  G^{\mathcal{N}}:=\sum_{i=1}^3\sum_{k\in\Lambda_{\mathcal{N}^i}}\rho_l^{-1}a^2_{(k)}\left(\fint_{\mathbb{T}^3}W_{(k)}\otimes W_{(k)}{\rm d}x-\fint_{\mathbb{T}^3}D^i_{(k)}\otimes D^i_{(k)}{\rm d}x\right).
\end{align}
In view of estimates \eqref{eq2.21}, \eqref{eq2.22}, \eqref{3.19s} and \eqref{3.20s}, we have that for $N\geq1$,
\begin{align}\label{3.22s}
    \|G^{\mathcal{N}}\|_{C_{t,x}}\lesssim l^{-2},\quad\|G^{\mathcal{N}}\|_{C^N_{t,x}}\lesssim l^{-4N-2},\quad\|G^{\mathcal{N}}\|_{L^1_{t,x}}\lesssim\delta_{q+1}.
\end{align}

Set
\begin{align}\label{3.23s} \varrho_m(t,x):=2\varepsilon_m^{-1}\delta_{q+1}\chi\left(\frac{|R^m_l(t,x)+G^{\mathcal{N}}(t,x)|}{\delta_{q+1}}\right),
\end{align}
where $\varepsilon_m>0$ is the small constant in the second geometric lemma, Lemma \ref{lem-2nd-geometric}. Then, by \eqref{eq2.8}, \eqref{eq2.9}, \eqref{eq2.26}, \eqref{3.9s}--\eqref{3.11s}, \eqref{3.22s} and standard H\"older estimates, for $1\leq N\leq4$,
\begin{align}
    &\left|\frac{R^m_l(t,x)+G^{\mathcal{N}}(t,x)}{\varrho_m(t,x)}\right|\leq\varepsilon_m,\quad\varrho_m(t,x)\geq\varepsilon_m^{-1}\delta_{q+1},\quad\forall t\in[0,T],x\in\mathbb{T}^3,\label{3.24s}\\   &\|\varrho_m\|_{L^p_{t,x}}\leq8\varepsilon_m^{-1}\left((16T\pi)^{\frac{1}{p}}\delta_{q+1}+\|R^m_l+G^{\mathcal{N}}\|_{L^p_{t,x}}\right),\quad\forall p\in[1,+\infty),\label{3.25s}\\
    &\|\varrho_m\|_{C_{t,x}}\lesssim l^{-2},\quad\|\varrho_m\|_{C^N_{t,x}}\lesssim l^{-5N-2}\delta_{q+1}^{-2N},\label{3.26s}\\
    &\|\varrho^{\frac{1}{2}}_m\|_{C_{t,x}}\lesssim l^{-1},\quad\|\varrho^{\frac{1}{2}}_m\|_{C^N_{t,x}}\lesssim l^{-5N-2}\delta_{q+1}^{-2N},\label{3.27s}\\
    &\|\varrho^{-1}_m\|_{C_{t,x}}\lesssim \delta^{-1}_{q+1},\quad\|\varrho_m^{-1}\|_{C^N_{t,x}}\lesssim l^{-5N-2}\delta_{q+1}^{-2N},\label{3.28s}
\end{align}
where the implicit constants are independent of $q$.

We choose the smooth temporal cut-off function $f_m(t)$ such that
\begin{itemize}
    \item $0\leq f_m\leq1\text{ and } f_m\equiv1\text{ on }{{\rm supp}_t(R^m_l)\cup{\rm supp}_t(G^\mathcal{N})}$;
     \item ${\rm supp}_t(f_m)\subseteq N_l({\rm supp}_t(R^m_l)\cup{\rm supp}_t(G^{\mathcal{N}}))\subseteq N_{2l}({\rm supp}_t(R^m_l)\cup(\bigcup_{i=1}^3{\rm supp}_t(M^i_l)))$;
     \item $\|f_m\|_{C^N_t}\lesssim l^{-N},1\leq N\leq5$.
\end{itemize}

The amplitudes of the momentum perturbations are defined by
\begin{align}\label{3.29s}
    a_{(k)}(t,x):=f_m(t)\varrho^{\frac{1}{2}}_m(t,x)\rho^{\frac{1}{2}}_l(t,x)\gamma_{(k)}\left(-\frac{R^m_l(t,x)+G^{\mathcal{N}}(t,x)}{\varrho_m(t,x)}\right),\quad k\in\Lambda_m,
\end{align}
where $\gamma_{(k)}$ and $\Lambda_m$ are given in Lemma \ref{lem-2nd-geometric}, and $\rho_l$ is the mollified density given by \eqref{eq2.16}.

Similar to \eqref{3.18s}, we can obtain the following identity by Lemma \ref{lem-2nd-geometric} and \eqref{3.29s},
\begin{equation}\label{3.30s}
    \begin{aligned}
        &\sum_{k\in\Lambda_{m^1}}\rho_l^{-1}a^2_{(k)}g^2_{(k)}\left(W_{(k)}\otimes W_{(k)}\right)-\sum_{i=1}^3\sum_{k\in\Lambda_{m^2_i}}\rho_l^{-1}a^2_{(k)}g^2_{(k)}\left(D^i_{(k)}\otimes D^i_{(k)}\right)\\
        &=-R^m_l-G^{\mathcal{N}}+\sum_{k\in\Lambda_{m^1}}\rho_l^{-1}a^2_{(k)}g^2_{(k)}\mathbb{P}_{\neq0}\left(W_{(k)}\otimes W_{(k)}\right)\\
        &\quad-\sum_{i=1}^3\sum_{k\in\Lambda_{m^2_i}}\rho_l^{-1}a^2_{(k)}g^2_{(k)}\mathbb{P}_{\neq0}\left(D^i_{(k)}\otimes D^i_{(k)}\right)\\
        &\quad+\sum_{k\in\Lambda_{m^1}}\rho_l^{-1}a^2_{(k)}(g^2_{(k)}-1)\fint_{\mathbb{T}^3}\left(W_{(k)}\otimes W_{(k)}\right){\rm d}x\\
        &\quad-\sum_{i=1}^3\sum_{k\in\Lambda_{m^2_i}}\rho_l^{-1}a^2_{(k)}(g^2_{(k)}-1)\fint_{\mathbb{T}^3}\left(D^i_{(k)}\otimes D^i_{(k)}\right){\rm d}x.
    \end{aligned}
\end{equation}

\begin{rem}
    Note the fact that $\Lambda_m=\Lambda_{m^1}\cup\Lambda_{m^2}$ and $\Lambda_{m^1}\cap\Lambda_{m^2}=\varnothing$, which leads that all the error caused by $W_{(k)}\otimes D^i_{(k^{\prime})}-D^i_{(k^{\prime})}\otimes W_{(k)}$ with $k\in\Lambda_{m^1}$ and $k^{\prime}\in\Lambda_{m^2}$ will vanish on the elastic equations.
\end{rem}

Next, the analytic estimates of momentum amplitudes are collected in Lemma \ref{lem-momentum-amplitude-estimate}, whose proof is similar to \cite[Lemma 4.2]{LZZ-2022-jmpa}.
\begin{lem}\label{lem-momentum-amplitude-estimate}(Estimates of momentum amplitudes). 
    For any $1\leq N\leq4$, $k\in\Lambda_m$, we have
\begin{align}
    &\|a_{(k)}\|_{L^2_{t,x}}\lesssim\delta_{q+1}^{\frac{1}{2}},\label{3.31s}\\
    &\|a_{(k)}\|_{C_{t,x}}\lesssim l^{-2},\quad\|a_{(k)}\|_{C^N_{t,x}}\lesssim l^{-9N-2},\label{3.32s}
\end{align}
where the implicit constants are independent of $q$.
\end{lem}
\subsubsection{Principal parts of perturbations}
The principal parts of the momentum and elastic perturbations are defined by
\begin{align}    &w_{q+1}^{(p)}:=\sum_{k\in\Lambda_{m^1}\cup\Lambda_{\mathcal{N}}}a_{(k)}g_{(k)}W_{(k)},\label{3.33s}\\    &f_{q+1}^{i,(p)}:=\sum_{k\in\Lambda_{m^2_i}\cup\Lambda_{\mathcal{N}^i}}a_{(k)}g_{(k)}D^i_{(k)},\quad i=1,2,3.\label{3.34s}
\end{align}
Since the temporal supports of $g_{(k)}$ are pairwise disjoint, all cross-interaction terms vanish. By \eqref{3.6s} and the algebraic identity \eqref{3.18s}, for all $i=1,2,3$,
\begin{equation}\label{3.35s}
    \begin{aligned}
        &\quad\rho_l^{-1}\left(f_{q+1}^{i,(p)}\otimes w_{q+1}^{(p)}-w_{q+1}^{(p)}\otimes f_{q+1}^{i,(p)}\right)+M^i_l\\       &=\sum_{k\in\cup\Lambda_{\mathcal{N}^i}}\rho_l^{-1}a^2_{(k)}g^2_{(k)}\left(D^i_{(k)}\otimes W_{(k)}-W_{(k)}\otimes D^i_{(k)}\right)+M^i_l\\
        &=\sum_{k\in\Lambda_{\mathcal{N}^i}}\rho_l^{-1}a^2_{(k)}g^2_{(k)}\mathbb{P}_{\neq0}\left(D^i_{(k)}\otimes W_{(k)}-W_{(k)}\otimes D^i_{(k)}\right)\\
  &\quad+\sum_{k\in\Lambda_{\mathcal{N}^i}}\rho_l^{-1}a^2_{(k)}\partial_t\left(\sigma^{-1}h_{(k)}\right)\fint_{\mathbb{T}^3}\left(D^i_{(k)}\otimes W_{(k)}-W_{(k)}\otimes D^i_{(k)}\right){\rm d}x
    \end{aligned}
\end{equation}
Moreover, for the nonlinearity in the momentum equations $\eqref{eq2.1}_2$, by \eqref{3.6s}, \eqref{3.21s} and \eqref{3.30s},
\begin{align}
&\quad\rho_l^{-1}\left(w_{q+1}^{(p)}\otimes w_{q+1}^{(p)}-\sum_{i=1}^3f_{q+1}^{i,(p)}\otimes f_{q+1}^{i,(p)}\right)+R^m_l\notag\\
        &=\sum_{k\in\Lambda_{m^1}}\rho_l^{-1}a^2_{(k)}g^2_{(k)}W_{(k)}\otimes W_{(k)}-\sum_{i=1}^3\sum_{k\in\Lambda_{m^2_i}}\rho_l^{-1}a^2_{(k)}g^2_{(k)}D^i_{(k)}\otimes D^i_{(k)}+R^m_l\notag\\
   &\quad+\sum_{i=1}^3\sum_{k\in\Lambda_{\mathcal{N}^i}}\rho_l^{-1}a^2_{(k)}g^2_{(k)}\left(W_{(k)}\otimes W_{(k)}-D^i_{(k)}\otimes D^i_{(k)}\right)\notag\\
   &=\sum_{k\in\Lambda_{m^1}}\rho_l^{-1}a^2_{(k)}g^2_{(k)}\mathbb{P}_{\neq0}\left(W_{(k)}\otimes W_{(k)}\right)-\sum_{i=1}^3\sum_{k\in\Lambda_{m^2_i}}\rho_l^{-1}a^2_{(k)}g^2_{(k)}\mathbb{P}_{\neq0}\left(D^i_{(k)}\otimes D^i_{(k)}\right)\label{3.36s}\\
        &\quad+\sum_{k\in\Lambda_{m^1}}\rho_l^{-1}a^2_{(k)}\partial_t\left(\sigma^{-1}h_{(k)}\right)\fint_{\mathbb{T}^3}W_{(k)}\otimes W_{(k)}{\rm d}x\notag\\
        &\quad-\sum_{i=1}^3\sum_{k\in\Lambda_{m^2_i}}\rho_l^{-1}a^2_{(k)}\partial_t\left(\sigma^{-1}h_{(k)}\right)\fint_{\mathbb{T}^3}D^i_{(k)}\otimes D^i_{(k)}{\rm d}x\notag\\
        &\quad+\sum_{i=1}^3\sum_{k\in\Lambda_{\mathcal{N}^i}}\rho_l^{-1}a^2_{(k)}g^2_{(k)}\mathbb{P}_{\neq0}\left(W_{(k)}\otimes W_{(k)}-D^i_{(k)}\otimes D^i_{(k)}\right)\notag\\
        &\quad+\sum_{i=1}^3\sum_{k\in\Lambda_{\mathcal{N}^i}}\rho_l^{-1}a^2_{(k)}\partial_t\left(\sigma^{-1}h_{(k)}\right)\fint_{\mathbb{T}^3}\left(W_{(k)}\otimes W_{(k)}-D^i_{(k)}\otimes D^i_{(k)}\right){\rm d}x.\notag
\end{align}

\subsubsection{Incompressible correctors}
Since the amplitude functions $\left\{a_{(k)}\right\}_{k\in\Lambda_m\cup\Lambda_{\mathcal{N}}}$ depend on spatial variables, the principal parts of the perturbations are not divergence-free in general. This leads to the definition of the incompressible correctors
\begin{align}
    &w_{q+1}^{(c)}:=\sum_{k\in\Lambda_{m^1}\cup\Lambda_{\mathcal{N}}}g_{(k)}\left({\rm curl} \left(\nabla a_{(k)}\times W^c_{(k)}\right)+\nabla a_{(k)}\times{\rm curl}W^c_{(k)}\right),\label{3.37s}\\
    &f_{q+1}^{i,(c)}:=\sum_{k\in\Lambda_{m^2_i}\cup\Lambda_{\mathcal{N}^i}}g_{(k)}\left({\rm curl} \left(\nabla a_{(k)}\times D^{i,c}_{(k)}\right)+\nabla a_{(k)}\times{\rm curl}D^{i,c}_{(k)}\right),\quad i=1,2,3,\label{3.38s}
\end{align}
where $W^c_{(k)}$ and $D^{i,c}_{(k)}$ are defined by
\begin{align}\label{3.39s}
    W^c_{(k)}:=\frac{1}{N_{\Lambda}^2\lambda^2}\Phi_{(k)}k_1,\quad D^{i,c}_{(k)}:=\frac{1}{N_{\Lambda}^2\lambda^2}\Phi_{(k)}k_2,\quad i=1,2,3.
\end{align}
Then, it follows that,
\begin{align}
    &w_{q+1}^{(p)}+w_{q+1}^{(c)}={\rm curl}{\rm curl}\left(\sum_{k\in\Lambda_{m^1}\cup\Lambda_{\mathcal{N}}}a_{(k)}g_{(k)}W^c_{(k)}\right),\label{3.40s}\\
    &f_{q+1}^{i,(p)}+f_{q+1}^{i,(c)}={\rm curl}{\rm curl}\left(\sum_{k\in\Lambda_{m^2_i}\cup\Lambda_{\mathcal{N}^i}}a_{(k)}g_{(k)}D^{i,c}_{(k)}\right),\quad i=1,2,3.\label{3.41s}
\end{align}
In particular, for all $i=1,2,3$,
\begin{align}\label{3.42s}
    {\rm div}\left(w_{q+1}^{(p)}+w_{q+1}^{(c)}\right)={\rm div}\left(f_{q+1}^{i,(p)}+f_{q+1}^{i,(c)}\right)=0.
\end{align}
\begin{rem}
    We note that this type of incompressible corrector is frequently employed in convex integration schemes for incompressible fluid, see for instance \cite{BV-2019-emssms,BBV-2020-annpde}. But for compressible Oldroyd-type model (MHD equations), although the elastic (magnetic) field requires the incompressible correctors, it seems not necessary to keep the momentum perturbations incompressible. However, inspired by \cite{LQZZ-2022-arxiv}, we still use this corrector to keep part of momentum perturbations incompressible, which reduces the error caused by density in subsequent convex integration scheme.
\end{rem}
\subsubsection{Temporal correctors}
The last type of corrector consists of the temporal correctors, in order to balance the high temporal oscillations in \eqref{3.35s} and \eqref{3.36s} caused by the temporal concentration functions $g_{(k)}$. Drawing inspiration from \cite{LZZ-2022-jmpa,LQZZ-2022-arxiv,NY-2025-jns}, the temporal correctors $w_{(q+1)}^{(o)}$ and $f_{(q+1)}^{i,(o)}$ are defined by
\begin{align}
   & \begin{aligned}
        w_{(q+1)}^{(o)}:=&-\sigma^{-1}\sum_{k\in\Lambda_{m^1}}h_{(k)}\fint_{\mathbb{T}^3}\left(W_{(k)}\otimes W_{(k)}\right){\rm d}x\nabla \left(\rho_l^{-1}a^2_{(k)}\right)\\
    &+\sigma^{-1}\sum_{i=1}^3\sum_{k\in\Lambda_{m^2_i}}h_{(k)}\fint_{\mathbb{T}^3}\left(D^i_{(k)}\otimes D^i_{(k)}\right){\rm d}x\nabla \left(\rho_l^{-1}a^2_{(k)}\right)\\       
        &-\sigma^{-1}\sum_{i=1}^3\sum_{k\in\Lambda_{\mathcal{N}^i}}h_{(k)}\fint_{\mathbb{T}^3}\left(W_{(k)}\otimes W_{(k)}-D^i_{(k)}\otimes D^i_{(k)}\right){\rm d}x\nabla \left(\rho_l^{-1}a^2_{(k)}\right)
    \end{aligned}\label{3.43s}\\
    &f_{(q+1)}^{i,(o)}:=-\sigma^{-1}\sum_{k\in\Lambda_{\mathcal{N}^i}}h_{(k)}{\rm div}\left(\rho^{-1}_la^2_{(k)}\fint_{\mathbb{T}^3}D^i_{(k)}\otimes W_{(k)}-W_{(k)}\otimes D^i_{(k)}{\rm d}x\right),\quad i=1,2,3,\label{3.44s}
\end{align}
where $h_{(k)}$ is given by \eqref{3.5s}.
\begin{rem}
    Since for all $i=1,2,3$, $\mathbb{P}_{=0}\left(D^i_{(k)}\otimes W_{(k)}-W_{(k)}\otimes D^i_{(k)}\right)=k_2\otimes k_1-k_1\otimes k_2$ are skew-symmetric, $f_{(q+1)}^{i,(o)}$ is divergence-free as well, where $\mathbb{P}_{=0}:={\rm Id}-\mathbb{P}_{\neq0}$.
\end{rem}

Then, by the Leibniz rule,
\begin{equation}\label{3.45s}
    \begin{aligned}     &\quad\partial_tw_{q+1}^{(o)}+\sum_{k\in\Lambda_{m^1}}\rho_l^{-1}a^2_{(k)}\partial_t\left(\sigma^{-1}h_{(k)}\right)\fint_{\mathbb{T}^3}W_{(k)}\otimes W_{(k)}{\rm d}x\nabla\left(\rho_l^{-1}a^2_{(k)}\right)\\
        &\quad-\sum_{i=1}^3\sum_{k\in\Lambda_{m^2_i}}\partial_t\left(\sigma^{-1}h_{(k)}\right)\fint_{\mathbb{T}^3}D^i_{(k)}\otimes D^i_{(k)}{\rm d}x\nabla\left(\rho_l^{-1}a^2_{(k)}\right)\\             &\quad+\sum_{i=1}^3\sum_{k\in\Lambda_{\mathcal{N}^i}}\partial_t\left(\sigma^{-1}h_{(k)}\right)\fint_{\mathbb{T}^3}\left(W_{(k)}\otimes W_{(k)}-D^i_{(k)}\otimes D^i_{(k)}\right){\rm d}x\nabla\left(\rho_l^{-1}a^2_{(k)}\right)\\
        &=-\sigma^{-1}\sum_{k\in\Lambda_{m^1}}h_{(k)}(k_1\cdot\nabla)\partial_t\left(\rho_l^{-1}a^2_{(k)}\right)k_1+\sigma^{-1}\sum_{k\in\Lambda_{m^2}}h_{(k)}(k_2\cdot\nabla)\partial_t\left(\rho_l^{-1}a^2_{(k)}\right)k_2\\
    &\quad-\sigma^{-1}\sum_{k\in\Lambda_{\mathcal{N}}}h_{(k)}\left(\left(k_1\cdot\nabla\right)\partial_t\left(\rho_l^{-1}a^2_{(k)}\right)k_1-\left(k_2\cdot\nabla\right)\partial_t\left(\rho_l^{-1}a^2_{(k)}\right)k_2\right),
    \end{aligned}
\end{equation}
and
\begin{equation}\label{3.46s}
    \begin{aligned}
&\quad\partial_tf_{q+1}^{i,(o)}+\sum_{k\in\Lambda_{\mathcal{N}^i}}\partial_t\left(\sigma^{-1}h_{(k)}\right)\fint_{\mathbb{T}^3}\left(D^i_{(k)}\otimes W_{(k)}-W_{(k)}\otimes D^i_{(k)}\right){\rm d}x\nabla\left(\rho_l^{-1}a^2_{(k)}\right)\\
&=-\sigma^{-1}\sum_{k\in\Lambda_{\mathcal{N}^i}}h_{(k)}\left(k_2\otimes k_1-k_1\otimes k_2\right)\partial_t\nabla\left(\rho_l^{-1}a^2_{(k)}\right),\quad i=1,2,3.
    \end{aligned}
\end{equation}

Now, we define the elastic and momentum perturbations at level $q+1$ by
\begin{align}
    &w_{q+1}:=w_{q+1}^{(p)}+w_{q+1}^{(c)}+w_{q+1}^{(o)},\label{3.47s}\\
    &f_{q+1}^i:=f_{q+1}^{i,(p)}+f_{q+1}^{i,(c)}+f_{q+1}^{i,(o)},\quad i=1,2,3.\label{3.48s}
\end{align}
Note that, by the constructions above, $w_{q+1}$ is mean-free but may not be divergence-free, and, $f^i_{q+1}$ are both mean-free and divergence-free.

The new momentum and elastic fields $m_{q+1}$ and $\mathcal{N}^i_{q+1}$ at $q+1$ level is then defined by
\begin{align}\label{3.49s}
    m_{q+1}:=m_l+w_{q+1},\quad \mathcal{N}^i_{q+1}:=\mathcal{N}^i_l+f^i_{q+1},\quad i=1,2,3,
\end{align}
where $m_l$ and $\mathcal{N}^i_l$ are mollified momentum and elastic fields defined in \eqref{eq2.16}.
\subsection{Density perturbation}
Because the momentum perturbation $w_{q+1}$ may not be divergence-free, we need to construct density perturbation under the constraints of the transport equation $\eqref{eq2.1}_1$ at level $q+1$.

More precisely, we define the density perturbation $z_{q+1}$ as
\begin{equation}\label{3.50s}
    \begin{aligned}
        z_{q+1}(t,x)&:=-\int_0^t{\rm div}w_{q+1}(s,x){\rm d}s\\
        &=\sigma^{-1}\sum_{k\in\Lambda_{m^1}}\int_0^th_{(k)}{\rm div}\left(k_1\otimes k_1\nabla \left(\rho_l^{-1}a^2_{(k)}\right)\right)(s,x){\rm d}s\\
    &\quad-\sigma^{-1}\sum_{k\in\Lambda_{m^2}}\int_0^th_{(k)}{\rm div}\left(k_2\otimes k_2\nabla \left(\rho_l^{-1}a^2_{(k)}\right)\right)(s,x){\rm d}s\\            &\quad+\sigma^{-1}\sum_{k\in\Lambda_{\mathcal{N}}}\int_0^th_{(k)}{\rm div}\left(\left(k_1\otimes k_1-k_2\otimes k_2\right)\nabla \left(\rho_l^{-1}a^2_{(k)}\right)\right)(s,x){\rm d}s
    \end{aligned}
\end{equation}
Then, the new density at level $q+1$ is defined by
\begin{align}\label{3.51s}
    \rho_{q+1}:=\rho_l+z_{q+1}.
\end{align}
Furthermore, we have 
\begin{align*}
    \partial_t\rho_{q+1}+{\rm div}m_{q+1}=\partial_tz_{q+1}+{\rm div}w_{q+1}=0,
\end{align*}
which coincides with the transport equation $\eqref{eq2.1}_1$ at level $q+1$.
\subsection{Estimates of perturbations}
In this section, we summarize the crucial estimates of the perturbations in Proposition \ref{prop-estimate-perturbation} below.
\begin{pro}\label{prop-estimate-perturbation}
    For any $\gamma\in[1,+\infty]$, $p\in(1,+\infty)$ and integer $0\leq N\leq4$, the following estimates hold,
    \begin{align}
        &\|\nabla^Nw_{q+1}^{(p)}\|_{L^{\gamma}_tL^p_x}+\sum_{i=1}^3\|\nabla^Nf_{q+1}^{i,(p)}\|_{L^{\gamma}_tL^p_x}\lesssim l^{-2}\lambda^Nr_{\perp}^{\frac{1}{p}-\frac{1}{2}}\tau^{\frac{1}{2}-\frac{1}{\gamma}},\label{3.52s}\\
        &\|\nabla^Nw_{q+1}^{(c)}\|_{L^{\gamma}_tL^p_x}+\sum_{i=1}^3\|\nabla^Nf_{q+1}^{i,(c)}\|_{L^{\gamma}_tL^p_x}\lesssim l^{-20}\lambda^{N-1}r_{\perp}^{\frac{1}{p}-\frac{1}{2}}\tau^{\frac{1}{2}-\frac{1}{\gamma}},\label{3.53s}\\
        &\|\nabla^Nw_{q+1}^{(o)}\|_{L^{\gamma}_tL^p_x}+\sum_{i=1}^3\|\nabla^Nf_{q+1}^{i,(o)}\|_{L^{\gamma}_tL^p_x}\lesssim l^{-9N-14}\sigma^{-1},\label{3.54s}\\
        &\|w_{q+1}\|_{C^N_{t,x}}+\sum_{i=1}^3\|f^i_{q+1}\|_{C^N_{t,x}}\lesssim\lambda^{4N+3}.\label{3.55s}
    \end{align}
    Moreover, for integer $0\leq N\leq4$ and $0\leq M\leq1$, it holds that,
    \begin{align}
        \|\partial_t^Mz_{q+1}\|_{C_tC^N_x}\lesssim l^{-9N-23}\sigma^{-1},\label{3.56s}
    \end{align}
    where the implicit constants are independent of $q$.
\end{pro}
\begin{pf}
    To begin with, using \eqref{3.6s}, \eqref{3.7s}, \eqref{3.33s}, \eqref{3.34s} and  Lemmas \ref{lem-mikado-flows}, \ref{lem-magnetic-amplitude-estimate} and \ref{lem-momentum-amplitude-estimate}, it leads for any $p\in[1,+\infty]$,
    \begin{align*}
&\quad\|\nabla^Nw_{q+1}^{(p)}\|_{L^{\gamma}_tL^p_x}+\sum_{i=1}^3\|\nabla^Nf_{q+1}^{i,(p)}\|_{L^{\gamma}_tL^p_x}\\
&\lesssim\sum_{k\in\Lambda_{m^1}\cup\Lambda_{\mathcal{N}}}\sum_{N_1+N_2=N}\|a_{(k)}\|_{C^{N_1}_{t,x}}\|g_{(k)}\|_{L^{\gamma}_t}\|\nabla^{N_2}W_{(k)}\|_{L^p_x}\\
&\quad+\sum_{i=1}^3\sum_{k\in\Lambda_{m^2_i}\cup\Lambda_{\mathcal{N}^i}}\sum_{N_1+N_2=N}\|a_{(k)}\|_{C^{N_1}_{t,x}}\|g_{(k)}\|_{L^{\gamma}_t}\|\nabla^{N_2}D^i_{(k)}\|_{L^p_x}\\
&\lesssim\sum_{N_1+N_2=N}l^{-9N_1-2}\tau^{\frac{1}{2}-\frac{1}{\gamma}}r_{\perp}^{\frac{1}{p}-\frac{1}{2}}\lambda^{N_2}\\
&\lesssim l^{-2}\lambda^Nr_{\perp}^{\frac{1}{p}-\frac{1}{2}}\tau^{\frac{1}{2}-\frac{1}{\gamma}},
    \end{align*}
where the last inequality is due to $l^{-9N_1}\ll \lambda^{N_1}$, that is, \eqref{3.52s} is valid.

Next, by \eqref{3.7s}, \eqref{3.37s}, \eqref{3.38s} and \eqref{3.39s} and Lemmas \ref{lem-mikado-flows}, \ref{lem-magnetic-amplitude-estimate} and \ref{lem-momentum-amplitude-estimate},
  \begin{align*}
&\quad\|\nabla^Nw_{q+1}^{(c)}\|_{L^{\gamma}_tL^p_x}+\sum_{i=1}^3\|\nabla^Nf_{q+1}^{i,(c)}\|_{L^{\gamma}_tL^p_x}\\
&\lesssim\sum_{k\in\Lambda_{m^1}\cup\Lambda_{\mathcal{N}}}\|g_{(k)}\|_{L^{\gamma}_t}\sum_{N_1+N_2=N}\left(\|a_{(k)}\|_{C^{N_1+2}_{t,x}}\|\nabla^{N_2}W^c_{(k)}\|_{L^p_x}+\|a_{(k)}\|_{C^{N_1+1}_{t,x}}\|\nabla^{N_2+1}W^c_{(k)}\|_{L^p_x}\right)\\
&\quad+\sum_{i=1}^3\sum_{k\in\Lambda_{m^2}\cup\Lambda_{\mathcal{N}^i}}\|g_{(k)}\|_{L^{\gamma}_t}\sum_{N_1+N_2=N}\left(\|a_{(k)}\|_{C^{N_1+2}_{t,x}}\|\nabla^{N_2}D^{i,c}_{(k)}\|_{L^p_x}+\|a_{(k)}\|_{C^{N_1+1}_{t,x}}\|\nabla^{N_2+1}D^{i,c}_{(k)}\|_{L^p_x}\right)\\
&\lesssim\tau^{\frac{1}{2}-\frac{1}{\gamma}}\sum_{N_1+N_2=N}\left(l^{-9N_1-20}r_{\perp}^{\frac{1}{p}-\frac{1}{2}}\lambda^{N_2-2}+l^{-9N_1-11}r_{\perp}^{\frac{1}{p}-\frac{1}{2}}\lambda^{N_2-1}\right)\\
&\lesssim l^{-20}\lambda^{N-1}r_{\perp}^{\frac{1}{p}-\frac{1}{2}}\tau^{\frac{1}{2}-\frac{1}{\gamma}},
    \end{align*}
    which verifies \eqref{3.53s}.

    Regarding the estimates of temporal correctors, using \eqref{eq2.21}, \eqref{eq2.22}, \eqref{3.8s}, \eqref{3.43s} and \eqref{3.44s} and Lemmas \ref{lem-mikado-flows}, \ref{lem-magnetic-amplitude-estimate} and \ref{lem-momentum-amplitude-estimate}, we get
    \begin{align*}
        &\quad\|\nabla^Nw_{q+1}^{(o)}\|_{L^{\gamma}_tL^p_x}+\sum_{i=1}^3\|\nabla^Nf_{q+1}^{i,(o)}\|_{L^{\gamma}_tL^p_x}\\     &\lesssim\sigma^{-1}\sum_{k\in\Lambda_m}\|h_{(k)}\|_{C_t}\left\|\nabla^{N+1}\left(\rho_{l}^{-1}a^2_{(k)}\right)\right\|_{C_{t,x}}\\      &\quad+\sigma^{-1}\sum_{i=1}^3\sum_{k\in\Lambda_{\mathcal{N}^i}}\|h_{(k)}\|_{C_t}\left\|\nabla^{N+1}\left(\rho_{l}^{-1}a^2_{(k)}\right)\right\|_{C_{t,x}}\\     &\lesssim\sigma^{-1}\sum_{k\in\Lambda_m}\|h_{(k)}\|_{C_t}\sum_{N_1+N_2=N+1}\|\rho_l^{-1}\|_{C^{N_1}_{t,x}}\|a^2_{(k)}\|_{C^{N_2}_{t,x}}\\       &\quad+\sigma^{-1}\sum_{i=1}^3\sum_{k\in\Lambda_{\mathcal{N}^i}}\|h_{(k)}\|_{C_t}\sum_{N_1+N_2=N+1}\|\rho_l^{-1}\|_{C^{N_1}_{t,x}}\|a^2_{(k)}\|_{C^{N_2}_{t,x}}\\
        &\lesssim\sigma^{-1}\lambda_q^{\frac{(N+1)\varepsilon}{4}}l^{-9N-13}\\
        &\lesssim l^{-9N-14}\sigma^{-1},
    \end{align*}
    where the last step goes under \eqref{eq2.2s} and \eqref{eq2.3s}.

    In order to establish the $C^N_{t,x}$-estimates of momentum and elastic perturbations, we use Lemmas \ref{lem-mikado-flows}--\ref{lem-momentum-amplitude-estimate}, and \eqref{eq2.3s}, \eqref{eq2.2s}, \eqref{eq2.17} and \eqref{eq3.1} to get
    \begin{align}
        &\quad\|w_{q+1}^{(p)}\|_{C^N_{t,x}}+\sum_{i=1}^3\|f_{q+1}^{i,(p)}\|_{C^N_{t,x}}\notag\\
            &\lesssim\sum_{k\in\Lambda_{m^1}\cup\Lambda_{\mathcal{N}}}\|a_{(k)}\|_{C^N_{t,x}}\sum_{0\leq N_1+N_2\leq N}\|g_{(k)}\|_{C^{N_1}_t}\|W_{(k)}\|_{C^{N_2}_{x}}\notag\\       &\quad+\sum_{i=1}^3\sum_{k\in\Lambda_{m^2_i}\cup\Lambda_{\mathcal{N}^i}}\|a_{(k)}\|_{C^N_{t,x}}\sum_{0\leq N_1+N_2\leq N}\|g_{(k)}\|_{C^{N_1}_t}\|D^i_{(k)}\|_{C^{N_2}_{x}}\label{3.57s}\\
            &\lesssim\sum_{0\leq N_1+N_2\leq N}\sigma^{N_1}\tau^{\frac{1}{2}+N_1}\lambda^{N_2}r_{\perp}^{-\frac{1}{2}}\left(l^{-9N-2}+l^{-4N-1}+l^{-2}\right)\notag\\
            &\lesssim \lambda^{4N+3}.\notag
    \end{align}

    Similarly, we obtain
        \begin{align}
            &\quad\|w_{q+1}^{(c)}\|_{C^N_{t,x}}+\sum_{i=1}^3\|f_{q+1}^{i,(c)}\|_{C^N_{t,x}}\notag\\
  &\lesssim\sum_{k\in\Lambda_{m^1}\cup\Lambda_{\mathcal{N}}}\|a_{(k)}\|_{C^{N+2}_{t,x}}\sum_{0\leq N_1+N_2\leq N}\|g_{(k)}\|_{C^{N_1}_t}\|W^c_{(k)}\|_{C^{N_2+1}_{x}}\notag\\         &\quad+\sum_{i=1}^3\sum_{k\in\Lambda_{m^2_i}\cup\Lambda_{\mathcal{N}^i}}\|a_{(k)}\|_{C^{N+2}_{t,x}}\sum_{0\leq N_1+N_2\leq N}\|g_{(k)}\|_{C^{N_1}_t}\|D^{i,c}_{(k)}\|_{C^{N_2+1}_{x}}\label{3.58s}\\
            &\lesssim\sum_{0\leq N_1+N_2\leq N}\sigma^{N_1}\tau^{\frac{1}{2}+N_1}\lambda^{N_2-1}r_{\perp}^{-\frac{1}{2}}\left(l^{-9N-20}+l^{-4N-9}+l^{-2}\right)\notag\\
            &\lesssim \lambda^{4N+2},\notag
        \end{align}
    and
     \begin{equation}\label{3.59s}
        \begin{aligned}
            &\quad\|w_{q+1}^{(o)}\|_{C^N_{t,x}}+\sum_{i=1}^3\|f_{q+1}^{i,(o)}\|_{C^N_{t,x}}\\
            &\lesssim\sigma^{-1}\sum_{k\in\Lambda_m\cup\Lambda_{\mathcal{N}}}\left\|h_{(k)}\nabla^{N+1}\left(\rho^{-1}_la^2_{(k)}\right)\right\|_{C^N_{t,x}}\\
            &\lesssim\sigma^{N-1}\tau^Nl^{-9N-5}\\
            &\lesssim\lambda^{3N+1},
        \end{aligned}
    \end{equation}
    where the last step is due to \eqref{eq2.3s}, \eqref{eq2.2s}, \eqref{eq2.17}, \eqref{eq3.1} and the fact that
    \begin{align*}
        \|h_{(k)}\|_{C^N_t}\lesssim\sigma^N\tau^N.
    \end{align*}
Thus, combining \eqref{3.57s}--\eqref{3.59s} verify \eqref{3.55s}.

It remains to check \eqref{3.56s}. This can be done by using \eqref{3.6s} and \eqref{3.50s} and Lemmas \ref{lem-mikado-flows}--\ref{lem-momentum-amplitude-estimate}:
\begin{align*}
    &\quad\|\partial_t^Mz_{q+1}\|_{C_tC^N_x}\\
    &\lesssim\sigma^{-1}\sum_{k\in\Lambda_{m^1}}\left\|\partial_t^M\int_0^th_{(k)}{\rm div}\left(k_1\otimes k_1\nabla \left(\rho_l^{-1}a^2_{(k)}\right)\right)(s,x){\rm d}s\right\|_{C_tC^N_x}\\
    &\quad+\sigma^{-1}\sum_{k\in\Lambda_{m^2}}\left\|\partial_t^M\int_0^th_{(k)}{\rm div}\left(k_2\otimes k_2\nabla \left(\rho_l^{-1}a^2_{(k)}\right)\right)(s,x){\rm d}s\right\|_{C_tC^N_x}\\            &\quad+\sigma^{-1}\sum_{k\in\Lambda_{\mathcal{N}}}\left\|\partial_t^M\int_0^th_{(k)}{\rm div}\left(\left(k_1\otimes k_1-k_2\otimes k_2\right)\nabla \left(\rho_l^{-1}a^2_{(k)}\right)\right)(s,x){\rm d}s\right\|_{C_tC^N_x}\\  &\lesssim\sigma^{-1}\sum_{k\in\Lambda_m\cup\Lambda_{\mathcal{N}}}\|h_{(k)}\|_{C_t}\|\rho_l^{-1}a^2_{(k)}\|_{C_tC^{N+2}_x}\\   &\lesssim\sigma^{-1}\sum_{k\in\Lambda_m\cup\Lambda_{\mathcal{N}}}\left(\|\rho_l^{-1}\|_{C_tC_x^{N+2}}\|a^2_{(k)}\|_{C_{t,x}}+\sum_{1\leq N^{\prime}\leq N+2}\|\rho_l^{-1}\|_{C_tC_x^{N+2-N^{\prime}}}\|a^2_{(k)}\|_{C_tC_x^{N^{\prime}}}\right)\\
    &\lesssim\sigma^{-1}\left(\lambda_q^{\frac{(N+2)\varepsilon}{4}}l^{-4}+\sum_{1\leq N^{\prime}\leq N+2}\lambda_q^{\frac{(N+2-N^{\prime})\varepsilon}{4}}l^{-9N^{\prime}-4}\right)\\
    &\lesssim l^{-9N-23}\sigma^{-1}.
\end{align*}
Therefore, the proof of Proposition \ref{prop-estimate-perturbation} has been finished.
    \hfill$\square$
\end{pf}
\subsection{Verification of inductive estimates}
We are now in the stage to verify the inductive estimates \eqref{eq2.5}--\eqref{eq2.7} and \eqref{eq2.10}--\eqref{eq2.13} for the perturbations.

Our first concern is the density function $\rho_{q+1}$. By \eqref{eq2.17}, \eqref{eq2.21}, \eqref{eq3.1}, \eqref{3.51s} and \eqref{3.56s},
\begin{align}\label{3.60s}
    \rho_{q+1}\leq\rho_l+\|z_{q+1}\|_{C_{t,x}}\leq C_2-\frac{1}{2}\lambda_q^{-\beta}+l^{-23}\sigma^{-1}\leq C_2-\frac{1}{2}\lambda_q^{-\beta}+\lambda_{q+1}^{-4\varepsilon}\leq C_2-\lambda_{q+1}^{-\beta},
\end{align}
where we used $\lambda_{q+1}^{-\varepsilon}+\lambda_{q+1}^{-\beta}\ll\frac{1}{2}\lambda_q^{-\beta}$ in the last inequality, because of \eqref{eq2.3s}. Similarly, we also have
\begin{align}\label{3.61s}
    \rho_{q+1}\geq\rho_l-\|z_{q+1}\|_{C_{t,x}}\geq C_1+\frac{1}{2}\lambda_q^{-\beta}-l^{-23}\sigma^{-1}\geq C_1+\frac{1}{2}\lambda_q^{-\beta}-\lambda_{q+1}^{-4\varepsilon}\geq C_1+\lambda_{q+1}^{-\beta}.
\end{align}
Combining \eqref{3.60s} and \eqref{3.61s} yield \eqref{eq2.5} at level $q+1$.

Next, by \eqref{3.50s} and \eqref{3.51s},
\begin{align}\label{3.62s}
    \int_{\mathbb{T}^3}\rho_{q+1}(t,x){\rm d}x=\int_{\mathbb{T}^3}\rho_l(t,x){\rm d}x+\int_{\mathbb{T}^3}z_{q+1}(t,x){\rm d}x=\int_{\mathbb{T}^3}\rho_{q}(t,x){\rm d}x,
\end{align}
for any $t\in[0,T]$, which gives that \eqref{eq2.10} at level $q+1$.

Moreover, for any $0\leq N\leq4$, $0\leq M\leq1$, we obtain 
\begin{equation}\label{3.63s}
    \begin{aligned}
        \|\partial_t^M\rho_{q+1}\|_{C_tC^N_x}
        &\lesssim\|\partial_t^M\rho_{l}\|_{C_tC^N_x}+\|\partial_t^Mz_{q+1}\|_{C_tC^N_x}\\
        &\lesssim\lambda_q^{\frac{N\varepsilon}{4}}+l^{-9N-23}\sigma^{-1}\\
        &\lesssim \lambda_{q+1}^{\frac{N\varepsilon}{4}},
    \end{aligned}
\end{equation}
by using \eqref{eq2.22} and \eqref{3.56s} which leads \eqref{eq2.6} at level $q+1$.

We also derive from \eqref{eq2.23} and \eqref{3.56s} that
\begin{equation}\label{3.64s}
    \begin{aligned}
        \|\rho_{q+1}-\rho_q\|_{C_tC^1_x}
        &\lesssim\|\rho_{l}-\rho_q\|_{C_tC^1_x}+\|z_{q+1}\|_{C_tC^1_x}\\
        &\lesssim l\lambda_q^{\frac{\varepsilon}{2}}+l^{-32}\sigma^{-1}\\
        &\lesssim\delta^{\frac{1}{2}}_{q+2},
    \end{aligned}
\end{equation}
where the last step goes under \eqref{eq2.3s}, \eqref{eq2.4}, \eqref{eq2.17} and \eqref{eq3.1}, which verifies \eqref{eq2.11}.

Then, we consider the momentum and elastic perturbations. To begin with, by virtue of \eqref{eq2.24} and \eqref{3.55s}, we have
\begin{equation}\label{3.65s}
    \begin{aligned}
        \|m_{q+1}\|_{C^N_{t,x}}+\|\mathcal{N}_{q+1}\|_{C^N_{t,x}}
        &\lesssim \|m_{l}\|_{C^N_{t,x}}+ \|w_{q+1}\|_{C^N_{t,x}}+\sum_{i=1}^3\left(\|\mathcal{N}^i_{l}\|_{C^N_{t,x}}+ \|f^i_{q+1}\|_{C^N_{t,x}}\right)\\
        &\lesssim l^{-N+1}\lambda^{4N+3}_q+\lambda_{q+1}^{4N+3}\\
        &\lesssim\lambda_{q+1}^{4N+3},
    \end{aligned}
\end{equation}
which leads \eqref{eq2.7} at level $q+1$.

Note that, the estimates in \eqref{3.52s} alone do not yield the decay required by \eqref{eq2.12}. The technique here is the key $L^p$-decorrelation, which permits to derive the $L^2_{t,x}$-decay of the principal parts $w^{(p)}_{q+1}$ and $f^{i,(p)}_{q+1}$.
\begin{lem}\label{lem-lp-decorrelation}(\cite[Lemma 3.7]{BV-2019-annmath};  \cite[Lemma 2.4]{CL-2021-annpde}). Let $\theta\in\mathbb{N}$ and $f,g:\mathbb{T}^d\rightarrow\mathbb{R}$ be smooth functions. Then for any $p\in[1,+\infty]$,
\begin{align*}
    \left|\|fg(\theta\cdot)\|_{L^p(\mathbb{T}^d)}-\|f\|_{L^p(\mathbb{T}^d)}\|g\|_{L^p(\mathbb{T}^d)}\right|\lesssim\theta^{-\frac{1}{p}}\|f\|_{C^1(\mathbb{T}^d)}\|g\|_{L^p(\mathbb{T}^d)}.
\end{align*}    
\end{lem}
We apply Lemma \ref{lem-lp-decorrelation} with $f=a_{(k)}$, $g=g_{(k)}\phi_{(k)}$ and $\theta=\lambda_{q+1}^{5\varepsilon}$, furthermore, by  \eqref{eq2.3s}, \eqref{eq2.2s} and  Lemmas \ref{lem-mikado-flows}--\ref{lem-momentum-amplitude-estimate}, we obtain
\begin{align}\label{3.66s}
    \begin{aligned}
       &\quad\|w_{q+1}^{(p)}\|_{L^2_{t,x}}+\sum_{i=1}^3\|f_{q+1}^{i,(p)}\|_{L^2_{t,x}}\\ &\lesssim\sum_{k\in\Lambda_{m}\cup\Lambda_{\mathcal{N}}}\left(\|a_{(k)}\|_{L^2_{t,x}}\|g_{(k)}\|_{L^2_{t}}\|\phi_{(k)}\|_{L^2_{x}}+\lambda^{-\frac{5\varepsilon}{2}}\|a_{(k)}\|_{C^1_{t,x}}\|g_{(k)}\|_{L^2_{t}}\|\phi_{(k)}\|_{L^2_{x}}\right)\\
       &\lesssim\delta_{q+1}^{\frac{1}{2}}+\lambda^{-\frac{5\varepsilon}{2}}(l^{-6}+l^{-11})\\
       &\lesssim\delta_{q+1}^{\frac{1}{2}}.
    \end{aligned}
\end{align}
Then, taking into account \eqref{eq2.3s}, \eqref{3.53s}, \eqref{3.54s} and \eqref{3.66s}, we deduce
\begin{equation}\label{3.67s}
    \begin{aligned}
       &\quad\|w_{q+1}\|_{L^2_{t,x}}+\sum_{i=1}^3\|f^i_{q+1}\|_{L^2_{t,x}}\\
       &\lesssim\|w^{(p)}_{q+1}\|_{L^2_{t,x}}+\|w^{(c)}_{q+1}\|_{L^2_{t,x}}+\|w^{(o)}_{q+1}\|_{L^2_{t,x}}\\
       &\quad+\sum_{i=1}^3\left(\|f^{i,(p)}_{q+1}\|_{L^2_{t,x}}+\|f^{i,(c)}_{q+1}\|_{L^2_{t,x}}+\|f^{i,(o)}_{q+1}\|_{L^2_{t,x}}\right)\\
       &\lesssim \delta_{q+1}^{\frac{1}{2}}+l^{-20}\lambda^{-1}+l^{-14}\sigma^{-1}\\
       &\lesssim\delta_{q+1}^{\frac{1}{2}}.
    \end{aligned}
\end{equation}
Moreover, using \eqref{eq2.3s} and Proposition \ref{prop-estimate-perturbation} with $\gamma=1$, $p=2$ and $N=0$ yields that
\begin{equation}\label{3.68s}
    \begin{aligned}
       &\quad\|w_{q+1}\|_{L^1_tL^2_x}+\sum_{i=1}^3\|f^i_{q+1}\|_{L^1_tL^2_x}\\   &\lesssim\|w^{(p)}_{q+1}\|_{L^1_tL^2_x}+\|w^{(c)}_{q+1}\|_{L^1_tL^2_x}+\|w^{(o)}_{q+1}\|_{L^1_tL^2_x}\\
       &\quad+\sum_{i=1}^3\left(\|f^{i,(p)}_{q+1}\|_{L^1_tL^2_x}+\|f^{i,(c)}_{q+1}\|_{L^1_tL^2_x}+\|f^{i,(o)}_{q+1}\|_{L^1_tL^2_x}\right)\\
       &\lesssim l^{-2}\tau^{-\frac{1}{2}}+l^{-20}\lambda^{-1}\tau^{-\frac{1}{2}}+l^{-14}\sigma^{-1}\\
       &\lesssim\delta_{q+2}^{\frac{1}{2}}.
    \end{aligned}
\end{equation}
Furthermore, by combining with \eqref{eq2.25}, \eqref{3.67s} and \eqref{3.68s}, we have
    \begin{align*}
       &\quad\|m_{q+1}-m_q\|_{L^2_{t,x}}+\|\mathcal{N}_{q+1}-\mathcal{N}_q\|_{L^2_{t,x}}\\
       &\lesssim \|m_{q+1}-m_l\|_{L^2_{t,x}}+\|m_{l}-m_q\|_{L^2_{t,x}}+\|\mathcal{N}_{q+1}-\mathcal{N}_l\|_{L^2_{t,x}}+\|\mathcal{N}_{l}-\mathcal{N}_q\|_{L^2_{t,x}}\\
       &\lesssim \|m_{l}-m_q\|_{C_{t,x}}+ \|\mathcal{N}_{l}-\mathcal{N}_q\|_{C_{t,x}}+\|w_{q+1}\|_{L^2_{t,x}}+\sum_{i=1}^3\|f^i_{q+1}\|_{L^2_{t,x}}\\
       &\lesssim l\lambda^5_q+\delta_{q+1}^{\frac{1}{2}}\\
       &\lesssim\delta_{q+1}^{\frac{1}{2}}.
    \end{align*}
and
    \begin{align*}
       &\quad\|m_{q+1}-m_q\|_{L^1_tL^2_x}+\|\mathcal{N}_{q+1}-\mathcal{N}_q\|_{L^1_tL^2_x}\\
       &\lesssim \|m_{q+1}-m_l\|_{L^1_tL^2_x}+\|m_{l}-m_q\|_{L^1_tL^2_x}+\|\mathcal{N}_{q+1}-\mathcal{N}_l\|_{L^1_tL^2_x}+\|\mathcal{N}_{l}-\mathcal{N}_q\|_{L^1_tL^2_x}\\
       &\lesssim \|m_{l}-m_q\|_{C_{t,x}}+ \|b_{l}-b_q\|_{C_{t,x}}+\|w_{q+1}\|_{L^1_tL^2_x}+\sum_{i=1}^3\|f^i_{q+1}\|_{L^1_tL^2_x}\\
       &\lesssim l\lambda^5_q+\delta_{q+2}^{\frac{1}{2}}\\
       &\lesssim\delta_{q+2}^{\frac{1}{2}}.
    \end{align*}

Thus far, the inductive estimates \eqref{eq2.5}--\eqref{eq2.7} and \eqref{eq2.10}--\eqref{eq2.13} are verified at level $q+1$.
\section{Elastic stress fluctuation and Reynolds stress}\label{sec4}
We are now in the stage to treat the delicate elastic stress fluctuation and Reynolds stress and check the corresponding inductive estimates \eqref{eq2.8} and \eqref{eq2.9} at level $q+1$.

To begin with, we recall from \cite{LZZ-2022-jmpa} the key inverse-divergence operators $\mathcal{R}^m$ and $\mathcal{R}^n$, which are defined by
\begin{align}  &\left(\mathcal{R}^mv\right)_{kl}:=\partial_k\Delta^{-1}v_l+\partial_l\Delta^{-1}v_k-\frac{1}{2}\left(\delta_{kl}+\partial_k\partial_l\Delta^{-1}\right){\rm div}\Delta^{-1}v,\label{eq4.1}\\
&\left(\mathcal{R}^nf\right)_{ij}=\epsilon_{ijk}(-\Delta)^{-1}\left({\rm curl}f\right)_k,\label{eq4.2}
\end{align}
where $\epsilon_{ijk}$ is the Levi-Civita tensor and the smooth vector fields $v$ and $f$ are both mean-free, moreover, $f$ is divergence-free.

The operator $\mathcal{R}^m$ maps smooth vector field to symmetric and trace-free matrices, while the operator $\mathcal{R}^n$ returns skew-symmetric matrices. Moreover, one has the following identities
\begin{align}\label{eq4.3}
    {\rm div}\mathcal{R}^m(v)=v,\quad{\rm div}\mathcal{R}^n(f)=f.
\end{align}
We note that $\mathcal{R}^m\nabla$, $\mathcal{R}^m{\rm div}$, $\mathcal{R}^n\nabla$ and $\mathcal{R}^n{\rm div}$ are all Calder\'on-Zygmund operators and thus they are bounded in periodic $L^{\widetilde{p}}$ spaces with $\widetilde{p}\in(1,+\infty)$. See \cite{CZ-1954-studiamath,DCS-2013-invent,BBV-2020-annpde} for details.
\subsection{Decomposition of elastic stress fluctuation and Reynolds stress}
For more precise estimates, we need to decompose elastic stress fluctuation and Reynolds stress appropriately.
\subsubsection{Decomposition of elastic stress fluctuation}
Using $\eqref{eq2.1}_3$ at level $q+1$, $\eqref{eq2.18}_3$, \eqref{eq2.20}, \eqref{3.18s}, \eqref{3.35s} and \eqref{3.46s}--\eqref{3.48s}, we derive that for $i=1,2,3$,
\begin{equation}\label{eq4.4}
    \begin{aligned}
        {\rm div}M^i_{q+1}=
&\partial_t\left(f_{q+1}^{i,(p)}+f_{q+1}^{i,(c)}\right)+{\rm div}\left((\rho_{q+1}^{-1}-\rho_l^{-1})(\mathcal{N}^i_l\otimes m_l-m_l\otimes \mathcal{N}^i_l)\right)\\
&\underbrace{+{\rm div}\left(\rho_{q+1}^{-1}(\mathcal{N}^i_l\otimes w_{q+1}+f^i_{q+1}\otimes m_l-m_l\otimes f^i_{q+1}-w_{q+1}\otimes \mathcal{N}^i_l)\right)\quad}_{{\rm div}M^i_{{\rm lin}}}\\
&+{\rm div}\left(\rho_l^{-1}(f_{q+1}^{i,(p)}\otimes w_{q+1}^{(p)}-w_{q+1}^{(p)}\otimes f_{q+1}^{i,(p)})+M^i_l\right)+\partial_tf_{q+1}^{i,(o)}\\
&\underbrace{+{\rm div}\left((\rho_{q+1}^{-1}-\rho_l^{-1})(f^i_{q+1}\otimes w_{q+1}-w_{q+1}\otimes f^i_{q+1})\right)\qquad\quad}_{{\rm div}M^i_{{\rm osc}}}\\
&+{\rm div}\left(\rho_l^{-1}\left(f_{q+1}^{i,(p)}\otimes (w_{q+1}^{(c)}+w_{q+1}^{(o)})-w_{q+1}^{(p)}\otimes (f_{q+1}^{i,(c)}+f_{q+1}^{i,(o)})\right)\right)\\
&\underbrace{+{\rm div}\left(\rho_l^{-1}\left((f_{q+1}^{i,(c)}+f_{q+1}^{i,(o)})\otimes w_{q+1}-(w_{q+1}^{(c)}+w_{q+1}^{(o)})\otimes f^i_{q+1}\right)\right)\quad}_{{\rm div}M^i_{{\rm cor}}}\\
&+{\rm div}M^i_{{\rm com}},
    \end{aligned}
\end{equation}
where $M^i_{{\rm com}}$ is the commutator stress to the elastic equation given by \eqref{eq2.20}.

Using the inverse-divergence operator $\mathcal{R}^n$, we define the new elastic stress fluctuation $\left(\mathscr{M}_{q+1}\right)_{ijk}:=\left(M^i_{q+1}\right)_{jk}$ at level $q+1$ by
\begin{align}\label{eq4.5}
    M^i_{q+1}:=M^i_{{\rm lin}}+M^i_{{\rm osc}}+M^i_{{\rm cor}}+M^i_{{\rm com}},\quad i=1,2,3,
\end{align}
where the linear error
\begin{equation}\label{eq4.6}
    \begin{aligned}
        M^i_{{\rm lin}}&:=
        \mathcal{R}^n\partial_t\left(f_{q+1}^{i,(p)}+f_{q+1}^{i,(c)}\right)+\mathcal{R}^n{\rm div}\left((\rho_{q+1}^{-1}-\rho_l^{-1})(\mathcal{N}^i_l\otimes m_l-m_l\otimes \mathcal{N}^i_l)\right)\\
&\quad+\mathcal{R}^n{\rm div}\left(\rho_{q+1}^{-1}(\mathcal{N}^i_l\otimes w_{q+1}+f^i_{q+1}\otimes m_l-m_l\otimes f^i_{q+1}-w_{q+1}\otimes \mathcal{N}^i_l)\right)\\
&=:M^i_{{\rm lin},1}+M^i_{{\rm lin},2}+M^i_{{\rm lin},3},
    \end{aligned}
\end{equation}
the oscillation error
\begin{equation}\label{eq4.7}
    \begin{aligned}
        M^i_{{\rm osc}}&:=      \sum_{k\in\Lambda_{\mathcal{N}^i}}g_{(k)}^2\mathcal{R}^n\mathbb{P}_{\neq0}\left(\mathbb{P}_{\neq0}(D^i_{(k)}\otimes W_{(k)}-W_{(k)}\otimes D^i_{(k)})\nabla\left(\rho_l^{-1}a^2_{(k)}\right)\right)\\
        &\quad-\sigma^{-1}\sum_{k\in\Lambda_{\mathcal{N}^i}}h_{(k)}\partial_t\mathcal{R}^n\left((k_2\otimes k_1-k_1\otimes k_2)\nabla(\rho_l^{-1}a^2_{(k)})\right)\\
        &\quad+\mathcal{R}^n{\rm div}\left((\rho_{q+1}^{-1}-\rho_l^{-1})(f^i_{q+1}\otimes w_{q+1}-w_{q+1}\otimes f^i_{q+1})\right)\\
        &=:M^i_{{\rm osc},1}+M^i_{{\rm osc},2}+M^i_{{\rm osc},3},
    \end{aligned}
\end{equation}
the corrector error
\begin{equation}\label{eq4.8}
    \begin{aligned}
        M^i_{{\rm cor}}
        &:=
\mathcal{R}^n{\rm div}\left(\rho_l^{-1}\left(f_{q+1}^{i,(p)}\otimes (w_{q+1}^{(c)}+w_{q+1}^{(o)})-w_{q+1}^{(p)}\otimes (f_{q+1}^{i,(c)}+f_{q+1}^{i,(o)})\right)\right)\\
&\quad+\mathcal{R}^n{\rm div}\left(\rho_l^{-1}\left((f_{q+1}^{i,(c)}+f_{q+1}^{i,(o)})\otimes w_{q+1}-(w_{q+1}^{(c)}+w_{q+1}^{(o)})\otimes f^i_{q+1}\right)\right),
    \end{aligned}
\end{equation}
and $M^i_{{\rm com}}$ is the commutator error given by \eqref{eq2.20}.

We also note that, the nonlinear terms in the
elastic equation is skew-symmetric, which, in particular, yields that
\begin{align}\label{eq4.9}
    {\rm div}\left({\rm div}M^i_{q+1}\right)=0.
\end{align}
Furthermore, it leads 
\begin{align}\label{eq4.10}
    M^i_{q+1}=\mathcal{R}^n{\rm div}M^i_{q+1}.
\end{align}
\subsubsection{Decomposition of Reynolds stress}
Now we concern the Reynolds stress. By virtue of $\eqref{eq2.1}_2$, $\eqref{eq2.18}_2$, \eqref{eq2.19}, \eqref{3.36s}, \eqref{3.45s}, \eqref{3.47s} and \eqref{3.48s}, we compute
\begin{align}
           {\rm div}R^m_{q+1}=
&\partial_t\left(w_{q+1}^{(p)}+w_{q+1}^{(c)}\right)+\nu^s(-\Delta)^{\alpha}\left((\rho^{-1}_{q+1}-\rho^{-1}_l)m_l+(\rho^{-1}_{q+1}-\rho^{-1}_l)w_{q+1}+\rho^{-1}_lw_{q+1}\right)\notag\\
&-(\nu^b+\frac{1}{3}\nu^s)\nabla{\rm div}\left((\rho^{-1}_{q+1}-\rho^{-1}_l)m_l+(\rho^{-1}_{q+1}-\rho^{-1}_l)w_{q+1}+\rho^{-1}_lw_{q+1}\right)\notag\\
&+{\rm div}\left(\rho_{q+1}^{-1}(m_l\otimes w_{q+1}+w_{q+1}\otimes m_l)+(\rho^{-1}_{q+1}-\rho^{-1}_l)m_l\otimes m_l\right)\notag\\
&\underbrace{-{\rm div}\left(\rho_{q+1}^{-1}\sum_{i=1}^3\left(\mathcal{N}^i_l\otimes f^i_{q+1}+f^i_{q+1}\otimes\mathcal{N}^i_l\right)-(\rho_{q+1}^{-1}-\rho_l^{-1})\sum_{i=1}^3\mathcal{N}^i_l\otimes\mathcal{N}^i_l\right)\qquad\qquad}_{{\rm div}R^m_{{\rm lin}}}\notag\\
&+{\rm div}\left(\rho_l^{-1}\left(w_{q+1}^{(p)}\otimes w_{q+1}^{(p)}-\sum_{i=1}^3f^{i,(p)}_{q+1}\otimes f^{i,(p)}_{q+1}\right)+R^m_l\right)+\partial_tw_{q+1}^{(o)}\notag\\
&\underbrace{+{\rm div}\left((\rho_{q+1}^{-1}-\rho_l^{-1})w_{q+1}\otimes w_{q+1}-(\rho_{q+1}^{-1}-\rho_l^{-1})\sum_{i=1}^3f_{q+1}^i\otimes f_{q+1}^i\right)}_{{\rm div}R^m_{{\rm osc}}}\notag\\
&+{\rm div}\left(\rho_l^{-1}\left(w_{q+1}^{(p)}\otimes (w_{q+1}^{(c)}+w_{q+1}^{(o)})+(w_{q+1}^{(c)}+w_{q+1}^{(o)})\otimes w_{q+1}\right)\right)\notag\\
&\underbrace{-{\rm div}\left(\rho_l^{-1}\sum_{i=1}^3\left(\left(f_{q+1}^{i,(c)}+f_{q+1}^{i,(o)}\right)\otimes f^i_{q+1}+f_{q+1}^{i,(p)}\otimes\left(f_{q+1}^{i,(c)}+f_{q+1}^{i,(o)}\right)\right)\right)}_{{\rm div}R^m_{{\rm cor}}}\notag\\
&+\underbrace{\nabla P(\rho_{q+1})-\nabla P(\rho_l)}_{R^m_{{\rm pre}}}+{\rm div}R^m_{{\rm com}},\label{eq4.11}
\end{align}
where $R^m_{{\rm com}}$ is the commutator stress to the momentum equation given by \eqref{eq2.19}.

Then, using the inverse-divergence operator $\mathcal{R}^m$ we define the new Reynolds stress at level $q+1$ by
\begin{align}\label{eq4.12}
    R^m_{q+1}:=R^m_{{\rm lin}}+R^m_{{\rm osc}}+R^m_{{\rm cor}}+R^m_{{\rm pre}}+R^m_{{\rm com}},
\end{align}
where the linear error
\begin{equation}\label{eq4.13}
    \begin{aligned}
        R^m_{{\rm lin}}&:=
        \mathcal{R}^m\partial_t\left(w_{q+1}^{(p)}+w_{q+1}^{(c)}\right)\\
        &\quad+\nu^s\mathcal{R}^m(-\Delta)^{\alpha}\left((\rho^{-1}_{q+1}-\rho^{-1}_l)m_l+(\rho^{-1}_{q+1}-\rho^{-1}_l)w_{q+1}+\rho^{-1}_lw_{q+1}\right)\\
&\quad-(\nu^b+\frac{1}{3}\nu^s)\mathcal{R}^m\nabla{\rm div}\left((\rho^{-1}_{q+1}-\rho^{-1}_l)m_l+(\rho^{-1}_{q+1}-\rho^{-1}_l)w_{q+1}+\rho^{-1}_lw_{q+1}\right)\\
&\quad+\mathcal{R}^m{\rm div}\left(\rho_{q+1}^{-1}(m_l\otimes w_{q+1}+w_{q+1}\otimes m_l)+(\rho^{-1}_{q+1}-\rho^{-1}_l)m_l\otimes m_l\right)\\
&\quad-\mathcal{R}^m{\rm div}\left(\rho_{q+1}^{-1}\sum_{i=1}^3\left(\mathcal{N}^i_l\otimes f^i_{q+1}+f^i_{q+1}\otimes\mathcal{N}^i_l\right)+(\rho^{-1}_{q+1}-\rho^{-1}_l)\sum_{i=1}^3\mathcal{N}^i_l\otimes\mathcal{N}^i_l\right)\\
&=:R^m_{{\rm lin},1}+R^m_{{\rm lin},2}+R^m_{{\rm lin},3}+R^m_{{\rm lin},4}+R^m_{{\rm lin},5},
    \end{aligned}
\end{equation}
the oscillation error
\begin{equation}\label{eq4.14}
    \begin{aligned}
        R^m_{{\rm osc}}&:=\sum_{i=1}^3\sum_{k\in\Lambda_{\mathcal{N}^i}}g^2_{(k)}\mathcal{R}^m\mathbb{P}_{\neq0}\left(\mathbb{P}_{\neq0}\left(W_{(k)}\otimes W_{(k)}-D^i_{(k)}\otimes D^i_{(k)}\right)\nabla(\rho_l^{-1}a^2_{(k)})\right)\\      &\quad+\sum_{k\in\Lambda_{m^1}}\mathcal{R}^m\mathbb{P}_{\neq0}\left(g^2_{(k)}\mathbb{P}_{\neq0}(W_{(k)}\otimes W_{(k)})\nabla(\rho_l^{-1}a^2_{(k)}) \right)\\
        &\quad\underbrace{-\sum_{i=1}^3\sum_{k\in\Lambda_{m^2_i}}\mathcal{R}^m\mathbb{P}_{\neq0}\left(g^2_{(k)}\mathbb{P}_{\neq0}(D^i_{(k)}\otimes D^i_{(k)})\nabla(\rho_l^{-1}a^2_{(k)})\right)\qquad\qquad\qquad}_{R^m_{{\rm osc},1}}\\
        &\quad-\sigma^{-1}\sum_{k\in\Lambda_{m^1}\cup\Lambda_{\mathcal{N}}}h_{(k)}\mathcal{R}^m\left(\left(k_1\cdot\nabla\right)\partial_t(\rho_l^{-1}a^2_{(k)})k_1\right)\\ &\quad\underbrace{+\sigma^{-1}\sum_{k\in\Lambda_{m^2}\cup\Lambda_{\mathcal{N}}}h_{(k)}\mathcal{R}^m\left(\left(k_2\cdot\nabla\right)\partial_t(\rho_l^{-1}a^2_{(k)})k_2\right)}_{R^m_{{\rm osc},2}}\\
        &\quad\underbrace{+\mathcal{R}^m{\rm div}\left((\rho_{q+1}^{-1}-\rho_l^{-1})w_{q+1}\otimes w_{q+1}-(\rho_{q+1}^{-1}-\rho_l^{-1})\sum_{i=1}^3f_{q+1}^i\otimes f_{q+1}^i\right)}_{R^m_{{\rm osc},3}}
    \end{aligned}
\end{equation}
the corrector error
\begin{equation}\label{eq4.15}
    \begin{aligned}
        R^m_{{\rm cor}}&:=
        \mathcal{R}^m{\rm div}\left(\rho_l^{-1}\left(w_{q+1}^{(p)}\otimes (w_{q+1}^{(c)}+w_{q+1}^{(o)})+(w_{q+1}^{(c)}+w_{q+1}^{(o)})\otimes w_{q+1}\right)\right)\\
&\quad-\mathcal{R}^m{\rm div}\left(\rho_l^{-1}\sum_{i=1}^3\left((f_{q+1}^{i,(c)}+f_{q+1}^{i,(o)})\otimes f^i_{q+1}+f_{q+1}^{i,(p)}\otimes (f_{q+1}^{i,(c)}+f_{q+1}^{i,(o)})\right)\right),
    \end{aligned}
\end{equation}
the pressure error 
\begin{equation}\label{eq4.16}
    R^m_{{\rm pre}}
:=\mathcal{R}^m\nabla\left(P(\rho_{q+1})-P(\rho_l)\right),
\end{equation}
and the commutator error
\begin{equation}\label{eq4.17}
   \begin{aligned}
        R^m_{{\rm com}}&:=
    \nu^s\mathcal{R}^m(-\Delta)^{\alpha}\left(\rho_l^{-1}m_l-((\rho_q^{-1}m_q)*_x\phi_l)*_t\varphi_l\right)+\mathcal{R}^m\nabla\left(P(\rho_l)-P_l\right)\\
        &\quad-(\nu^b+\frac{1}{3}\nu^s)\mathcal{R}^m\nabla{\rm div}\left(\rho_l^{-1}m_l-((\rho_q^{-1}m_q)*_x\phi_l)*_t\varphi_l\right)\\
        &\quad+\mathcal{R}^m{\rm div}\left(\rho_l^{-1}m_l\otimes m_l-\left(\rho_q^{-1}m_q\otimes m_q\right)*_x\phi_l*_t\varphi_l\right)\\
        &\quad-\mathcal{R}^m{\rm div}\left(\rho_l^{-1}\sum_{i=1}^3\mathcal{N}^i_l\otimes \mathcal{N}^i_l-\sum_{i=1}^3\left(\rho_q^{-1}\mathcal{N}^i_q\otimes \mathcal{N}^i_q\right)*_x\phi_l*_t\varphi_l\right)\\
        &=:R^m_{{\rm com},1}+R^m_{{\rm com},2}+R^m_{{\rm com},3}+R^m_{{\rm com},4}+R^m_{{\rm com},5}.
   \end{aligned}
\end{equation}

By \eqref{3.36s}--\eqref{3.46s} and the fact that for any $i=1,2,3$, ${\rm div}(W_{(k)}\otimes W_{(k)})={\rm div}(D^i_{(k)}\otimes D^i_{(k)})=0$, we have
\begin{align}\label{eq4.18}
    R^m_{q+1}=\mathcal{R}^m{\rm div}R^m_{q+1}.
\end{align}
\subsection{Verification of $C^1_{t,x}$-estimates}
The purpose of this subsection is to verify the $C^1_{t,x}$-estimates \eqref{eq2.8} of $R^m_{q+1}$ and $M^i_{q+1}$ for all $i=1,2,3$. 

First, by \eqref{eq4.10}, \eqref{eq4.18}, the Sobolev embedding $W^{1,6}_x\hookrightarrow C_x$, equations $\eqref{eq2.1}_2$ and $\eqref{eq2.1}_3$ at level $q+1$ and the fact $\mathcal{R}^m,\mathcal{R}^n\sim|\nabla|^{-1}$ in Sobolev spaces $W^{s,p}$, we obtain
\begin{equation}\label{eq4.19}
    \begin{aligned}
        &\quad\|R^m_{q+1}\|_{C_tC^1_x}+\sum_{i=1}^3\|M^i_{q+1}\|_{C_tC^1_x}\\
        &\lesssim\left\|\mathcal{R}^m\left({\rm div}R^m_{q+1}\right)\right\|_{C_tW^{2,6}_x}+\sum_{i=1}^3\left\|\mathcal{R}^n\left({\rm div}M^i_{q+1}\right)\right\|_{C_tW^{2,6}_x}\\
        &\lesssim\|\partial_tm_{q+1}\|_{C_tW^{1,6}_x}+\sum_{i=1}^3\|\partial_t\mathcal{N}^i_{q+1}\|_{C_tW^{1,6}_x}+\|\nu^s(-\Delta)^{\alpha}(\rho_{q+1}^{-1}m_{q+1})\|_{C_tW^{1,6}_x}\\
        &\quad+\|\nabla P(\rho_{q+1})\|_{C_tW^{1,6}_x}+\left\|(\nu^b+\frac{1}{3}\nu^s)\nabla{\rm div}(\rho_{q+1}^{-1}m_{q+1})\right\|_{C_tW^{1,6}_x}\\
        &\quad+\sum_{i=1}^3\left\|{\rm div}\left(\rho_{q+1}^{-1}(\mathcal{N}^i_{q+1}\otimes m_{q+1}-m_{q+1}\otimes \mathcal{N}^i_{q+1})\right)\right\|_{C_tW^{1,6}_x}\\
        &\quad+\left\|{\rm div}\left(\rho_{q+1}^{-1}\left(m_{q+1}\otimes m_{q+1}-\sum_{i=1}^3\mathcal{N}_{q+1}^i\otimes\mathcal{N}^i_{q+1}\right)\right)\right\|_{C_tW^{1,6}_x}.
    \end{aligned}
\end{equation}
Note that, by interpolation inequality and \eqref{eq2.5} at level $q+1$, we have
\begin{equation}\label{eq4.20}
    \begin{aligned}
      &\quad\|\nu^s(-\Delta)^{\alpha}(\rho_{q+1}^{-1}m_{q+1})\|_{C_tW^{1,6}_x}\\
        &\lesssim\|\rho_{q+1}^{-1}m_{q+1}\|^{1-\frac{2\alpha+1}{4}}_{C_tC_x}\|\rho_{q+1}^{-1}m_{q+1}\|^{\frac{2\alpha+1}{4}}_{C_tW_x^{4,\infty}}\\
        &\lesssim\|m_{q+1}\|_{C_{t,x}}^{\frac{3-2\alpha}{4}}\left(\|\rho_{q+1}^{-1}\|_{C_tC_x^4}\|m_{q+1}\|_{C_{t,x}^4}\right)^{\frac{2\alpha+1}{4}}.
    \end{aligned}
\end{equation}
Moreover,
\begin{equation}\label{4.21s}
        \left\|(\nu^b+\frac{1}{3}\nu^s)\nabla{\rm div}(\rho_{q+1}^{-1}m_{q+1})\right\|_{C_tW^{1,6}_x}\lesssim\|\rho_{q+1}^{-1}m_{q+1}\|_{C_tC_x^3}\lesssim\|\rho_{q+1}^{-1}\|_{C_tC_x^3}\|m_{q+1}\|_{C_tC^3_x},
\end{equation}
Furthermore, by Leibniz rule, we get
\begin{equation}\label{4.22s}
    \begin{aligned}
        &\quad\sum_{i=1}^3\left\|{\rm div}\left(\rho_{q+1}^{-1}(\mathcal{N}^i_{q+1}\otimes m_{q+1}-m_{q+1}\otimes \mathcal{N}^i_{q+1})\right)\right\|_{C_tW^{1,6}_x}\\
        &\lesssim\sum_{i=1}^3\|\rho_{q+1}^{-1}(\mathcal{N}^i_{q+1}\otimes m_{q+1}-m_{q+1}\otimes \mathcal{N}^i_{q+1})\|_{C_tC^2_x}\\      &\lesssim\|\rho_{q+1}^{-1}\|_{C_tC^2_x}\sum_{N_1+N_2=2}\sum_{i=1}^3\|m_{q+1}\|_{C_tC^{N_1}_x}\|\mathcal{N}^i_{q+1}\|_{C_tC^{N_2}_x},
    \end{aligned}
\end{equation}
and
\begin{align}
    &\quad\left\|{\rm div}\left(\rho_{q+1}^{-1}\left(m_{q+1}\otimes m_{q+1}-\sum_{i=1}^3\mathcal{N}^i_{q+1}\otimes\mathcal{N}^i_{q+1}\right)\right)\right\|_{C_tW^{1,6}_x}\notag\\
        &\lesssim\left\|\rho_{q+1}^{-1}\left(m_{q+1}\otimes m_{q+1}-\sum_{i=1}^3\mathcal{N}^i_{q+1}\otimes\mathcal{N}^i_{q+1}\right)\right\|_{C_tC^2_x}\notag\\     &\lesssim\|\rho_{q+1}^{-1}\|_{C_tC^2_x}\sum_{N_1+N_2=2}\|m_{q+1}\|_{C_tC^{N_1}_x}\|m_{q+1}\|_{C_tC^{N_2}_x}\label{4.23s}\\      &\quad+\|\rho_{q+1}^{-1}\|_{C_tC^2_x}\sum_{N_1+N_2=2}\sum_{i=1}^3\|\mathcal{N}^i_{q+1}\|_{C_tC^{N_1}_x}\|\mathcal{N}^i_{q+1}\|_{C_tC^{N_2}_x}.\notag
\end{align}
Then, plugging \eqref{eq4.20}--\eqref{4.23s} into \eqref{eq4.19} yields that
\begin{equation}\label{4.24s}
    \begin{aligned}
        &\quad\|R^m_{q+1}\|_{C_tC^1_x}+\sum_{i=1}^3\|M^i_{q+1}\|_{C_tC^1_x}\\       &\lesssim\|\partial_tm_{q+1}\|_{C_tW^{1,6}_x}+\sum_{i=1}^3\|\partial_t\mathcal{N}^i_{q+1}\|_{C_tW^{1,6}_x}+\|\nabla P(\rho_{q+1})\|_{C_tW^{1,6}_x}\\
        &\quad+\|m_{q+1}\|_{C_{t,x}}^{\frac{3-2\alpha}{4}}\left(\|\rho_{q+1}^{-1}\|_{C_tC_x^4}\|m_{q+1}\|_{C_{t,x}^4}\right)^{\frac{2\alpha+1}{4}}\\      &\quad+\|\rho_{q+1}^{-1}\|_{C_tC^2_x}\sum_{N_1+N_2=2}\sum_{i=1}^3\|m_{q+1}\|_{C_tC^{N_1}_x}\|\mathcal{N}^i_{q+1}\|_{C_tC^{N_2}_x}\\
        &\quad+\|\rho_{q+1}^{-1}\|_{C_tC^2_x}\sum_{N_1+N_2=2}\left(\|m_{q+1}\|_{C_tC^{N_1}_x}\|m_{q+1}\|_{C_tC^{N_2}_x}+\sum_{i=1}^3\|\mathcal{N}^i_{q+1}\|_{C_tC^{N_1}_x}\|\mathcal{N}^i_{q+1}\|_{C_tC^{N_2}_x}\right).
    \end{aligned}
\end{equation}

Using \eqref{eq2.5} and \eqref{eq2.6} at level $q+1$, we obtain
\begin{equation}\label{4.25s}
    \|\partial_t^M\rho_{q+1}^{-1}\|_{C_tC^N_x}\lesssim\lambda^{\frac{N\varepsilon}{4}}.
\end{equation}
Moreover, since $P^{\prime}$ and $P^{\prime\prime}$ are continuous and $\rho_{q+1}$ is uniformly bounded away from zero and infinity, that is, by Leibniz rule again, we get
\begin{equation}\label{4.26s}
    \|\nabla P(\rho_{q+1})\|_{C_tC^1_x}\lesssim\|P^{\prime}(\rho_{q+1})\|_{C_{t,x}}\|\rho_{q+1}\|_{C_tC^2_x}+\|P^{\prime\prime}(\rho_{q+1})\|_{C_{t,x}}\|\rho_{q+1}\|^2_{C_tC^1_x}\lesssim\lambda^{\varepsilon}.
\end{equation}
Thus, combining with \eqref{eq2.2s}, \eqref{4.24s}--\eqref{4.26s} and \eqref{eq2.7} at level $q+1$, it yields
\begin{equation}\label{4.27s}
\|R^m_{q+1}\|_{C_tC^1_x}+\sum_{i=1}^3\|M^i_{q+1}\|_{C_tC^1_x}\lesssim\lambda_{q+1}^{19+\varepsilon}+\lambda_{q+1}^{15+\frac{3\varepsilon}{4}}+\lambda_{q+1}^{14+\frac{\varepsilon}{2}}+\lambda_{q+1}^{\varepsilon}\lesssim\lambda_{q+1}^{20}.
\end{equation}

Now we turn to the $C^1_tC_x$-estimates. Similarly, we deduce
\begin{align}
           &\quad\|\partial_tR^m_{q+1}\|_{C_{t,x}}+\sum_{i=1}^3\|\partial_tM^i_{q+1}\|_{C_{t,x}}\notag\\
        &=\left\|\partial_t\mathcal{R}^m\left({\rm div}R^m_{q+1}\right)\right\|_{C_{t,x}}+\sum_{i=1}^3\left\|\partial_t\mathcal{R}^n\left({\rm div}M^i_{q+1}\right)\right\|_{C_{t,x}}\notag\\      
        &\lesssim\left\|\partial_t\mathcal{R}^m\left({\rm div}R^m_{q+1}\right)\right\|_{C_tW^{1,6}_x}+\sum_{i=1}^3\left\|\partial_t\mathcal{R}^n\left({\rm div}M^i_{q+1}\right)\right\|_{C_tW^{1,6}_x}\notag\\
&\lesssim\|\partial_t{\rm div}R^m_{q+1}\|_{C_{t}L^{6}_x}+\sum_{i=1}^3\|\partial_t{\rm div}M^i_{q+1}\|_{C_{t}L^{6}_x}\notag\\ 
        &\lesssim\|\partial_t{\rm div}R^m_{q+1}\|_{C_{t}L^{\infty}_x}+\sum_{i=1}^3\|\partial_t{\rm div}M^i_{q+1}\|_{C_{t}L^{\infty}_x}\notag\\ &\lesssim\|\partial^2_tm_{q+1}\|_{C_tL^{\infty}_x}+\sum_{i=1}^3\|\partial^2_t\mathcal{N}^i_{q+1}\|_{C_tL^{\infty}_x}+\|\nu^s(-\Delta)^{\alpha}\partial_t(\rho_{q+1}^{-1}m_{q+1})\|_{C_tL^{\infty}_x}\notag\\
        &\quad+\left\|(\nu^b+\frac{1}{3}\nu^s)\nabla{\rm div}\partial_t(\rho_{q+1}^{-1}m_{q+1})\right\|_{C_tL^{\infty}_x}+\|\nabla\partial_t P(\rho_{q+1})\|_{C_tL^{\infty}_x}\notag\\
        &\quad+\sum_{i=1}^3\left\|{\rm div}\partial_t\left(\rho_{q+1}^{-1}(\mathcal{N}^i_{q+1}\otimes m_{q+1}-m_{q+1}\otimes \mathcal{N}^i_{q+1})\right)\right\|_{C_tL^{\infty}_x}\notag\\
        &\quad+\left\|{\rm div}\partial_t\left(\rho_{q+1}^{-1}\left(m_{q+1}\otimes m_{q+1}-\sum_{i=1}^3\mathcal{N}^i_{q+1}\otimes\mathcal{N}^i_{q+1}\right)\right)\right\|_{C_tL^{\infty}_x}\notag\\
        &\lesssim\|m_{q+1}\|_{C^2_{t,x}}+\sum_{i=1}^3\|\mathcal{N}^i_{q+1}\|_{C^2_{t,x}}+\|(-\Delta)^{\alpha}(\partial_t\rho_{q+1}^{-1}m_{q+1})\|_{C_{t,x}}+\|(-\Delta)^{\alpha}(\rho_{q+1}^{-1}\partial_tm_{q+1})\|_{C_{t,x}}\notag\\
        &\quad+\sum_{M_1+M_2=1}\|\partial_t^{M_1}\rho_{q+1}^{-1}\|_{C_tC^2_x}\|\partial_t^{M_2}m_{q+1}\|_{C_tC^2_x}\notag\\
        &\quad+\|\partial_t\rho_{q+1}^{-1}\|_{C_tC^1_x}\sum_{\substack{N_1+N_2=2\\M_1+M_2=1}}\left(\sum_{i=1}^3\|\partial_t^{M_1}m_{q+1}\|_{C_tC^{N_1}_x}\|\partial_t^{M_2}\mathcal{N}^i_{q+1}\|_{C_tC^{N_2}_x}\right.\notag\\
        &\qquad+\left.\|\partial_t^{M_1}m_{q+1}\|_{C_tC^{N_1}_x}\|\partial_t^{M_2}m_{q+1}\|_{C_tC^{N_2}_x}+\sum_{i=1}^3\|\partial_t^{M_1}\mathcal{N}^i_{q+1}\|_{C_tC^{N_1}_x}\|\partial_t^{M_2}\mathcal{N}^i_{q+1}\|_{C_tC^{N_2}_x}\right)\notag\\
        &\quad+\|P^{\prime}(\rho_{q+1})\|_{C_{t,x}}\|\partial_t\nabla\rho_{q+1}\|_{C_{t,x}}+\|P^{\prime\prime}(\rho_{q+1})\|_{C_{t,x}}\|\nabla\rho_{q+1}\|_{C_{t,x}}\|\partial_t\rho_{q+1}\|_{C_{t,x}}.\label{4.28s}
\end{align}
Concerning the estimate of hyper-dissipative terms, using interpolation inequality, \eqref{3.65s} and \eqref{4.25s}, we get
\begin{equation}\label{4.29s}
    \begin{aligned}
        &\quad\|(-\Delta)^{\alpha}(\partial_t\rho_{q+1}^{-1}m_{q+1})\|_{C_{t,x}}+\|(-\Delta)^{\alpha}(\rho_{q+1}^{-1}\partial_tm_{q+1})\|_{C_{t,x}}\\
        &\lesssim\left(\|\partial_t\rho_{q+1}^{-1}\|_{C_{t,x}}\|m_{q+1}\|_{C_{t,x}}\right)^{1-\frac{2\alpha}{3}}\left(\|\partial_t\rho_{q+1}^{-1}\|_{C_tC^3_x}\|m_{q+1}\|_{C_tC^3_x}\right)^{\frac{2\alpha}{3}}\\
        &\quad+\left(\|\rho_{q+1}^{-1}\|_{C_{t,x}}\|\partial_tm_{q+1}\|_{C_{t,x}}\right)^{1-\frac{2\alpha}{3}}\left(\|\rho_{q+1}^{-1}\|_{C_tC^3_x}\|\partial_tm_{q+1}\|_{C_tC^3_x}\right)^{\frac{2\alpha}{3}}\\
        &\lesssim\left(\lambda^{3+\varepsilon}\right)^{1-\frac{2\alpha}{3}}\left(\lambda^{15+\varepsilon}\right)^{\frac{2\alpha}{3}}+\left(\lambda^{7+\varepsilon}\right)^{1-\frac{2\alpha}{3}}\left(\lambda^{19+\varepsilon}\right)^{\frac{2\alpha}{3}}\\
        &\lesssim\lambda^{20}.
    \end{aligned}
\end{equation}
Thus, taking into account \eqref{eq2.6} and \eqref{eq2.7} at level $q+1$, \eqref{4.25s}, \eqref{4.28s} and \eqref{4.29s}, we deduce that
    \begin{align*}
    \|\partial_tR^m_{q+1}\|_{C_{t,x}}+\sum_{i=1}^3\|\partial_tM^i_{q+1}\|_{C_{t,x}}
    &\lesssim\lambda_{q+1}^{11}+\lambda_{q+1}^{20}+\sum_{M_1+M_2=1}\lambda_{q+1}^{\frac{\varepsilon}{2}}\lambda_{q+1}^{11+4M_2}\\
    &\quad+\lambda_{q+1}^{\frac{\varepsilon}{2}}\sum_{\substack{N_1+N_2=2\\M_1+M_2=1}}\lambda_{q+1}^{4(N_1+M_1)+3}\lambda_{q+1}^{4(N_2+M_2)+3}+\lambda_{q+1}^{\varepsilon}\\
    &\lesssim\lambda_{q+1}^{20},
    \end{align*}
where we used \eqref{eq2.2s} in the last step. 

Therefore, we have shown that the $C^1_{t,x}$-estimates \eqref{eq2.8} is valid at level $q+1$.
\subsection{Verification of $L^1_{t,x}$-decay}
We are now in the stage to verify the $L^1_{t,x}$-decay \eqref{eq2.9} at level $q+1$. Since Calder\'on-Zygmund operators are bounded in $L^{\widetilde{p}}_x$ with the integrability exponent $\widetilde{p}\in(1,+\infty)$, we choose
\begin{align*}
    \widetilde{p}:=\frac{4\alpha-4+30\varepsilon}{4\alpha-4+29\varepsilon}\in(1,2),
\end{align*}
where $\varepsilon$ is given by \eqref{eq2.2s}. In particular,
\begin{align*}
    (2-2\alpha-15\varepsilon)\left(\frac{1}{\widetilde{p}}-\frac{1}{2}\right)=1-\alpha-7\varepsilon,
\end{align*}
and
\begin{align*}
    r_{\perp}^{\frac{1}{\widetilde{p}}-\frac{1}{2}}=\lambda^{1-\alpha-7\varepsilon}.
\end{align*}
Next, we will estimate the new Reynolds stresses in the space $L^1_tL^{\widetilde{p}}_x$ rather than $L^1_{t,x}$.
\subsubsection{Linear errors $R^m_{{\rm lin}}$ and $M^i_{{\rm lin}}$}
To begin with, concerning the temporal derivative terms  $M^i_{{\rm lin},1}$ and $R^m_{{\rm lin}}$ in \eqref{eq4.6} and \eqref{eq4.13}, using Lemmas \ref{lem-mikado-flows}--\ref{lem-momentum-amplitude-estimate}, we have
\begin{equation}\label{4.30s}
    \begin{aligned}
    &\quad\|R^m_{{\rm lin},1}\|_{L^1_tL^{\widetilde{p}}_x}+\sum_{i=1}^3\|M^i_{{\rm lin},1}\|_{L^1_tL^{\widetilde{p}}_x}\\
& \lesssim \sum_{k \in \Lambda_{m^1}\cup\Lambda_{\mathcal{N}}}\left\|\mathcal{R}^m {\rm curl} {\rm curl}\partial_t\left(g_{(k)} a_{(k)} W_{(k)}^c\right)\right\|_{L_t^1L^{\widetilde{p}}_x} \\
&\quad+\sum_{i=1}^3\sum_{k \in\Lambda_{m^2_i}\cup\Lambda_{\mathcal{N}^i}}\left\|\mathcal{R}^n {\rm curl} {\rm curl}\partial_t\left(g_{(k)} a_{(k)} D_{(k)}^{i,c}\right)\right\|_{L_t^1L^{\widetilde{p}}_x} \\
& \lesssim \sum_{k \in \Lambda_{m^1}\cup\Lambda_\mathcal{N}}\left(\left\|g_{(k)}\right\|_{L_t^1}\left\|a_{(k)}\right\|_{C_{t, x}^2}+\left\|\partial_t g_{(k)}\right\|_{L_t^1}\left\|a_{(k)}\right\|_{C_{t, x}^1}\right)\left\| W_{(k)}^c\right\|_{W^{1,\widetilde{p}}_x} \\
&\quad+\sum_{i=1}^3\sum_{k \in \Lambda_{m^2_i}\cup\Lambda_{\mathcal{N}^i}}\left(\left\|g_{(k)}\right\|_{L_t^1}\left\|a_{(k)}\right\|_{C_{t, x}^2}+\left\|\partial_t g_{(k)}\right\|_{L_t^1}\left\|a_{(k)}\right\|_{C_{t, x}^1}\right)\left\| D_{(k)}^{i,c}\right\|_{W^{1,\widetilde{p}}_x} \\
&\lesssim\left(l^{-20}\tau^{-\frac{1}{2}}+l^{-11}\sigma\tau^{\frac{1}{2}}\right)r_{\perp}^{\frac{1}{\widetilde{p}}-\frac{1}{2}}\lambda^{-1}\\
&\lesssim \left(l^{-20}\lambda^{-\alpha}+l^{-11}\lambda^{\alpha+5\varepsilon}\right)\lambda^{-\alpha-7\varepsilon}\\
&\lesssim l^{-20}\lambda^{-2\varepsilon},
\end{aligned}
\end{equation}
where we also used \eqref{eq2.2s} in the last inequality.

Regarding the hyper-dissipative term $R^m_{{\rm lin},2}$ in \eqref{eq4.6}, by using interpolation inequality, we obtain
\begin{equation}\label{4.31s}
    \begin{aligned}
&\quad\left\|R^m_{{\rm lin}, 2}\right\|_{L_t^1L^{\widetilde{p}}_x}\\
&\lesssim \left\||\nabla|^{2\alpha-1}\left(\rho_l^{-1} w_{q+1}+\left(\rho_{q+1}^{-1}-\rho_l^{-1}\right) m_{\ell}+\left(\rho_{q+1}^{-1}-\rho_l^{-1}\right) w_{q+1}\right)\right\|_{L_t^1 L^{\widetilde{p}}_x} \\
&\lesssim\left\|\left(\rho_l^{-1} w_{q+1}+\left(\rho_{q+1}^{-1}-\rho_l^{-1}\right) m_{\ell}+\left(\rho_{q+1}^{-1}-\rho_l^{-1}\right) w_{q+1}\right)\right\|_{L_t^1 L^{\widetilde{p}}_x}^{1-\frac{2\alpha-1}{3}} \\
&\quad\times\left\|\left(\rho_l^{-1} w_{q+1}+\left(\rho_{q+1}^{-1}-\rho_l^{-1}\right) m_{\ell}+\left(\rho_{q+1}^{-1}-\rho_l^{-1}\right) w_{q+1}\right)\right\|_{L_t^1 W^{3,{\widetilde{p}}}_x}^{\frac{2\alpha-1}{3}}\\
&\lesssim\left(\|\rho_l^{-1}\|_{C_{t,x}}\|w_{q+1}\|_{L^1_tL^{\widetilde{p}}_x}+\|\rho_{q+1}^{-1}-\rho_l^{-1}\|_{C_{t,x}}\left(\|w_{q+1}\|_{L^1_tL^{\widetilde{p}}_x}+\|m_l\|_{L^1_tL^{\widetilde{p}}_x}\right)\right)^{1-\frac{2\alpha-1}{3}}\\
&\quad\times\left(\|\rho_l^{-1}\|_{C_tC^3_x}\|w_{q+1}\|_{L^1_tW^{3,{\widetilde{p}}}_x}+\|\rho_{q+1}^{-1}-\rho_l^{-1}\|_{C_tC^3_x}\left(\|w_{q+1}\|_{L^1_tW^{3,{\widetilde{p}}}_x}+\|m_l\|_{L^1_tW^{3,{\widetilde{p}}}_x}\right)\right)^{\frac{2\alpha-1}{3}}
\end{aligned}
\end{equation}
Concerning the density, by \eqref{eq2.22}, \eqref{3.51s} \eqref{3.56s} and \eqref{4.25s}, 
\begin{equation}\label{4.32s}
    \begin{aligned}
        \|\rho_{q+1}^{-1}-\rho_l^{-1}\|_{C_tC^N_x}
        &\lesssim\|\rho_{q+1}^{-1}\|_{C_tC^N_x}\|\rho_l^{-1}\|_{C_tC^N_x}\|z_{q+1}\|_{C_tC^N_x}\\
        &\lesssim\lambda_{q+1}^{\frac{N\varepsilon}{4}}\lambda_q^{\frac{N\varepsilon}{4}}l^{-9N-23}\sigma^{-1}\\
        &\lesssim l^{-9N-24}\lambda^{-4\varepsilon}.
    \end{aligned}
\end{equation}
Hence, substituting \eqref{4.32s} into \eqref{4.31s} and then using Proposition \ref{prop-estimate-perturbation}, we have
\begin{equation}\label{4.33s}
    \begin{aligned}
        &\quad\left\|R^m_{{\rm lin}, 2}\right\|_{L_t^1 L^{\widetilde{p}}_x}\\
        &\lesssim\left(\lambda_q^{\frac{\varepsilon^2}{4}}\left(l^{-2}\tau^{-\frac{1}{2}}r_{\perp}^{\frac{1}{\widetilde{p}}-\frac{1}{2}}+l^{-14}\sigma^{-1}\right)+l^{-24}\lambda^{-4\varepsilon}\left(l\lambda^7_q+l^{-14}\sigma^{-1}+l^{-2}\tau^{-\frac{1}{2}}r_{\perp}^{\frac{1}{\widetilde{p}}-\frac{1}{2}}\right)\right)^{1-\frac{2\alpha-1}{3}}\\
        &\times\left(\lambda_q^{\frac{3\varepsilon^2}{4}}\left(l^{-2}\tau^{-\frac{1}{2}}r_{\perp}^{\frac{1}{\widetilde{p}}-\frac{1}{2}}\lambda^3+l^{-41}\sigma^{-1}\right)+l^{-51}\lambda^{-4\varepsilon}\left(l^{-2}\lambda^7_q+l^{-41}\sigma^{-1}+l^{-2}\tau^{-\frac{1}{2}}r_{\perp}^{\frac{1}{\widetilde{p}}-\frac{1}{2}}\lambda^3\right)\right)^{\frac{2\alpha-1}{3}}\\
        &\lesssim l^{-38-54\left(\frac{2\alpha-1}{3}\right)}\lambda^{-4\varepsilon}\\
        &\lesssim l^{-74}\lambda^{-4\varepsilon}.
    \end{aligned}
\end{equation}

Now we consider the bulk viscosity term $R^m_{{\rm lin},3}$ in \eqref{eq4.13}, by a direct calculation and \eqref{3.42s},
\begin{equation}\label{4.34s}
    \begin{aligned}
        \left\|R^m_{{\rm lin},3}\right\|_{L^1_tL^{\widetilde{p}}_x}     &\lesssim\left\|\nabla\rho_l^{-1}\right\|_{C_{t,x}}\left\|w_{q+1}\right\|_{L^1_tL^{\widetilde{p}}_x}+\left\|\nabla(\rho^{-1}_{q+1}-\rho_l^{-1})\right\|_{C_{t,x}}\left(\left\|m_l\right\|_{L^1_tL^{\widetilde{p}}_x}+\left\|w_{q+1}\right\|_{L^1_tL^{\widetilde{p}}_x}\right)\\
        &\quad+\left\|\rho^{-1}_{q+1}-\rho_l^{-1}\right\|_{C_{t,x}}\left(\left\|{\rm div}m_l\right\|_{L^1_tL^{\widetilde{p}}_x}+\left\|{\rm div}w^{(o)}_{q+1}\right\|_{L^1_tL^{\widetilde{p}}_x}\right)\\
        &\quad+\left\|\rho_l^{-1}\right\|_{C_{t,x}}\left\|{\rm div}w_{q+1}^{(o)}\right\|_{L^1_tL^{\widetilde{p}}_x}\\
        &\lesssim\lambda_q^{\frac{\varepsilon}{4}}\left(l^{-2}\tau^{-\frac{1}{2}}r_{\perp}^{\frac{1}{\widetilde{p}}-\frac{1}{2}}+l^{-14}\sigma^{-1}\right)+l^{-32}\lambda^{-4\varepsilon}\left(\lambda_q^{7}+l^{-2}\tau^{-\frac{1}{2}}r_{\perp}^{\frac{1}{\widetilde{p}}-\frac{1}{2}}+l^{-14}\sigma^{-1}\right)\\   &\quad+l^{-23}\lambda^{-4\varepsilon}\left(\lambda_q^{7}+l^{-23}\sigma^{-1}\right)+l^{-23}\sigma^{-1}\\
   &\lesssim l^{-46}\lambda^{-4\varepsilon}.
    \end{aligned}
\end{equation}

For the remaining terms $M^i_{{\rm lin},2}$, $M^i_{{\rm lin},3}$, $R^m_{{\rm lin},4}$ and $R^m_{{\rm lin},5}$ in \eqref{eq4.6} and \eqref{eq4.13}, using \eqref{eq2.21}, \eqref{eq2.24}, \eqref{4.32s} and Proposition \cref{prop-estimate-perturbation}, it leads that
\begin{equation}\label{4.35s}
    \begin{aligned}      &\quad\left\|R^m_{{\rm lin},4}\right\|_{L^1_tL^{\widetilde{p}}_x}+\left\|R^m_{{\rm lin},5}\right\|_{L^1_tL^{\widetilde{p}}_x}+\sum_{i=1}^3\left\|M^i_{{\rm lin},2}\right\|_{L^1_tL^{\widetilde{p}}_x}+\sum_{i=1}^3\left\|M^i_{{\rm lin},3}\right\|_{L^1_tL^{\widetilde{p}}_x}\\  &\lesssim\left\|\rho_{q+1}^{-1}\right\|_{C_{t,x}}\left(\left\|m_l\right\|_{C_{t,x}}+\sum_{i=1}^3\left\|\mathcal{N}^i_l\right\|_{C_{t,x}}\right)\left(\sum_{i=1}^3\left\|f^i_{q+1}\right\|_{L^1_tL^{\widetilde{p}}_x}+\left\|w_{q+1}\right\|_{L^1_tL^{\widetilde{p}}_x}\right)\\
    &\quad+\left\|\rho_{q+1}^{-1}-\rho_l^{-1}\right\|_{C_{t,x}}\left(\left(\left\|m_l\right\|_{C_{t,x}}+\sum_{i=1}^3\left\|\mathcal{N}^i_l\right\|_{C_{t,x}}\right)^2+\sum_{i=1}^3\left\|\mathcal{N}^i_l\right\|_{C_{t,x}}\left\|f^i_{q+1}\right\|_{L^1_tL^{\widetilde{p}}_x}\right)\\
   &\lesssim \lambda_q^7\left(l^{-2}\tau^{-\frac{1}{2}}r_{\perp}^{\frac{1}{\widetilde{p}}-\frac{1}{2}}+l^{-14}\sigma^{-1}\right)+l^{-23}\sigma^{-1}\left(\lambda_q^{14}+\lambda_q^7\left(l^{-2}\tau^{-\frac{1}{2}}r_{\perp}^{\frac{1}{\widetilde{p}}-\frac{1}{2}}+l^{-14}\sigma^{-1}\right)\right)\\
   &\lesssim l^{-37}\lambda^{-4\varepsilon}.
    \end{aligned}
\end{equation}

We now conclude from \eqref{4.30s}, \eqref{4.31s} and \eqref{4.33s}--\eqref{4.35s} altogether that
\begin{equation}\label{4.36s}
    \begin{aligned}
\left\|R^m_{{\rm lin}}\right\|_{L^1_tL^{\widetilde{p}}_x}+\sum_{i=1}^3\left\|M^i_{{\rm lin}}\right\|_{L^1_tL^{\widetilde{p}}_x}
&\lesssim l^{-20}\lambda^{-2\varepsilon}+l^{-74}\lambda^{-4\varepsilon} +l^{-46}\lambda^{-4\varepsilon}+l^{-37}\lambda^{-4\varepsilon}\\
&\lesssim l^{-74}\lambda^{-2\varepsilon}.
    \end{aligned}
\end{equation}
\subsubsection{Oscillation errors $R^m_{{\rm osc}}$ and $M^i_{{\rm osc}}$}
We now deal with the oscillation errors. First, we concern the high-low spatial oscillation errors $M^i_{{\rm osc},1}$ and $R^m_{{\rm osc},1}$ in \eqref{eq4.7} and \eqref{eq4.14}. 

The key fact is that, the momentum and elastic flows are of high oscillations while the amplitude function is slowly varying. Thus, we use the following stationary phase lemma whose proof is similar to \cite{DCS-2013-invent,LZZ-2024-jfa,lt20}, here we omit the details.
\begin{lem}(\cite[Lemma 6]{lt20}; \cite[Lemma 5.3]{LZZ-2022-jmpa})\label{lem-stationary-phase}
     Assume that $\vartheta\in C^2(\mathbb{T}^3)$. For any $\widetilde{p}\in(1,+\infty)$ we have
    \begin{align*}
        \left\||\nabla|^{-1}\mathbb{P}_{\neq0}\left(\vartheta(x)\mathbb{P}_{\geq\widetilde{\lambda}}\varphi(x)\right)\right\|_{L^{\widetilde{p}}_x(\mathbb{T}^3)}\lesssim\widetilde{\lambda}^{-1}\left\|\vartheta\right\|_{C^2_x(\mathbb{T}^3)}\|\varphi\|_{L^{\widetilde{p}}_x(\mathbb{T}^3)},
    \end{align*}
holds for any smooth functions $\varphi\in L^{\widetilde{p}}_x(\mathbb{T}^3)$ and the implicit constant is independent of $q$.
\end{lem}

Note that 
\begin{align*}
    &\mathbb{P}_{\neq0}\left(W_{(k)}\otimes W_{(k)}\right)=\mathbb{P}_{\geq(\lambda r_{\perp}/2)}\left(W_{(k)}\otimes W_{(k)}\right)=\mathbb{P}_{\geq(\lambda r_{\perp}/2)}(\phi_{(k)}^2)k_1\otimes k_1,\\
    &\mathbb{P}_{\neq0}\left(D^i_{(k)}\otimes D^i_{(k)}\right)=\mathbb{P}_{\geq(\lambda r_{\perp}/2)}\left(D^i_{(k)}\otimes D^i_{(k)}\right)=\mathbb{P}_{\geq(\lambda r_{\perp}/2)}(\phi_{(k)}^2)k_2\otimes k_2,\quad i=1,2,3,
\end{align*}
and
\begin{align*}
    &\quad\mathbb{P}_{\neq0}\left(D^i_{(k)}\otimes W_{(k)}-W_{(k)}\otimes D^i_{(k)}\right)\\
    &=\mathbb{P}_{\geq(\lambda r_{\perp}/2)}\left(D^i_{(k)}\otimes W_{(k)}-W_{(k)}\otimes D^i_{(k)}\right)\\
    &=\mathbb{P}_{\geq(\lambda r_{\perp}/2)}(\phi_{(k)}^2)(k_2\otimes k_1-k_1\otimes k_2), \quad i=1,2,3,
\end{align*}
thus, we apply Lemmas \ref{lem-magnetic-amplitude-estimate}, \ref{lem-momentum-amplitude-estimate} and \ref{lem-stationary-phase} with $\vartheta=\nabla(\rho_l^{-1}a^2_{k})$, $\varphi=\phi^2_{(k)}$ and $\widetilde{\lambda}=\lambda r_{\perp}/2$ to get
\begin{equation}\label{4.37s}
    \begin{aligned}
&\quad\left\|R^m_{{\rm osc},1}\right\|_{L_t^1 L_x^{\widetilde{p}}}+\sum_{i=1}^3\left\|M^i_{{\rm osc},1}\right\|_{L_t^1 L_x^{\widetilde{p}}}\\
& \lesssim \sum_{k \in \Lambda_{m^1}\cup\Lambda_{\mathcal{N}}}\left\|g_{(k)}\right\|_{L_t^2}^2\left\|\mathcal{R}^m \mathbb{P}_{\neq 0}\left(\mathbb{P}_{\geq(\lambda r_{\perp} / 2)}\left(W_{(k)} \otimes W_{(k)}\right) \nabla\left(\rho_{l}^{-1} a_{(k)}^2\right)\right)\right\|_{C_t L_x^{\widetilde{p}}} \\
&\quad+\sum_{i=1}^3\sum_{k \in \Lambda_{m^2_i}}\left\|g_{(k)}\right\|_{L_t^2}^2\left\|\mathcal{R}^m \mathbb{P}_{\neq 0}\left(\mathbb{P}_{\geq(\lambda r_{\perp} / 2)}\left(D^i_{(k)} \otimes D^i_{(k)}\right) \nabla\left(\rho_l^{-1}a_{(k)}^2\right)\right)\right\|_{C_t L_x^{\widetilde{p}}}\\
&\quad+\sum_{i=1}^3\sum_{k \in\Lambda_{\mathcal{N}^i}}\left\|g_{(k)}\right\|_{L_t^2}^2\left\|\mathcal{R}^n \mathbb{P}_{\neq 0}\left(\mathbb{P}_{\geq(\lambda r_{\perp} / 2)}\left(W_{(k)} \otimes W_{(k)}-D^i_{(k)} \otimes D^i_{(k)}\right) \nabla\left(\rho_{l}^{-1} a_{(k)}^2\right)\right)\right\|_{C_t L_x^{\widetilde{p}}} \\
&\quad+\sum_{i=1}^3\sum_{k \in \Lambda_{m^2_i}\cup\Lambda_{\mathcal{N}^i}}\left\|g_{(k)}\right\|_{L_t^2}^2\left\|\mathcal{R}^n \mathbb{P}_{\neq 0}\left(\mathbb{P}_{\geq(\lambda r_{\perp} / 2)}\left(D^i_{(k)} \otimes W_{(k)}-W_{(k)} \otimes D^i_{(k)}\right) \nabla\left(\rho_{l}^{-1} a_{(k)}^2\right)\right)\right\|_{C_t L_x^{\widetilde{p}}} \\
&\lesssim\sum_{k\in\Lambda_m\cup\Lambda_{\mathcal{N}}}\lambda^{-1}r^{-1}_{\perp}\left\|\nabla\left(\rho_l^{-1}a^2_{k}\right)\right\|_{C_tC^2_x}\|\phi^2_{(k)}\|_{L^{\widetilde{p}}_x}\\
&\lesssim \lambda_q^{\frac{3\varepsilon}{4}}l^{-31}\lambda^{-1}r^{\frac{1}{\widetilde{p}}-2}_{\perp}
\end{aligned}
\end{equation}

Regarding the temporal oscillation errors $M^i_{{\rm osc},2}$ and $R^m_{{\rm osc},2}$ in \eqref{eq4.7} and \eqref{eq4.14}, we use \eqref{eq2.21}, \eqref{eq2.22}, \eqref{3.8s}, \eqref{3.20s} and \eqref{3.32s} to estimate
\begin{equation}\label{4.38s}
    \begin{aligned}
&\quad\left\|R^m_{{\rm osc},3}\right\|_{L_t^1 L_x^{\widetilde{p}}}+\sum_{i=1}^3\left\|M^i_{{\rm osc},2}\right\|_{L_t^1L_x^{\widetilde{p}}}\\
& \lesssim \sigma^{-1} \sum_{k \in \Lambda_m\cup\Lambda_\mathcal{N}}\left\|h_{(k)}\right\|_{C_t}\left\|\partial_t \nabla\left(\rho_{l}^{-1} a_{(k)}^2\right)\right\|_{L_t^1 L_x^{\widetilde{p}}}\\
&\lesssim\sigma^{-1}\sum_{k\in\Lambda_m\cup\Lambda_{\mathcal{N}}}\left\|h_{(k)}\right\|_{C_t}\left(\left\|\rho^{-1}_l\right\|_{C^1_{t,x}}\left\|a_{(k)}\right\|_{C_{t,x}}\left\|a_{(k)}\right\|_{C^1_{t,x}}+\left\|\rho^{-1}_l\right\|_{C^2_{t,x}}\left\|a_{(k)}\right\|^2_{C_{t,x}}\right.\\
&\quad\qquad\qquad\qquad+\left.\left\|\rho^{-1}_l\right\|_{C_{t,x}}\left\|a_{(k)}\right\|^2_{C^1_{t,x}}+\left\|\rho^{-1}_l\right\|_{C_{t,x}}\left\|a_{(k)}\right\|_{C_{t,x}}\left\|a_{(k)}\right\|_{C^2_{t,x}}\right)\\
&\lesssim\sigma^{-1}\left(\lambda_q^{\frac{\varepsilon}{4}}l^{-13}+\lambda_q^{\frac{\varepsilon}{2}}l^{-4}+l^{-22}\right)\\
&\lesssim l^{-22}\sigma^{-1}.
\end{aligned}
\end{equation}

Finally, for the density errors $M^i_{{\rm osc},3}$ and $R^m_{{\rm osc},3}$ in \eqref{eq4.7} and \eqref{eq4.14}, by \eqref{4.32s} and Proposition \ref{prop-estimate-perturbation},
\begin{equation}\label{4.39s}
    \begin{aligned}
        &\quad\left\|R^m_{{\rm osc},3}\right\|_{L_t^1 L_x^{\widetilde{p}}}+\sum_{i=1}^3\left\|M^i_{{\rm osc},3}\right\|_{L_t^1 L_x^{\widetilde{p}}}\\
        &\lesssim\left\|\rho_{q+1}^{-1}-\rho_l^{-1}\right\|_{C_{t,x}}\left(\left\|w_{q+1}\right\|^2_{L^2_tL_x^{2\widetilde{p}}}+\sum_{i=1}^3\left(\left\|w_{q+1}\right\|_{L^2_tL_x^{2\widetilde{p}}}\left\|f^i_{q+1}\right\|_{L^2_tL_x^{2\widetilde{p}}}+\left\|f^i_{q+1}\right\|_{L^2_tL_x^{2\widetilde{p}}}^2\right)\right)\\
        &\lesssim l^{-24}\lambda^{-4\varepsilon}\left(l^{-2}\lambda^{\frac{\varepsilon}{2}}+l^{-14}\lambda^{-5\varepsilon}\right)^2\\
        &\lesssim l^{-52}\lambda^{-3\varepsilon}.
    \end{aligned}
\end{equation}

Thus, combining \eqref{4.37s}--\eqref{4.39s} altogether we conclude
\begin{equation}\label{4.40s}
    \begin{aligned}
        \left\|R^m_{{\rm osc}}\right\|_{L_t^1 L_x^{\widetilde{p}}}+\sum_{i=1}^3\left\|M^i_{{\rm osc}}\right\|_{L_t^1 L_x^{\widetilde{p}}}
        &\lesssim \lambda_q^{\frac{3\varepsilon}{4}}l^{-31}\lambda^{-1}r_{\perp}^{\frac{1}{\widetilde{p}}-2}+l^{-22}\sigma^{-1}+l^{-52}\lambda^{-3\varepsilon}\\
        &\lesssim l^{-52}\lambda^{-3\varepsilon},
    \end{aligned}
\end{equation}
where the last step is due to $2\alpha-3\leq-100\varepsilon$ implied by \eqref{eq2.2s}.
\subsubsection{Corrector errors $R^m_{{\rm cor}}$ and $M^i_{{\rm cor}}$} Now we consider the corrector errors $M^i_{{\rm cor}}$ and $R^m_{{\rm cor}}$ in \eqref{eq4.8} and \eqref{eq4.15}, using \eqref{eq2.21} and Proposition \ref{prop-estimate-perturbation} we derive
\begin{equation}\label{4.41s}
 \begin{aligned}
&\quad\left\|R^m_{{\rm cor}}\right\|_{L_t^1L_x^{\widetilde{p}}}+\sum_{i=1}^3\left\|M^i_{{\rm cor}}\right\|_{L_t^1 L_x^{\widetilde{p}}}\\  
& \lesssim\left\|\rho_{l}^{-1}\right\|_{C_{t,x}}\left\|w_{q+1}^{(c)}+w_{q+1}^{(o)}\right\|_{L_t^2 L_x^{\infty}}\left(\left\|w_{q+1}^{(p)}\right\|_{L_t^2 L_x^{\widetilde{p}}}+\left\|w_{q+1}\right\|_{L_t^2 L_x^{\widetilde{p}}}\right)\\
&\quad+\sum_{i=1}^3\left\|\rho_{l}^{-1}\right\|_{C_{t,x}}\left\|f_{q+1}^{i,(c)}+f_{q+1}^{i,(o)}\right\|_{L_t^2 L_x^{\infty}}\left(\left\|f_{q+1}^{i,(p)}\right\|_{L_t^2 L_x^{\widetilde{p}}}+\left\|f^i_{q+1}\right\|_{L_t^2 L_x^{\widetilde{p}}}\right)\\
&\quad+\sum_{i=1}^3\left\|\rho_{l}^{-1}\right\|_{C_{t,x}}\left(\left\|f_{q+1}^{i,(c)}+f_{q+1}^{i,(o)}\right\|_{L_t^2 L_x^{\infty}}\left\|w_{q+1}\right\|_{L_t^2 L_x^{\widetilde{p}}}+\left\|w_{q+1}^{(c)}+w_{q+1}^{(o)}\right\|_{L_t^2 L_x^{\infty}}\left\|f^i_{q+1}\right\|_{L_t^2 L_x^{\widetilde{p}}}\right)\\
&\lesssim\left(l^{-20}\lambda^{-1}r_{\perp}^{-\frac{1}{2}}+l^{-14}\sigma^{-1}\right)\left(l^{-2}r_{\perp}^{\frac{1}{\widetilde{p}}-\frac{1}{2}}+l^{-20}\lambda^{-1}r_{\perp}^{\frac{1}{\widetilde{p}}-\frac{1}{2}}+l^{-14}\sigma^{-1}\right)\\
&\lesssim l^{-40}\lambda^{-5\varepsilon}.
\end{aligned}
\end{equation}
\subsubsection{Pressure error $R^m_{{\rm pre}}$}
We will use the boundedness of Calder\'on-Zygmund operators and the mean value theorem to deal with the pressure error $R^m_{{\rm pre}}$ in \eqref{eq4.16}.

More  precisely,
\begin{equation}\label{4.42s}
    \begin{aligned}
\left\|R^m_{{\rm pre}}\right\|_{L_t^1 L^{\widetilde{p}}_x} & \lesssim\left\|P\left(\rho_{q+1}\right)-P\left(\rho_{l}\right)\right\|_{L^1_t L^{\widetilde{p}}_x} \\
& \lesssim\|P^{\prime}(\zeta)(\rho_{q+1}-\rho_l)\|_{L^1_tL^{\widetilde{p}}_x}\\
& \lesssim\left\|P^{\prime}(\zeta)\left(\rho_{q+1}-\rho_{l}\right)\right\|_{C_{t, x}}\\
& \lesssim\left\|P^{\prime}(\zeta)\right\|_{C_{t, x}}\left\|\rho_{q+1}-\rho_{l}\right\|_{C_{t, x}} \\
\end{aligned}
\end{equation}
where $\min\{\rho_{q+1},\rho_l\}\leq\zeta\leq\max\{\rho_{q+1},\rho_l\}$.

Since $\rho_{q+1}$ and $\rho_l$ are uniformly away from zero and infinity, and $P$, $P^{\prime}$, $P^{\prime\prime}$ are all continuous, we have
\begin{equation}\label{4.43s}
\left\|P^{\prime}(\zeta)\right\|_{C_{t, x}}\lesssim 1 .
\end{equation}
Thus, we conclude the following estimates by using \eqref{eq2.6} at level $q+1$, \eqref{eq2.22} and \eqref{3.56s}:
\begin{equation}\label{4.44s}
    \left\|R^m_{{\rm pre}}\right\|_{L_t^1 L^{\widetilde{p}}_x}\lesssim l^{-23}\lambda^{-5\varepsilon}.
\end{equation}
\subsubsection{Commutator errors $R^m_{{\rm com}}$ and $M^i_{{\rm com}}$}
It remains to deal with the commutator errors $R^m_{{\rm com}}$ and $M^i_{{\rm com}}$ given by\eqref{eq2.19} and \eqref{eq2.20}.

To begin with, concerning the shear viscous commutator $R^m_{{\rm com},1}$. By using interpolation inequality will leads 
\begin{equation}\label{4.45s}
\left\|R^m_{{\rm com},1}\right\|_{L_t^1 L^{\widetilde{p}}_x} \lesssim\left\|\rho_{l}^{-1} m_{l}-\left(\rho_q^{-1} m_q\right) *_x \phi_{l} *_t \varphi_{l}\right\|_{{L_t^1 L^{\widetilde{p}}_x}}^{1-\frac{2 \alpha-1}{2}}\left\|\rho_{l}^{-1} m_{l}-\left(\rho_q^{-1} m_q\right) *_x \phi_{l} *_t \varphi_{l}\right\|_{{L_t^1 W^{2,{\widetilde{p}}}_x}}^{\frac{2 \alpha-1}{2}} .
\end{equation}
Furthermore, by \eqref{eq2.6}, \eqref{eq2.7}, \eqref{eq2.25}, \eqref{4.25s}, and using \eqref{4.32s} at level $q$,
\begin{equation}\label{4.46s}
\begin{aligned}
& \quad\left\|\rho_{l}^{-1} m_{l}-\left(\rho_q^{-1} m_q\right) *_x \phi_{l} *_t \varphi_{l}\right\|_{L_t^1 L^{\widetilde{p}}_x} \\
&\lesssim \left\|\rho_{l}^{-1} m_{l}-\rho_{l}^{-1} m_q\right\|_{L_t^1 L^{\widetilde{p}}_x}+\left\|\rho_{l}^{-1} m_q-\rho_q^{-1} m_q\right\|_{L_t^1 L^{\widetilde{p}}_x}+\left\|\rho_q^{-1} m_q-\left(\rho_q^{-1} m_q\right) *_x \phi_{l} *_t \varphi_{l}\right\|_{L_t^1 L^{\widetilde{p}}_x} \\
 &\lesssim \left\|\rho_{l}^{-1}\right\|_{C_{t, x}}\left\|m_{l}-m_q\right\|_{C_{t, x}}+\left\|\rho_{l}^{-1}-\rho_q^{-1}\right\|_{C_{t, x}}\left\|m_q\right\|_{C_{t, x}}+l\left\|\rho_q^{-1} m_q\right\|_{C_{t, x}^1} \\
 &\lesssim l\left\|m_q\right\|_{C_{t, x}^1}+\left\|\rho_{l}^{-1}\right\|_{C_{t, x}}\left\|\rho_q^{-1}\right\|_{C_{t, x}}\left\|\rho_{l}-\rho_q\right\|_{C_{t, x}}\left\|m_q\right\|_{C_{t, x}}+l\left\|\rho_q^{-1}\right\|_{C_{t, x}^1}\left\|m_q\right\|_{C_{t, x}^1} \\
 &\lesssim l \lambda_q^7+l \lambda_q^{3+\frac{\varepsilon}{4}}+l \lambda_q^{7+\frac{\varepsilon}{4}} \lesssim l^{\frac{1}{2}},
\end{aligned}
\end{equation}
and
\begin{equation}\label{4.47s}
\begin{aligned}
& \quad\left\|\rho_{l}^{-1} m_{l}-\left(\rho_q^{-1} m_q\right) *_x \phi_{l} *_t \varphi_{l}\right\|_{L_t^1 W^{2,\widetilde{p}}_x} \\
 &\lesssim \left\|\rho_{l}^{-1}\right\|_{C_t C_x^2}\left\|m_{l}-m_q\right\|_{C_t C_x^2}+\left\|\rho_{l}^{-1}-\rho_q^{-1}\right\|_{C_t C_x^2}\left\|m_q\right\|_{C_t C_x^2}+l\left\|\rho_q^{-1} m_q\right\|_{C_t C_x^3}+l\left\|\rho_q^{-1} m_q\right\|_{C_t^1 C_x^2} \\
 &\lesssim l \lambda_q^{\frac{\varepsilon}{2}}\left\|m_q\right\|_{C_{t, x}^3}+l^{-42} \lambda^{-4 \varepsilon} \lambda_q^11+l\left\|\rho_q^{-1}\right\|_{C_t C_x^3}\left\|m_q\right\|_{C_t C_x^3}+l\left\|\rho_q^{-1}\right\|_{C_t^1 C_x^2}\left\|m_q\right\|_{C_t^1 C_x^2} \\
 &\lesssim l \lambda_q^{15+\frac{\varepsilon}{2}}+l^{-43} \lambda_{q+1}^{-4 \varepsilon}+l \lambda_q^{15+\varepsilon} \lesssim l^{\frac{1}{3}}.
\end{aligned}
\end{equation}
Thus, we conclude from \eqref{4.45s}--\eqref{4.47s} that
\begin{equation}\label{4.48s}
    \left\|R^m_{{\rm com},1}\right\|_{L_t^1 L^{\widetilde{p}}_x}\lesssim l^{\frac{1}{3}}. 
\end{equation}

Then, for the pressure commutator error $R^m_{{\rm com},2}$, similar to \eqref{4.44s}, we have
\begin{equation}\label{4.49s}
\begin{aligned}
\left\|R^m_{{\rm com},2}\right\|_{L_t^1 L^{\widetilde{p}}_x} & \lesssim\left\|P\left(\rho_{l}\right)-P_{l}\right\|_{C_{t, x}} \\
& \lesssim\left\|P\left(\rho_{l}\right)-P\left(\rho_q\right)\right\|_{C_{t, x}}+\left\|P\left(\rho_q\right)-P_{l}\right\|_{C_{t, x}}\\
& \lesssim\left\|\rho_{l}-\rho_q\right\|_{C_{t, x}}+l\left\|P\left(\rho_q\right)\right\|_{C_{t, x}^1}\\
&\lesssim l\lambda_q^{\frac{\varepsilon}{4}}\lesssim l^{\frac{1}{2}}.
\end{aligned}
\end{equation}

Next, the bulk viscous commutator error $R^m_{{\rm com},3}$ can be estimated in a similar fashion as $R^m_{{\rm com},1}$:
\begin{equation}\label{4.50s}
\left\|R^m_{{\rm com},3}\right\|_{L_t^1 L^{\widetilde{p}}_x} \lesssim\left\|\rho_{l}^{-1} m_{l}-\left(\rho_q^{-1} m_q\right) *_x \phi_{l} *_t \varphi_{l}\right\|_{C_{t}C^1_{x}}\lesssim l^{\frac{1}{3}} .
\end{equation}

At last, concerning the nonlinear commutator errors $R^m_{{\rm com},4}$, $R^m_{{\rm com},5}$ and $M^i_{{\rm com}}$, we see that
\begin{equation}\label{4.51s}
\begin{aligned}
&\quad\left\|R^m_{{\rm com},4}\right\|_{L_t^1 L^{\widetilde{p}}_x}+\left\|R^m_{{\rm com},5}\right\|_{L_t^1 L^{\widetilde{p}}_x} +\sum_{i=1}^3\left\|M^i_{{\rm com}}\right\|_{L_t^1 L^{\widetilde{p}}_x} \\ 
&\lesssim \left\|\rho_{l}^{-1} m_{l} \otimes m_{l}-\rho_q^{-1} m_q \otimes m_q\right\|_{L_t^1 L^{\widetilde{p}}_x}+\left\|\rho_q^{-1} m_q \otimes m_q-\left(\rho_q^{-1} m_q \otimes m_q\right) *_x \phi_{l} *_t \varphi_{l}\right\|_{L_t^1 L^{\widetilde{p}}_x}\\
&\quad+\sum_{i=1}^3\left(\left\|\rho_{l}^{-1} \mathcal{N}^i_{l} \otimes \mathcal{N}^i_{l}-\rho_q^{-1} \mathcal{N}^i_q \otimes \mathcal{N}^i_q\right\|_{L_t^1 L^{\widetilde{p}}_x}+\left\|\rho_q^{-1} \mathcal{N}^i_q \otimes \mathcal{N}^i_q-\left(\rho_q^{-1} \mathcal{N}^i_q \otimes \mathcal{N}^i_q\right) *_x \phi_{l} *_t \varphi_{l}\right\|_{L_t^1 L^{\widetilde{p}}_x}\right)\\
&\quad+\sum_{i=1}^3\left(\left\|\rho_{l}^{-1} \mathcal{N}^i_{l} \otimes m_{l}-\rho_q^{-1} \mathcal{N}^i_q \otimes m_q\right\|_{L_t^1 L^{\widetilde{p}}_x}+\left\|\rho_{l}^{-1} m_{l} \otimes \mathcal{N}^i_{l}-\rho_q^{-1} m_q \otimes \mathcal{N}^i_q\right\|_{L_t^1 L^{\widetilde{p}}_x}\right)\\
&\quad+\sum_{i=1}^3\left(\left\|\rho_q^{-1} \mathcal{N}^i_q \otimes m_q-\left(\rho_q^{-1} \mathcal{N}^i_q \otimes m_q\right) *_x \phi_{l} *_t \varphi_{l}\right\|_{L_t^1 L^{\widetilde{p}}_x}\right.\\
&\qquad+\left.\left\|\rho_q^{-1} m_q \otimes \mathcal{N}^i_q-\left(\rho_q^{-1} m_q \otimes \mathcal{N}^i_q\right) *_x \phi_{l} *_t \varphi_{l}\right\|_{L_t^1 L^{\widetilde{p}}_x}\right)\\
&=:K_{{\rm com}}.
\end{aligned}
\end{equation}
Denote 
$$
\mathcal{H}_l:=\mathcal{N}^i_l \text{ or }m_l,\quad\mathcal{H}_q:=\mathcal{N}^i_q \text{ or }m_q.
$$
For simplicity, let us abuse the notation
\begin{align*}
    &\mathcal{H}_l \otimes \mathcal{H}_l\in\left\{\mathcal{N}_l^i\otimes\mathcal{N}_l^i, m_l\otimes m_l,\mathcal{N}^i_l\otimes m_l,m_l\otimes\mathcal{N}_l^i\right\},\\
    &\mathcal{H}_q \otimes \mathcal{H}_q\in\left\{\mathcal{N}_q^i\otimes\mathcal{N}_q^i,m_q\otimes m_q,\mathcal{N}^i_q\otimes m_q\text{ or }m_q\otimes\mathcal{N}_q^i\right\},
\end{align*}
then we divide $K_{\rm com}$ into $K_1$ and $K_2$ according to the form of each terms,
\begin{align*}
    &K_1=\sum\left\|\rho_{l}^{-1} \mathcal{H}_l \otimes \mathcal{H}_l-\rho_q^{-1} \mathcal{H}_q \otimes \mathcal{H}_q\right\|_{L_t^1 L^{\widetilde{p}}_x},\\
    &K_2=\sum\left\|\rho_q^{-1} \mathcal{H}_q \otimes \mathcal{H}_q-\left(\rho_q^{-1} \mathcal{H}_q \otimes \mathcal{H}_q\right) *_x \phi_{l} *_t \varphi_{l}\right\|_{L_t^1 L^{\widetilde{p}}_x}.
\end{align*}
By a direct calculation we have
\begin{equation}\label{eq4.52}
\begin{aligned}
K_{1} & \lesssim\sum\left(\left\|\rho_{l}^{-1} \mathcal{H}_{l} \otimes\left(\mathcal{H}_{l}-\mathcal{H}_q\right)\right\|_{L_t^1 L^{\widetilde{p}}_x}+\left\|\rho_{l}^{-1}\left(\mathcal{H}_{l}-\mathcal{H}_q\right) \otimes \mathcal{H}_q\right\|_{L_t^1 L^{\widetilde{p}}_x}\right)\\
&\quad+\sum\left\|\left(\rho_{l}^{-1}-\rho_q^{-1}\right) \mathcal{H}_q \otimes \mathcal{H}_q\right\|_{L_t^1 L^{\widetilde{p}}_x} \\
& \lesssim\sum\left\|\rho_{l}^{-1}\right\|_{C_{t,x}}\left\|\mathcal{H}_{l}-\mathcal{H}_q\right\|_{C_{t,x}}\left(\left\|\mathcal{H}_{l}\right\|_{C_{t,x}}+\left\|\mathcal{H}_q\right\|_{C_{t,x}}\right)\\
&\quad+\sum\left\|\rho_{l}^{-1}-\rho_q^{-1}\right\|_{C_{t,x}}\left\|\mathcal{H}_q\right\|_{C_{t,x}}^2\\
&\lesssim l\lambda_q^7\left(l\lambda_q^7+\lambda_q^7\right)+l\lambda_q^{\frac{\varepsilon}{4}}\lambda_q^{6}\lesssim l^{\frac{1}{2}}.
\end{aligned}
\end{equation}
On the other hand, using \eqref{eq2.6}, \eqref{eq2.7} and interpolation we obtain
\begin{equation}\label{4.53s}
\begin{aligned}
K_{2} & \lesssim\sum\left\|\rho_q^{-1} \mathcal{H}_q \otimes \mathcal{H}_q-\left(\rho_q^{-1} \mathcal{H}_q \otimes \mathcal{H}_q\right) *_x \phi_{l} *_t \varphi_{l}\right\|_{C_{t, x}}\\
& \lesssim l\sum\left\|\rho_q^{-1} \mathcal{H}_q \otimes \mathcal{H}_q\right\|_{C_{t, x}^1}\\
& \lesssim l\sum\left\|\rho_q^{-1}\right\|_{C_{t, x}^1}\left\|\mathcal{H}_q\right\|_{C_{t, x}^1}^2\\
& \lesssim l \lambda_q^{\frac{\varepsilon}{4}} \lambda_q^{14}\lesssim l^{\frac{1}{2}}.
\end{aligned}
\end{equation}
Thus, plugging \eqref{eq4.52} and \eqref{4.53s} into \eqref{4.51s}, we derive that
\begin{align*}
    \left\|R^m_{{\rm com},4}\right\|_{L_t^1 L^{\widetilde{p}}_x}+\left\|R^m_{{\rm com},5}\right\|_{L_t^1 L^{\widetilde{p}}_x} +\sum\left\|M^i_{{\rm com}}\right\|_{L_t^1 L^{\widetilde{p}}_x}\lesssim l^{\frac{1}{2}},
\end{align*}
which together with \eqref{4.48s}--\eqref{4.50s} gives that
\begin{equation}\label{4.54s}
    \left\|R^m_{{\rm com}}\right\|_{L_t^1 L^{\widetilde{p}}_x} +\sum_{i=1}^3\left\|M^i_{{\rm com}}\right\|_{L_t^1 L^{\widetilde{p}}_x}\lesssim l^{\frac{1}{3}}.
\end{equation}

Finally, combining \eqref{4.36s}, \eqref{4.40s}, \eqref{4.41s}, \eqref{4.44s} and \eqref{4.54s} altogether and using \eqref{eq2.3s} we conclude that
 \begin{align*}    &\quad\|R^m_{q+1}\|_{L^1_tL^{\widetilde{p}}_x}+\sum_{i=1}^3\|M^i_{q+1}\|_{L^1_tL^{\widetilde{p}}_x}\\
       &\lesssim \|R^m_{{\rm lin}}\|_{L^1_tL^{\widetilde{p}}_x}+\|R^m_{{\rm osc}}\|_{L^1_tL^{\widetilde{p}}_x}+\|R^m_{{\rm cor}}\|_{L^1_tL^{\widetilde{p}}_x}+\|R^m_{{\rm pre}}\|_{L^1_tL^{\widetilde{p}}_x}+\|R^m_{{\rm com}}\|_{L^1_tL^{\widetilde{p}}_x}\\
       &\quad+\sum_{i=1}^3\left(\|M^i_{{\rm lin}}\|_{L^1_tL^{\widetilde{p}}_x}+\|M^i_{{\rm osc}}\|_{L^1_tL^{\widetilde{p}}_x}+\|M^i_{{\rm cor}}\|_{L^1_tL^{\widetilde{p}}_x}+\|M^i_{{\rm com}}\|_{L^1_tL^{\widetilde{p}}_x}\right)\\
       &\lesssim l^{-74}\lambda_{q+1}^{-2\varepsilon}+l^{-52}\lambda_{q+1}^{-3\varepsilon}+l^{-40}\lambda_{q+1}^{-5\varepsilon}+l^{-23}\lambda_{q+1}^{-5\varepsilon}+l^{\frac{1}{3}}\\
       &\lesssim\delta_{q+2},
\end{align*}
which means that the inductive estimate \eqref{eq2.9} is verified at level $q+1$.
\section{Proof of main results}\label{sec5}
We are now in the stage to prove the main results, namely Theorems \ref{thm-nonuniquenes-1}, \ref{thm-nonuniqueness-2} and \ref{thm-main-iteration}. To begin, we prove the main iteration result Theorem \ref{thm-main-iteration}.
\subsection{Proof of Theorem \ref{thm-main-iteration}}
Since the iterative estimates \eqref{eq2.5}--\eqref{eq2.13} have been verified in the previous sections, we only need to check the inductive estimate \eqref{eq2.15} for the temporal support.

First, by the definition of the perturbations and $T_q$, we have
\begin{equation}\label{eq5.1}
    \bigcup_{i=1}^3{\rm supp}_t (w_{q+1},f^i_{q+1})\subseteq\bigcup_{k\in\Lambda_m\cup\Lambda_\mathcal{N}}{\rm supp}_t(a_{(k)})\subseteq N_{3l}([T_q,T]),
\end{equation}
and
\begin{equation}\label{eq5.2}
    {\rm supp}_tz_{q+1}\subseteq N_{2l}([T_q,T]).
\end{equation}
Furthermore, by \eqref{3.49s} and \eqref{3.51s},
\begin{equation}\label{eq5.3}
    \begin{aligned}
        &\quad{\rm supp}_t(\nabla\rho_{q+1},m_{q+1},\mathcal{N}_{q+1}-\overline{\mathcal{N}})\\
        &\subseteq{\rm supp}_t(\nabla\rho_l,m_l,\mathcal{N}_l-\overline{\mathcal{N}})\quad\cup\left(\bigcup_{i=1}^3{\rm supp}_t(\nabla z_{q+1},w_{q+1},f^i_{q+1})\right)\\
        &\subseteq N_{3l}([T_q,T]).
    \end{aligned}
\end{equation}

Next, for the Reynolds stresses and elastic stress fluctuation, by \eqref{eq4.5} and \eqref{eq4.12}, we derive
\begin{equation}\label{eq5.4}
    \begin{aligned}
        &\quad{\rm supp}_t(R^m_{q+1},\mathscr{M}_{q+1})\subseteq\bigcup_{i=1}^3{\rm supp}_t(R^m_{q+1},M^i_{q+1})\\
        &\subseteq\bigcup_{k\in\Lambda_m\cup\Lambda_{\mathcal{N}}}{\rm supp}_t(a_{(k)},m_l,\mathcal{N}_l-\overline{\mathcal{N}},\nabla\rho_{q+1},\nabla\rho_l)\cup N_l\left(\bigcup_{i=1}^3{\rm supp}_t(R^m_q,M^i_q)\right)\\
        &\subseteq\bigcup_{k\in\Lambda_m\cup\Lambda_\mathcal{N}}{\rm supp}_t(a_{(k)},m_l,\mathcal{N}_l-\overline{\mathcal{N}},\nabla z_{q+1},\nabla\rho_l)\cup N_l\left(\bigcup_{i=1}^3{\rm supp}_t(R^m_q,M^i_q)\right)\\
        &\subseteq N_{3l}([T_q,T]).
    \end{aligned}
\end{equation}
Therefore, combining \eqref{eq5.3} and \eqref{eq5.4} altogether we have
\begin{equation}\label{5.5s}
   {\rm supp}_t(\nabla\rho_{q+1},m_{q+1},\mathcal{N}_{q+1}-\overline{\mathcal{N}},R^m_{q+1},\mathscr{M}_{q+1})\subseteq N_{3l}([T_q,T]),
\end{equation}
which along with the fact that $3l\ll\delta_{q+2}^{\frac{1}{2}}$ give that \eqref{eq2.15}. Now the proof of Theorem \ref{thm-main-iteration} is complete. \hfill$\square$
\subsection{Proof of Theorem \ref{thm-nonuniquenes-1}}
We prove the statements $(i)$-$(v)$ in Theorem \ref{thm-nonuniquenes-1} below.

$(i)$. Let $\rho_0=\widetilde{\rho}$, $m_0=\widetilde{m}$, $\mathcal{N}_0=\widetilde{\mathcal{N}}$ and
\begin{align}
    &\begin{aligned}
        R^m_0
        &:=\mathcal{R}^m\left(\partial_t\widetilde{m}+\nu^s(-\Delta)^{\alpha}(\widetilde{\rho}^{-1}\widetilde{m})-(\nu^b+\frac{1}{3}\nu^s)\nabla{\rm div}(\widetilde{\rho}^{-1}\widetilde{m})+\nabla P(\widetilde{\rho})\right)\\
        &\quad+\widetilde{\rho}^{-1}\left(\widetilde{m}\otimes\widetilde{m}-\widetilde{\mathcal{N}}\widetilde{\mathcal{N}}^{\rm T}\right),
    \end{aligned}\label{eq5.6}\\
    & M^i_0:=\mathcal{R}^n\left(
\partial_t\widetilde{\mathcal{N}}^i
\right)+\widetilde{\rho}^{-1}\left(\widetilde{\mathcal{N}}^i\otimes\widetilde{m}-\widetilde{m}\otimes\widetilde{\mathcal{N}}^i
\right),
\quad i=1,2,3,\label{eq5.7}\\
&\left(\mathscr{M}_0\right)_{ijk}:=\left(M^i_0\right)_{jk}.\notag
\end{align}
Since $(\widetilde{\rho},\widetilde{m})$ is a smooth solution to the transport equation $\eqref{1.3s}_1$, $\widetilde{\mathcal{N}}^i$ are all divergence-free, we have 
\begin{align*}
    &\begin{aligned}
    {\rm div}R^m_0=&\partial_t\widetilde{m}+\nu^s(-\Delta)^{\alpha}(\widetilde{\rho}^{-1}\widetilde{m})-(\nu^b+\frac{1}{3}\nu^s)\nabla{\rm div}(\widetilde{\rho}^{-1}\widetilde{m})+\nabla P(\widetilde{\rho})\\
    &\quad+{\rm div}\left(\widetilde{\rho}^{-1}\left(\widetilde{m}\otimes\widetilde{m}-\widetilde{\mathcal{N}}\widetilde{\mathcal{N}}^{\rm T}\right)\right),
\end{aligned}\\
&\begin{aligned}
    {\rm div}M^i_0=&\partial_t\widetilde{\mathcal{N}}^i+{\rm div}\left(\widetilde{\rho}^{-1}\left(\widetilde{\mathcal{N}}^i\otimes\widetilde{m}-\widetilde{m}\otimes\widetilde{\mathcal{N}}^i
\right)\right),\quad i=1,2,3.
\end{aligned}
\end{align*}
Then we derive that $(\rho_0,m_0,\mathcal{N}_0,R^m_0,\mathscr{M}_0)$ is a relaxed solution to \eqref{eq2.1}. Moreover, for $a$ large enough, the inductive estimates \eqref{eq2.5}--\eqref{eq2.9} are all satisfied at level $q=0$. Thus, by Theorem \ref{thm-main-iteration}, there exists a sequence of relaxed solutions $\{(\rho_q,m_q,\mathcal{N}_q,R^m_q,\mathscr{M}_q)\}_{q\geq0}$ to \eqref{eq2.1}, which satisfy \eqref{eq2.5}--\eqref{eq2.15} for any $q\geq0$.

First, concerning the density, using \eqref{eq2.11} we have that for $a$ large enough,
\begin{equation}\label{eq5.8}
    \sum_{q \geq 0}\left\|\rho_{q+1}-\rho_q\right\|_{C_t C_x^1} \lesssim \sum_{q \geq 0} \delta_{q+2}^{\frac{1}{2}}=
a^{\frac{3\beta}{2}b}
\sum_{q\geq0}a^{-\beta b^{q+2}}
\leq \varepsilon_*.
\end{equation}
It means that $\{\rho_q\}_{q\geq0}$ is a Cauchy sequence in the space $C_tC^1_x$, that is, there exists $\rho\in C_tC^1_x$ such that
\begin{equation}\label{eq5.9}
    \lim_{q\rightarrow+\infty}\rho_q=\rho\text{ in }C_tC^1_x,
\end{equation}
which along with \eqref{eq2.5} yields that
\begin{equation}\label{eq5.10}
    \lim_{q\rightarrow+\infty}\rho^{-1}_q=\rho^{-1}\text{ in }C_{t,x}.
\end{equation}

Next, for the momentum and elastic field, using interpolation inequality, \eqref{eq2.4}, \eqref{eq2.7} and \eqref{eq2.12} we derive that for any $0<\beta^{\prime}<\beta/(5+\beta)$,
\begin{equation}\label{eq5.11}
\begin{aligned}
\sum_{q \geq 0}\left\|m_{q+1}-m_q\right\|_{H_{t,x}^{\beta^{\prime}}} & \leq \sum_{q \geq 0}\left\|m_{q+1}-m_q\right\|_{L_{t,x}^2}^{1-\beta^{\prime}}\left\|m_{q+1}-m_q\right\|_{C_{t, x}^1}^{\beta^{\prime}} \\
& \lesssim \sum_{q \geq 0} \delta_{q+1}^{\frac{1-\beta^{\prime}}{2}} \lambda_{q+1}^{5 \beta^{\prime}} \\
& \lesssim \delta_1^{\frac{1-\beta^{\prime}}{2}} \lambda_1^{5 \beta^{\prime}}+\sum_{q \geq 1} \lambda_{q+1}^{-\beta\left(1-\beta^{\prime}\right)+5 \beta^{\prime}}<\infty,
\end{aligned}
\end{equation}
and
\begin{equation}\label{eq5.12}
\begin{aligned}
\sum_{q \geq 0}\left\|\mathcal{N}_{q+1}-\mathcal{N}_q\right\|_{H_{t,x}^{\beta^{\prime}}} & \leq \sum_{i=1}^3\sum_{q \geq 0}\left\|\mathcal{N}^i_{q+1}-\mathcal{N}^i_q\right\|_{L_{t,x}^2}^{1-\beta^{\prime}}\left\|\mathcal{N}^i_{q+1}-\mathcal{N}^i_q\right\|_{C_{t, x}^1}^{\beta^{\prime}} \\
& \lesssim \sum_{q \geq 0} \delta_{q+1}^{\frac{1-\beta^{\prime}}{2}} \lambda_{q+1}^{5 \beta^{\prime}} \\
& \lesssim \delta_1^{\frac{1-\beta^{\prime}}{2}} \lambda_1^{5 \beta^{\prime}}+\sum_{q \geq 1} \lambda_{q+1}^{-\beta\left(1-\beta^{\prime}\right)+5 \beta^{\prime}}<\infty,
\end{aligned}
\end{equation}
where we use $-\beta(1-\beta^{\prime})+5\beta^{\prime}<0$ in the last inequalities of both \eqref{eq5.11} and \eqref{eq5.12}. Thus, $\{m_q\}_{q\geq0}$ and $\{\mathcal{N}_q\}_{q\geq0}$ are Cauchy sequences in the space $H_{t,x}^{\beta^{\prime}}$, that is there exist $m,\mathcal{N}\in H_{t,x}^{\beta^{\prime}}$ such that
\begin{equation}\label{5.13s}
    \lim_{q\rightarrow+\infty}m_q=m\quad\text{and }\lim_{q\rightarrow+\infty}\mathcal{N}_q=\mathcal{N}\text{ in }H_{t,x}^{\beta^{\prime}}.
\end{equation}
Furthermore, by \eqref{eq5.10} and \eqref{5.13s}, we derive that
\begin{equation}\label{5.14s}
\left\{
    \begin{aligned}
        &\lim_{q\rightarrow+\infty}\rho^{-1}_qm_q=\rho^{-1}m\text{ in }L^2_{t,x},\\
        &\lim_{q\rightarrow+\infty}\rho^{-1}_q\left(m_q\otimes m_q-\mathcal{N}_q\mathcal{N}_q^{{\rm T}}\right)=\rho^{-1}\left(m\otimes m-\mathcal{N}\mathcal{N}^{{\rm T}}\right)\text{ in }L^1_{t,x},\\
        &\lim_{q\rightarrow+\infty}\rho^{-1}_q\left(\mathcal{N}^i_q\otimes m_q-m_q\otimes \mathcal{N}^i_q\right)=\rho^{-1}\left(\mathcal{N}^i\otimes m-m\otimes \mathcal{N}^i\right)\text{ in }L^1_{t,x},\quad i=1,2,3.
    \end{aligned}
    \right.
\end{equation}

On the other hand, by mean-valued theorem, \eqref{eq5.9} and the continuity of $P$, $P^{\prime}$ and $P^{\prime\prime}$,
\begin{equation}\label{5.15s}
    \lim_{q\rightarrow+\infty}P(\rho_q)=P(\rho)\text{ in }C_{t,x}.
\end{equation}

Thus, taking into account \eqref{eq2.9},
\begin{equation}\label{5.16s}
    \lim_{q\rightarrow+\infty}R^m_q=0\quad\text{and }\lim_{q\rightarrow+\infty}M^i_q=0\text{ in }L^1_{t,x},\quad i=1,2,3.
\end{equation}
We thus conclude that $(\rho,m,\mathcal{N})$ is a weak solution to \eqref{1.3s}.

$(ii)$. We need only to check
\begin{equation}\label{5.17s}
    m,\mathcal{N}\in L^{\gamma}_tW^{s,p}_x\cap C_tH^{-\frac{3}{2}}_x.
\end{equation}
Using \eqref{eq2.13}, by a similar technique of \eqref{eq5.11} and \eqref{eq5.12} we get
\begin{equation}\label{5.18s}
    \begin{aligned}
        &\quad\sum_{q\geq0}\left\|m_{q+1}-m_q\right\|_{L^{\gamma}_tW^{s,p}_x\cap C_tH^{-\frac{3}{2}}_x}+\sum_{q\geq0}\left\|\mathcal{N}_{q+1}-\mathcal{N}_q\right\|_{L^{\gamma}_tW^{s,p}_x\cap C_tH^{-\frac{3}{2}}_x}\\
        &\lesssim\sum_{q\geq0}\left(\left\|m_{l}-m_q\right\|_{C_tC^2_x}+\left\|w_{q+1}\right\|_{L^{\gamma}_tW^{s,p}_x}+\sum_{i=1}^3\left(\left\|\mathcal{N}^i_{l}-\mathcal{N}^i_q\right\|_{L^{\gamma}_tW^{s,p}_x}+\left\|f^i_{q+1}\right\|_{L^{\gamma}_tW^{s,p}_x}\right)\right)\\
        &\quad+\sum_{q\geq0}\left(\left\|m_l-m_q\right\|_{C_{t,x}}+\left\|w_{q+1}^{(p)}+w_{q+1}^{(c)}\right\|_{C_tH^{-\frac{3}{2}}_x}+\left\|w_{q+1}^{(o)}\right\|_{C_tL^2_x}\right)\\
        &\quad+\sum_{q\geq0}\sum_{i=1}^3\left(\left\|\mathcal{N}^i_l-\mathcal{N}^i_q\right\|_{C_{t,x}}+\left\|f_{q+1}^{i,(p)}+f_{q+1}^{i,(c)}\right\|_{C_tH^{-\frac{3}{2}}_x}+\left\|f_{q+1}^{i,(o)}\right\|_{C_tL^2_x}\right)\\
        &\lesssim\sum_{q\geq0}\left(l\lambda_q^{11}+l\lambda_q^7+l^{-2}\lambda_{q+1}^sr_{\perp}^{\frac{1}{p}-\frac{1}{2}}\tau^{\frac{1}{2}-\frac{1}{\gamma}}+l^{-32}\sigma^{-1}++l^{-14}\sigma^{-1}\right)\\
        &\quad+\sum_{q\geq0}\sum_{k\in\Lambda_{m^1}\cup\Lambda_{\mathcal{N}}}\left\|{\rm curl}\left(a_{(k)}g_{(k)}W^c_{(k)}\right)\right\|_{C_tH^{-\frac{1}{2}}_x}\\
        &\quad+\sum_{i=1}^3\sum_{k\in\Lambda_{m^2_i}\cup\Lambda_{\mathcal{N}}}\left\|{\rm curl}\left(a_{(k)}g_{(k)}D^{i,c}_{(k)}\right)\right\|_{C_tH^{-\frac{1}{2}}_x}\\
        &\lesssim \sum_{q\geq0}\left(\lambda_q^{-11}+\lambda_{q+1}^{\left(s+2\alpha-1-\frac{2\alpha}{\gamma}-\frac{2\alpha-2}{p}+\varepsilon\left(6-\frac{10}{p}\right)\right)}+\lambda_{q+1}^{-4\varepsilon}+l^{11}\lambda_{q+1}^{\alpha-\frac{3}{2}}\right).
    \end{aligned}
\end{equation}
Taking into account \eqref{eq2.2s} we have
\begin{align*}
    \alpha-\frac{3}{2},\quad s+2\alpha-1-\frac{2\alpha}{\gamma}-\frac{2\alpha-2}{p}+\varepsilon\left(6-\frac{10}{p}\right)<-90\varepsilon,
\end{align*}
which yields that
\begin{equation}\label{5.19s}
    \sum_{q\geq0}\left\|m_{q+1}-m_q\right\|_{L^{\gamma}_tW^{s,p}_x\cap C_tH^{-\frac{3}{2}}_x}+\sum_{q\geq0}\left\|\mathcal{N}_{q+1}-\mathcal{N}_q\right\|_{L^{\gamma}_tW^{s,p}_x\cap C_tH^{-\frac{3}{2}}_x}\leq\sum_{q\geq0}\delta_{q+2}^{\frac{1}{2}}<\infty,
\end{equation}
where we also use \eqref{eq2.3s} and \eqref{eq2.4}. Therefore, $\{m_q\}_{q\geq0}$ and $\{\mathcal{N}_q\}_{q\geq0}$ are both Cauchy sequences in $L^{\gamma}_tW^{s,p}_x\cap C_tH^{-\frac{3}{2}}_x$. Moreover, we have the following weak convergence:
\begin{align}\label{eq5.20}
\left\{
    \begin{aligned}
    &m(t)\rightharpoonup\omega_0^m\text{ weakly in }L^2(\mathbb{T}^3)\text{ as }t\rightarrow0,\\
     &\mathcal{N}(t)\rightharpoonup\Omega_0^{\mathcal{N}}\text{ weakly in }L^2(\mathbb{T}^3)\text{ as }t\rightarrow0.\\
\end{aligned}
\right.
\end{align}
Hence, by the uniqueness of weak limits, we obtain \eqref{5.17s} which together with \eqref{5.13s} leads the regularity statement $(ii)$.

$(iii)$. By \eqref{eq2.10},
\begin{align*}
     \int_{\mathbb{T}^3}\rho_q(t,x){\rm d}x=\int_{\mathbb{T}^3}\rho_0(t,x){\rm d}x=\int_{\mathbb{T}^3}\widetilde{\rho}(t,x){\rm d}x,
\end{align*}
for any $q\in\mathbb{N}$ and $t\in[0,T]$. Then using \eqref{eq5.9} to pass to the limit $q\rightarrow+\infty$, we get
\begin{equation}\label{eq5.21}
    \int_{\mathbb{T}^3}\rho(t,x){\rm d}x=\int_{\mathbb{T}^3}\widetilde{\rho}(t,x){\rm d}x,\quad\forall t\in[0,T],
\end{equation}
which means that the mass preservation statement $(iii)$ is valid.

$(iv)$. The small deviation between $\rho$ and $\widetilde{\rho}$ in $C_tC^1_x$ has been implied by \eqref{eq5.8}. Next, using \eqref{eq2.13}, \eqref{5.19s} and estimating as in \eqref{eq5.11} and \eqref{eq5.12} we obtain
\begin{equation}\label{eq5.22}
\begin{aligned}
\|m-\widetilde{m}\|_{L_t^{1}L^2_x}
\leq 2 \sum_{q \geq 0} \delta_{q+2}^{\frac{1}{2}} \leq \varepsilon_*,
\end{aligned}
\end{equation}
and
\begin{equation}\label{eq5.23}
\begin{aligned}
&\|\mathcal{N}-\widetilde{\mathcal{N}}\|_{L_t^{1}L^2_x}
\leq  2 \sum_{q \geq 0} \delta_{q+2}^{\frac{1}{2}} \leq \varepsilon_*,
\end{aligned}
\end{equation}
for $a$ sufficiently large. Thus, we have proved the small deviation between $(\rho,m,\mathcal{N})$ and $(\widetilde{\rho},\widetilde{m},\widetilde{\mathcal{N}})$, which verified the statement $(iv)$.

$(v)$. Finally, for the temporal supports, by \eqref{eq5.9} and \eqref{5.13s}, we obtain
\begin{equation}\label{5.24s}
    {\rm supp}_t(\nabla\rho,m,\mathcal{N}-\overline{\mathcal{N}})\subseteq N_{\sum_{q\geq0}\delta_{q+2}^{\frac{1}{2}}}([\widetilde{T},T])\subseteq N_{\varepsilon_*}([\widetilde{T},T]),
\end{equation}
for $a$ is large enough and the fact that $\sum_{q\geq0}\delta_{q+2}^{\frac{1}{2}}\leq\varepsilon_*$.

Therefore, we have finished the proof of Theorem \ref{thm-nonuniquenes-1}.
\hfill$\square$
\subsection{Proof of Theorem \ref{thm-nonuniqueness-2}}
 Without loss of generality, we choose $\widetilde{\rho}=1$ and  non-trivial divergence-free, mean-free vector fields $\widetilde{m}$ and tensor field $\widetilde{\mathcal{N}}$:
 \begin{equation}\label{5.25s}
     \begin{aligned}
         &\widetilde{m}:=\Psi(t)({\rm sin}x_3,0,0)^{{\rm T}},\\
         &\widetilde{\mathcal{N}}:=\Psi(t)
         \begin{pmatrix}
             0 & {\rm sin}x_2 & {\rm sin}x_3\\
             {\rm sin}x_1 & 0 & {\rm cos}x_3\\
             {\rm cos}x_1 & {\rm cos}x_2 & 0\\
         \end{pmatrix},
     \end{aligned}
 \end{equation}
 where $\Psi(t)$ is any cut-off function such that ${\rm supp}_t(\Psi(t))\subseteq[\frac{T}{4},\frac{3T}{4}]$, $\Psi\equiv1$ on $[\frac{5T}{16},\frac{11T}{16}]$ and $0\leq\Psi(t)\leq1$ on $[0,T]$.
 Next, for every $j\geq1$, let
 \begin{equation}\label{5.26s}
     \widetilde{m}_j:=\frac{j}{c_0}\widetilde{m}\quad\text{and}\quad\widetilde{\mathcal{N}}_j:=\frac{j}{c_1}\widetilde{\mathcal{N}}+\begin{pmatrix}
             1 & 0 & 0\\
             0 & 1 & 0\\
             0 & 0 & 1\\
         \end{pmatrix},
 \end{equation}
 where $j\in\mathbb{N}_+$, $c_0:=\|\widetilde{m}\|_{L^1(0,T;L^2_x)}$ and $c_1:=\|\widetilde{\mathcal{N}}\|_{L^1(0,T;L^2_x)}$. Thus $(\widetilde{\rho},\widetilde{m}_j)$ are solutions to the transport equation \eqref{1.6-s} and $\widetilde{\mathcal{N}}_j$ are divergence-free.

 Let $0<\varepsilon_*<{\rm min}\{\frac{1}{8},\frac{T}{8}\}$. Then for every $j\geq1$, Theorem \ref{thm-nonuniquenes-1} gives weak solutions to \eqref{1.3s} such that
 \begin{equation}\label{5.27s}
 \left\{
     \begin{aligned}
         &\|m_j-\widetilde{m}_j\|_{L^1(0,T;L^2_x)}\leq\varepsilon_*,\\
         &\|\mathcal{N}_j-\widetilde{\mathcal{N}}_j\|_{L^1(0,T;L^2_x)}\leq\varepsilon_*,\\
     \end{aligned}
     \right.
 \end{equation}
 and
 \begin{align}
     &\int_{\mathbb{T}^3}\rho_j(t,x){\rm d}x=\int_{\mathbb{T}^3}\widetilde{\rho}(t,x){\rm d}x,\quad\forall t\in[0,T],\label{5.28s}\\
     &{\rm supp}_t(\nabla\rho_j,m_j,\mathcal{N}_j-\overline{\mathcal{N}})\subseteq[\frac{T}{8},T],\label{5.29s}
 \end{align}
 where
 \begin{align*}
     \overline{\mathcal{N}}=\begin{pmatrix}
             1 & 0 & 0\\
             0 & 1 & 0\\
             0 & 0 & 1\\
         \end{pmatrix}.
 \end{align*}
 
 Thus, by \eqref{5.29s}, $\rho_j$ is independent of the spatial variable for $t\in[0,\frac{T}{8}]$, that is, we derive that
\begin{align*}
    \rho_j(0)=\fint_{\mathbb{T}^3}\rho_j(0,x){\rm d}x=\fint_{\mathbb{T}^3}\widetilde{\rho}_j(0,x){\rm d}x=1.
\end{align*}
 It means that $(\rho_j,m_j,\mathcal{N}_j)$ are  weak solutions to \eqref{1.3s} with the same initial datum
 \begin{align*}
     (\rho_j(0),m_j(0),\mathcal{N}_j(0))=\left(1,(0,0,0)^T,\begin{pmatrix}
             1 & 0 & 0\\
             0 & 1 & 0\\
             0 & 0 & 1\\
         \end{pmatrix}\right).
 \end{align*}

Next, using \eqref{5.26s} and \eqref{5.27s}, we derive that for $j\neq k$,
\begin{align*}
    \|m_j-m_k\|_{L^1(0,T;L^2_x)}
    &\geq\|\widetilde{m}_j-\widetilde{m}_k\|_{L^1(0,T;L^2_x)}-\|m_j-\widetilde{m}_j\|_{L^1(0,T;L^2_x)}-\|m_k-\widetilde{m}_k\|_{L^1(0,T;L^2_x)}\\
    &\geq\frac{|j-k|}{c_0}\|\widetilde{m}\|_{L^1(0,T;L^2_x)}-2\varepsilon_*\\
    &=|j-k|-2\varepsilon_*>\frac{1}{2},
\end{align*}
and
\begin{align*}
    \|\mathcal{N}_j-\mathcal{N}_k\|_{L^1(0,T;L^2_x)}
    &\geq\|\widetilde{\mathcal{N}}_j-\widetilde{\mathcal{N}}_k\|_{L^1(0,T;L^2_x)}-\|\mathcal{N}_j-\widetilde{\mathcal{N}}_j\|_{L^1(0,T;L^2_x)}-\|\mathcal{N}_k-\widetilde{\mathcal{N}}_k\|_{L^1(0,T;L^2_x)}\\
    &\geq\frac{|j-k|}{c_1}\|\widetilde{\mathcal{N}}\|_{L^1(0,T;L^2_x)}-2\varepsilon_*\\
    &=|j-k|-2\varepsilon_*>\frac{1}{2},
\end{align*}
which means that $m_j\neq m_k$ and $\mathcal{N}_j\neq\mathcal{N}_k$ on $[0,T]$ for every $j\neq k$.

Therefore, the proof of Theorem \ref{thm-nonuniqueness-2} is complete.\hfill$\square$

\subsection{Proof of Theorem \ref{thm-vanishing limit}}Let $\{\phi_{\varepsilon}\}_{\varepsilon>0}$ and $\{\varphi_{\varepsilon}\}_{\varepsilon>0}$ be two families of standard compactly supported mollifiers on $\mathbb{T}^3$ and $\mathbb{T}$, respectively. For every $n\geq1$, set
\begin{equation}\label{eq5.30}
    \begin{aligned}     &\rho_n:=\left(\rho*_x\phi_{\lambda_n^{-1}}\right)*_t\varphi_{\lambda_n^{-1}},\quad P_n:=\left(P(\rho)*_x\phi_{\lambda_n^{-1}}\right)*_t\varphi_{\lambda_n^{-1}},\\     &m_n:=\left(m*_x\phi_{\lambda_n^{-1}}\right)*_t\varphi_{\lambda_n^{-1}},\quad \mathcal{N}^i_n:=\left(\mathcal{N}^i*_x\phi_{\lambda_n^{-1}}\right)*_t\varphi_{\lambda_n^{-1}},\quad i=1,2,3,
    \end{aligned}
\end{equation}
restricted to $[0,T]$, where $\lambda_n:=a^{b^n}$ and $\delta_n:=\lambda^{-2\beta}_n$.

Here, we choose $a$, $b$ and $\beta$ as in \eqref{eq2.2s} and \eqref{eq2.3s}, and additionally require that 
\begin{align*}
    b^2\beta<\min \left\{\frac{1}{100},\frac{\widetilde{\beta}}{4}\right\}.
\end{align*}

Since $(\rho,m,\mathcal{N})$ is a weak solution to the compressible fundamental elastodynamic system \eqref{1.4ss}, we infer that $(\rho_n,m_n,\mathcal{N}_n)$ satisfies the following relaxed systems
\begin{equation}\label{5.31ss}
    \left\{
    \begin{aligned}
        &\partial_t\rho_n+{\rm div}m_n=0,\\
        &\partial_tm_n+\kappa_n\nu^s(-\Delta)^{\alpha}(\rho^{-1}_nm_n)-\kappa_n(\nu^b+\frac{1}{3}\nu^s)\nabla{\rm div}(\rho^{-1}_nm_n)+\nabla P(\rho_n)\\
        &\qquad\qquad\qquad+{\rm div}\left(\rho_n^{-1}\left(m_n\otimes m_n-\sum_{i=1}^3\mathcal{N}_n^i\otimes\mathcal{N}_n^i\right)\right)={\rm div}R^m_n,\\
        &\partial_t\mathcal{N}^i_n+{\rm div}\left(\rho_n^{-1}\left(\mathcal{N}^i_n\otimes m_n-m_n\otimes \mathcal{N}^i_n\right)\right)={\rm div}M^i_n,\quad i=1,2,3,\\
        &{\rm div}\mathcal{N}^i_n=0,\quad i=1,2,3,\\
    \end{aligned}
    \right.
\end{equation}
where $\kappa_n=\lambda_n^{-2\alpha}$, $\mathcal{N}_n=\left(\mathcal{N}_n^1,\mathcal{N}_n^2,\mathcal{N}_n^3\right)$ and the Reynolds stress and elastic stress fluctuation
\begin{align}
    &   \begin{aligned}
        R^m_n:=
        &\kappa_n\nu^s\mathcal{R}^m(-\Delta)^{\alpha}\left(\rho_n^{-1}m_n\right)-\kappa_n(\nu^b+\frac{1}{3}\nu^s)\mathcal{R}^m\nabla{\rm div}\left(\rho_n^{-1}m_n\right)\\
        &+\mathcal{R}^m\nabla\left(P(\rho_n)-P_n\right)+\rho_n^{-1}\left(m_n\otimes m_n-\sum_{i=1}^3\mathcal{N}_n^i\otimes\mathcal{N}_n^i\right)\\
        &-\left(\left(\rho^{-1}\left(m\otimes m-\sum_{i=1}^3\mathcal{N}^i\otimes\mathcal{N}^i\right)\right)*_x\phi_{\lambda_n^{-1}}\right)*_t\varphi_{\lambda_n^{-1}},
    \end{aligned}\label{eq5.32}\\
    & \begin{aligned}
        M^i_n:=
        &\rho_n^{-1}\left(\mathcal{N}^i_n\otimes m_n-m_n\otimes \mathcal{N}^i_n\right)\\
    &-\left(\left(\rho^{-1}\left(\mathcal{N}^i\otimes m-m\otimes \mathcal{N}^i\right)\right)*_x\phi_{\lambda_n^{-1}}\right)*_t\varphi_{\lambda_n^{-1}},\quad i=1,2,3,
    \end{aligned}\label{eq5.33}\\
        &\begin{aligned}
        \left(\mathscr{M}_n\right)_{ijk}:=&\left(M^i_n\right)_{jk}.
    \end{aligned}\notag
\end{align}
\begin{clm}
    For $a$ sufficiently large, $(\rho_n,m_n,\mathcal{N}_n,R^m_n,\mathscr{M}_n)$ satisfy the inductive estimates \eqref{eq2.5}--\eqref{eq2.9} at level $q=n+1$.
\end{clm}

To this end, let us start with the most delicate $L^1_{t,x}$-decay estimates \eqref{eq2.9} of $R^m_n$ and $M^i_n$, 
\begin{equation}\label{5.34ss}
    \begin{aligned}
        \|R^m_n\|_{L^1_{t,x}}
        &\lesssim\lambda_n^{-2\alpha}\||\nabla|^{2\alpha-1}(\rho_n^{-1}m_n)\|_{L^1_{t,x}}+\lambda_n^{-2\alpha}\|{\rm div}(\rho_n^{-1}m_n)\|_{L^1_{t,x}}+\|P(\rho_n)-P_n\|_{L^1_{t,x}}\\
        &\quad+\|\rho_n^{-1}m_n\otimes m_n-(\rho^{-1}m\otimes  m)*_x\phi_{\lambda_n^{-1}}*_t\varphi_{\lambda_n^{-1}}\|_{L^1_{t,x}}\\
        &\quad+\left\|\rho_n^{-1}\sum_{i=1}^3\mathcal{N}^i_n\otimes\mathcal{N}_n^i-\left(\rho^{-1}\sum_{i=1}^3\mathcal{N}^i\otimes\mathcal{N}^i\right)*_x\phi_{\lambda_n^{-1}}*_t\varphi_{\lambda_n^{-1}}\right\|_{L^1_{t,x}}\\
        &=:K_1+K_2+K_3+K_4+K_5,
    \end{aligned}
\end{equation}
and for $i=1,2,3$,
\begin{equation}\label{5.35ss}
        \|M^i_n\|_{L^1_{t,x}}
\lesssim\|\rho_n^{-1}\left(\mathcal{N}^i_n\otimes m_n-m_n\otimes \mathcal{N}^i_n\right)-\left(\rho^{-1}\left(\mathcal{N}^i\otimes m-m\otimes \mathcal{N}^i\right)\right)*_x\phi_{\lambda_n^{-1}}*_t\varphi_{\lambda_n^{-1}}\|_{L^1_{t,x}}.
\end{equation}
By the standard mollification estimates we have
\begin{equation}\label{5.36ss}
    0<c_1\leq\rho_n\leq c_2.
\end{equation}

Let
\begin{equation}\label{eq5.37}
    \widetilde{M}:={\rm max}\left\{\|\rho\|_{C^{\widetilde{\beta}}_{t,x}},\|m\|_{H^{\widetilde{\beta}}_{t,x}},\|\mathcal{N}\|_{H^{\widetilde{\beta}}_{t,x}}\right\}.
\end{equation}
We start with the viscous errors $K_1$ and $K_2$, estimating as in \eqref{4.31s}, \eqref{4.32s} and \eqref{4.34s}, we obtain
\begin{equation}\label{eq5.38}
    \begin{aligned}
        K_1       &\lesssim\lambda_n^{-2\alpha}\left\||\nabla|^{2\alpha-1}\left(\rho_n^{-1}m_n\right)\right\|_{L^1_tL^2_x}\\
        &\lesssim\lambda_n^{-2\alpha}\|\rho^{-1}_nm_n\|_{L^1_tL^2_x}^{1-\frac{2\alpha-1}{2}}\|\rho^{-1}_nm_n\|_{L^1_tH^2_x}^{\frac{2\alpha-1}{2}}\\
        &\lesssim\lambda_n^{-2\alpha}\widetilde{M}^{1-\frac{2\alpha-1}{2}}\left(\sum_{N_1+N_2=2}\|\rho_n^{-1}\|_{C_tC^{N_1}_x}\|\nabla^{N_2}m_n\|_{L^1_tL^2_x}\right)^{\frac{2\alpha-1}{2}}\\
        &\lesssim\widetilde{M}\lambda_n^{-1},
    \end{aligned}
\end{equation}
and, similarly,
\begin{equation}\label{eq5.39}
    K_2\lesssim\widetilde{M}\lambda_n^{-2\alpha+1}.
\end{equation}
Then we consider the pressure error $K_3$, by mean-value theorem, it follows that
\begin{equation}\label{eq5.40}
        \begin{aligned}
            K_3
        &\lesssim \|P^{\prime}(\zeta)\|_{C_{t,x}}\|\rho_n-\rho\|_{C_{t,x}}+\|P(\rho)-P_n\|_{C_{t,x}}\\
        &\lesssim\lambda_n^{-\widetilde{\beta}}\|\rho\|_{C_{t,x}^{\widetilde{\beta}}}
        \lesssim\widetilde{M}\lambda_n^{-\widetilde{\beta}}.
        \end{aligned}
\end{equation}
Now we concern the commutator errors $K_4$, $K_5$ and $M_n^i$. Denote $\mathcal{Q}=m\text{ or }\mathcal{N}^i$ for $i=1,2,3$. For simplicity, let us abuse the notation
\begin{align*}
    &\mathcal{Q} \otimes \mathcal{Q}\in\left\{\mathcal{N}^i\otimes\mathcal{N}^i, m\otimes m,\mathcal{N}^i\otimes m,m\otimes\mathcal{N}^i\right\},\\
    &\mathcal{Q}_n \otimes \mathcal{Q}_n\in\left\{\mathcal{N}_n^i\otimes\mathcal{N}_n^i,m_n\otimes m_n,\mathcal{N}^i_n\otimes m_n,m_n\otimes\mathcal{N}_n^i\right\}.
\end{align*}
Note that all the commutator errors have the same form as follows:
\begin{align*}
    \sum\left\|\rho_n^{-1}\mathcal{Q}_n\otimes\mathcal{Q}_n-\left(\rho^{-1}\mathcal{Q}\otimes\mathcal{Q}\right)*_x\phi_{\lambda_n^{-1}}*_t\varphi_{\lambda_n^{-1}}\right\|_{L^1_{t,x}},
\end{align*}
By a similar technique of \cite{LZZ-2022-jmpa} we have
\begin{align}
    &\left\|\frac{\mathcal{Q}(t,x)-\mathcal{Q}(t-s,x-y)}{\left(|s|+|y|\right)^{2+\widetilde{\beta}}}\right\|_{L^2_{t,x}L^2_{s,y}}\lesssim\left\|\mathcal{Q}\right\|_{H^{\widetilde{\beta}}_{t,x}},\label{eq5.41}\\
    &\left\|\mathcal{Q}-\mathcal{Q}_n\right\|_{L^2_{t,x}}\lesssim\lambda_n^{-\widetilde{\beta}}\left\|\mathcal{Q}\right\|_{H^{\widetilde{\beta}}_{t,x}},\label{eq5.42}\\
    &\left|\rho^{-1}(t,x)-\rho^{-1}(t-s,x-y)\right|_{L^{\infty}_{t,x}}\lesssim\left\|\rho^{-1}\right\|_{C_{t,x}}^2\left\|\rho\right\|_{C_{t,x}^{\widetilde{\beta}}}\left(|s|+|y|\right)^{\widetilde{\beta}}.\label{eq5.43}
\end{align}
Using \eqref{eq5.41} and \eqref{eq5.42} we obtain
\begin{align}
    &\quad\left\|\rho^{-1}\mathcal{Q}\otimes\mathcal{Q}-\left(\rho^{-1}\mathcal{Q}\otimes\mathcal{Q}\right)*_x\phi_{\lambda_n^{-1}}*_t\varphi_{\lambda_n^{-1}}\right\|_{L^1_{t,x}}\notag\\        
        &\lesssim\left\|\int_0^T\int_{\mathbb{T}^3}\rho^{-1}(t-s,x-y)\delta_{s,y}\mathcal{Q}(t,x)\otimes\delta_{s,y}\mathcal{Q}(t,x)\phi_{\lambda_n^{-1}}\varphi_{\lambda_n^{-1}}{\rm d}s{\rm d}y\right\|_{L^1_{t,x}}\notag\\ 
        &\quad+\left\|\int_0^T\int_{\mathbb{T}^3}\rho^{-1}(t-s,x-y)\delta_{s,y}\mathcal{Q}(t,x)\otimes\mathcal{Q}(t,x)\phi_{\lambda_n^{-1}}\varphi_{\lambda_n^{-1}}{\rm d}s{\rm d}y\right\|_{L^1_{t,x}}\notag\\
        &\quad+\left\|\int_0^T\int_{\mathbb{T}^3}\delta_{s,y}\rho^{-1}(t,x)\mathcal{Q}(t,x)\otimes\mathcal{Q}(t,x)\phi_{\lambda_n^{-1}}\varphi_{\lambda_n^{-1}}{\rm d}s{\rm d}y\right\|_{L^1_{t,x}}\notag\\
        &\lesssim\left\|\rho^{-1}\right\|_{C_{t,x}}\left\|\left(|s|+|y|\right)^{4+2\widetilde{\beta}}\phi_{\lambda_n^{-1}}\varphi_{\lambda_n^{-1}}\right\|_{L^{\infty}_{s,y}}\left\|\frac{\mathcal{Q}(t,x)-\mathcal{Q}(t-s,x-y)}{\left(|s|+|y|\right)^{2+\widetilde{\beta}}}\right\|_{L^2_{t,x}L^2_{s,y}}^2\notag\\
        &\quad+\left\|\rho^{-1}\right\|_{C_{t,x}}|\mathcal{Q}\|_{L^2_{t,x}}\left\|\left(|s|+|y|\right)^{2+\widetilde{\beta}}\phi_{\lambda_n^{-1}}\varphi_{\lambda_n^{-1}}\right\|_{L^{\infty}_{s,y}}\left\|\frac{\mathcal{Q}(t,x)-\mathcal{Q}(t-s,x-y)}{\left(|s|+|y|\right)^{2+\widetilde{\beta}}}\right\|_{L^2_{t,x}L^2_{s,y}}\notag\\
        &\quad+\lambda_n^{-\widetilde{\beta}}\left\|\rho^{-1}\right\|_{C_{t,x}}\|\rho\|_{C^{\widetilde{\beta}}_{t,x}}\|\mathcal{Q}\|_{L^2_{t,x}}^2\left\|\left(|s|+|y|\right)^{\widetilde{\beta}}\phi_{\lambda_n^{-1}}\varphi_{\lambda_n^{-1}}\right\|_{L^{\infty}_{s,y}}\notag\\
        &\lesssim\lambda_n^{-\widetilde{\beta}}\widetilde{M}^3,\label{eq5.44}
\end{align}
where $\delta_{s,y}f(t,x):=f(t,x)-f(t-s,x-y)$.

Moreover, we note that, by using \eqref{eq5.37}, \eqref{eq5.41}--\eqref{eq5.42} altogether,
\begin{equation}\label{eq5.45}
    \begin{aligned}
        &\quad K_4+K_5+\sum_{i=1}^3M^i_n\\
        &\lesssim\sum_{i=1}^3\sum\left\|\rho_n^{-1}\mathcal{Q}_n\otimes\mathcal{Q}_n-\left(\rho^{-1}\mathcal{Q}\otimes\mathcal{Q}\right)*_x\phi_{\lambda_n^{-1}}*_t\varphi_{\lambda_n^{-1}}\right\|_{L^1_{t,x}}\\
        &\lesssim  \sum\left\|\rho_n^{-1} \mathcal{Q}_n \otimes\left(\mathcal{Q}_n-\mathcal{Q}\right)\right\|_{L^1_{t,x}}+\sum\left\|\rho_n^{-1}\left(\mathcal{Q}_n-\mathcal{Q}\right) \otimes \mathcal{Q}\right\|_{L^1_{t,x}}\\
& \quad+\sum\left\|\rho^{-1} \mathcal{Q} \otimes \mathcal{Q}-\left(\rho^{-1} \mathcal{Q} \otimes \mathcal{Q}\right) *_x \phi_{\lambda_n^{-1}} *_t\varphi_{\lambda_n^{-1}}\right\|_{L^1_{t,x}} \\
        &\quad+\sum\left\|\left(\rho_n^{-1}-\rho^{-1}\right) \mathcal{Q} \otimes \mathcal{Q}\right\|_{L^1_{t,x}} \\
        &\lesssim\sum\left\|\rho_n^{-1}\right\|_{C_{t, x}}\left\|\mathcal{Q}_n-\mathcal{Q}\right\|_{L^2_{t, x}}\left(\left\|\mathcal{Q}_n\right\|_{L^2_{t, x}}+\|\mathcal{Q}\|_{L^2_{t, x}}\right)\\
&\quad+\sum\left\|\rho_n^{-1}-\rho^{-1}\right\|_{C_{t, x}}\|\mathcal{Q}\|_{L^2_{t, x}}^2+\lambda_n^{-\widetilde{\beta}}\widetilde{M}^3\\
        &\lesssim\lambda_n^{-\widetilde{\beta}}\widetilde{M}^3.
    \end{aligned}
\end{equation}

Thus, plugging \eqref{eq5.38}--\eqref{eq5.40} and \eqref{eq5.45} into \eqref{5.34ss} and \eqref{5.35ss}, we arrive at
\begin{align*}
\|R^m_n\|_{L^1{t,x}}+\sum_{i=1}^3\|M^i_n\|_{L^1_{t,x}}
    &\lesssim\lambda_n^{-1}\widetilde{M}+\lambda_n^{-2\alpha+1}\widetilde{M}+\lambda_n^{-\widetilde{\beta}}(\widetilde{M}+\widetilde{M}^3)\leq \delta_{n+2},
\end{align*}
which means that \eqref{eq2.9} is valid at level $n+1$.

Now we turn to the inductive estimates of \eqref{eq2.5}--\eqref{eq2.8}, by \eqref{5.36ss} we get
\begin{align*}
     c_1-\lambda_{n+1}^{-\beta} \leq \rho_n \leq c_2+\lambda_{n+1}^{-\beta}.
\end{align*}
Moreover, for $1 \leq N \leq 4$ and $M=0,1$, by a direct calculation, we have
\begin{align*}
    \left\|\partial_t^M \rho_n\right\|_{C_t C_x^N} \lesssim \lambda_n^{M+N}\|\rho\|_{C_{t, x}} \lesssim \lambda_n^{1+N} \leq \lambda_{n+1}^{\frac{\varepsilon}{4}}.
\end{align*}
Hence, \eqref{eq2.5} and \eqref{eq2.6} are verified at level $n+1$.

Moreover, by using the Sobolev embedding $H^3_{t,x} \hookrightarrow L_{t, x}^{\infty}$, we also see that
\begin{equation}\label{5.46-s}
    \begin{aligned}
\left\|m_n\right\|_{C_{t, x}^N}+\sum_{i=1}^3\left\|\mathcal{N}^i_n\right\|_{C_{t, x}^N}
    &\lesssim\left\|m_n\right\|_{H_{t, x}^{N+3}}+\sum_{i=1}^3\left\|\mathcal{N}^i_n\right\|_{H_{t, x}^{N+3}}\\
    &\lesssim \lambda_n^{N+3}\left(\|m\|_{L^2_{t, x}}+\sum_{i=1}^3\left\|\mathcal{N}^i_n\right\|_{L^2_{t, x}}\right)\lesssim\lambda_n^{N+3}\widetilde{M}\ll\lambda_{n+1}^{4N+3},
\end{aligned}
\end{equation}
which yields \eqref{eq2.7} at level $n+1$.

Finally, for the $C^1_{t,x}$-estimates \eqref{eq2.8}, in view of the Sobolev embedding $W_{t, x}^{1,5} \hookrightarrow L_{t, x}^{\infty}$ and $W_{t, x}^{5,1} \hookrightarrow L_{t, x}^{\infty}$, we obtain
\begin{align}\label{5.47-s}
\begin{aligned}
\left\|R^m_n\right\|_{C_{t, x}^1} 
&\lesssim \left\|P\left(\rho_n\right)\right\|_{W_{t, x}^{2,5}}+\left\|P(\rho) *_x \phi_{\lambda_n^{-1} }*_t \varphi_{\lambda_n^{-1}}\right\|_{W_{t, x}^{2,5}}+\left\|\lambda_n^{-2\alpha}|\nabla|^{2 \alpha-1}\left(\rho_n^{-1} m_n\right)\right\|_{W_{t, x}^{2,5}} \\
&\quad +\left\|\lambda_n^{-2\alpha} {\rm div}\left(\rho_n^{-1} m_n\right)\right\|_{W_{t, x}^{2,5}}+\left\|\rho_n^{-1} m_n \otimes m_n\right\|_{W_{t, x}^{6,1}}+\left\|\left(\rho^{-1} m \otimes m\right) *_x \phi_{\lambda_n^{-1}}*_t \varphi_{\lambda_n^{-1}}\right\|_{W_{t, x}^{6,1}} \\
&\quad+\sum_{i=1}^3\left\|\rho_n^{-1} \mathcal{N}^i_n \otimes \mathcal{N}^i_n\right\|_{W_{t, x}^{6,1}}+\sum_{i=1}^3\left\|\left(\rho^{-1} \mathcal{N}^i \otimes \mathcal{N}^i\right) *_x \phi_{\lambda_n^{-1}}*_t \varphi_{\lambda_n^{-1}}\right\|_{W_{t, x}^{6,1}} \\
&\lesssim \left\|P\left(\rho_n\right)\right\|_{C_{t, x}^2}+\lambda_n^2\|P(\rho)\|_{C_{t, x}}+\lambda_n^{-2\alpha}\left\||\nabla|^{2 \alpha-1}\left(\rho_n^{-1} m_n\right)\right\|_{W_{t, x}^{2, 5}}+\lambda_n^{-2\alpha}\left\|\rho_n^{-1} m_n\right\|_{C_{t, x}^3} \\
&\quad +\left\|\rho_n^{-1} m_n \otimes m_n\right\|_{C_{t, x}^6}+\sum_{N_1+N_2\leq6}\|\rho^{-1}\|_{C_{t,x}}\|m\|_{L^2_{t,x}}^2\|\partial_t^{N_1}\varphi_{\lambda_n^{-1}}\|_{L^1_t}\|\nabla^{N_2}\phi_{\lambda_n^{-1}}\|_{L^1_x}\\
&\quad+\sum_{i=1}^3\left\|\rho_n^{-1} \mathcal{N}^i_n \otimes \mathcal{N}^i_n\right\|_{C^6_{t,x}}+\sum_{\substack{N_1+N_2\leq6\\1\leq i\leq 3}}\|\rho^{-1}\|_{C_{t,x}}\|\mathcal{N}^i\|_{L^2_{t,x}}^2\|\partial_t^{N_1}\varphi_{\lambda_n^{-1}}\|_{L^1_t}\|\nabla^{N_2}\phi_{\lambda_n^{-1}}\|_{L^1_x}\\
&\lesssim \left\|P\left(\rho_n\right)\right\|_{C_{t, x}^2}+\lambda_n^2\|P(\rho)\|_{C_{t, x}}+\lambda_n^{-2 \alpha}\left\||\nabla|^{2 \alpha-1}\left(\rho_n^{-1} m_n\right)\right\|_{W_{t, x}^{2, 5}}\\
&\quad+\lambda_n^{-2\alpha} \sum_{N_1+N_2 \leq 3}\left\|\rho_n^{-1}\right\|_{C_{t, x}^{N_1}}\left\|m_n\right\|_{C_{t, x}^{N_2}}+\sum_{N_1+N_2+N_3 \leq 6}\left\|\rho_n^{-1}\right\|_{C_{t, x}^{N_1}}\left\|m_n\right\|_{C_{t, x}^{N_2}}\left\|m_n\right\|_{C_{t, x}^{N_3}}\\
&\quad+\sum_{i=1}^3\sum_{N_1+N_2+N_3 \leq 6}\left\|\rho_n^{-1}\right\|_{C_{t, x}^{N_1}}\left\|\mathcal{N}^i_n\right\|_{C_{t, x}^{N_2}}\left\|\mathcal{N}^i_n\right\|_{C_{t, x}^{N_3}}+\lambda_n^6,\\
\end{aligned}
\end{align}
and
\begin{align}
    \sum_{i=1}^3\left\|M^i_n\right\|_{C_{t, x}^1} 
&\lesssim\sum_{i=1}^3\left\|\left(\rho^{-1} (\mathcal{N}^i\otimes m-m\otimes \mathcal{N}^i)\right) *_x \phi_{\lambda_n^{-1}}*_t \varphi_{\lambda_n^{-1}}\right\|_{W_{t, x}^{6, 1}}\notag\\
&\quad+\sum_{i=1}^3\left\|\rho_n^{-1} (\mathcal{N}^i_n \otimes m_n-m_n\otimes \mathcal{N}^i_n)\right\|_{W_{t, x}^{6, 1}}\notag\\
&\lesssim\sum_{\substack{N_1+N_2\leq6\\1\leq i\leq 3}}\|\rho^{-1}\|_{C_{t,x}}\|\mathcal{N}^i\|_{L^2_{t,x}}\|m\|_{L^2_{t,x}}\|\partial_t^{N_1}\varphi_{\lambda_n^{-1}}\|_{L^1_t}\|\nabla^{N_2}\phi_{\lambda_n^{-1}}\|_{L^1_x} \label{5.48-s}\\
&\quad+\sum_{i=1}^3\left\|\rho_n^{-1} (\mathcal{N}^i_n \otimes m_n-m_n\otimes \mathcal{N}^i_n)\right\|_{C^6_{t,x}}\notag\\
&\lesssim\lambda_n^6+\sum_{\substack{N_1+N_2+N_3 \leq 6\\1\leq i\leq3}}\left\|\rho_n^{-1}\right\|_{C_{t, x}^{N_1}}\left\|\mathcal{N}^i_n\right\|_{C_{t, x}^{N_2}}\left\|m_n\right\|_{C_{t, x}^{N_3}}.\notag
\end{align}
By \eqref{eq5.37} and standard mollification estimates, we have
\begin{equation}\label{5.49-s}
    \begin{aligned}
\left\|P\left(\rho_n\right)\right\|_{C_{t, x}^2} \lesssim & \left\|P\left(\rho_n\right)\right\|_{C_{t, x}}+\left\|P^{\prime}\left(\rho_n\right)\right\|_{C_{t, x}}\left(\left\|\rho_n\right\|_{C_t C_x^2}+\left\|\rho_n\right\|_{C_t^2 C_x}+\left\|\rho_n\right\|_{C_t^1 C_x^1}\right) \\
& +\left\|P^{\prime \prime}\left(\rho_n\right)\right\|_{C_{t, x}}\left(\left\|\rho_n\right\|_{C_t C_x^1}^2+\left\|\rho_n\right\|_{C_t^1 C_x}^2\right) \\
\lesssim & \lambda_n^2\widetilde{M}.
\end{aligned}
\end{equation}

On the other hand,  by using the interpolation inequality gives
\begin{equation}\label{5.50-s}
    \begin{aligned}
\left\||\nabla|^{2 \alpha-1}\left(\rho_n^{-1} m_n\right)\right\|_{W_{t, x}^{2, 5}}
& \lesssim\left\|\rho_n^{-1} m_n\right\|_{C_{t, x}^{2}}^{1-\frac{2\alpha-1}{2}}\left\|\rho_n^{-1} m_n\right\|_{C_{t, x}^4}^{\frac{2\alpha-1}{2}} \\
& \lesssim\left(\sum_{N_1+N_2 \leq 2}\left\|\rho_n^{-1}\right\|_{C_{t, x}^{N_1}}\left\|m_n\right\|_{C_{t, x}^{N_2}}\right)^{1-\frac{2\alpha-1}{2}}\\
&\quad\times\left(\sum_{N_1+N_2 \leq 4}\left\|\rho_n^{-1}\right\|_{C_{t, x}^{N_1}}\left\|m_n\right\|_{C_{t, x}^{N_2}}\right)^{\frac{2\alpha-1}{2}} \\
& \lesssim\left(\widetilde{M}^2\sum_{N_1+N_2 \leq 2} \lambda_n^{N_1+N_2}\right)^{1-\frac{2\alpha-1}{2}}\left(\widetilde{M}^2\sum_{N_1+N_2 \leq 4} \lambda_n^{N_1+N_2}\right)^{\frac{2\alpha-1}{2}} \\
&\lesssim \lambda_n^{2 \alpha+1}\widetilde{M}^2.
\end{aligned}
\end{equation}
Plugging \eqref{5.46-s},  \eqref{5.49-s} and \eqref{5.50-s} into \eqref{5.47-s} and \eqref{5.48-s}, we deduce that
\begin{align*}
       \|R^m_n\|_{C^1_{t,x}}+\sum_{i=1}^3\|M^i_n\|_{C^1_{t,x}}\lesssim\lambda_n^6\ll\lambda_{n+1}^{20},
\end{align*}
which verifies \eqref{eq2.8} at level $n+1$.

Then, by the main iteration Theorem \ref{thm-main-iteration} we can obtain a sequence of relaxed solutions $\{(\rho_{n,q},m_{n,q},\mathcal{N}_{n,q})\}_{q\geq n+1}$ to \eqref{5.31ss}. Let $q\rightarrow+\infty$ we obtain a weak solution $(\rho^{(n)},m^{(n)},\mathcal{N}^{(n)})\in C_{t,x}\times (H^{\beta^{\prime}}_{t,x})^2$ to \eqref{1.3s} with parameters $\kappa_n\nu^s$ and $\kappa_n\nu^b$ for some $0<\beta^{\prime}<{\rm min}\{\widetilde{\beta},\beta/(8+\beta)\}$.

Furthermore, we can deduce from \eqref{eq2.7}, \eqref{eq2.12} that
\begin{align*}
    \left\|m^{(n)}-m\right\|_{H^{\beta^{\prime}}_{t,x}}
& \leq\left\|m^{(n)}-m_n\right\|_{H^{\beta^{\prime}}_{t,x}}+\left\|m-m_n\right\|_{H^{\beta^{\prime}}_{t,x}} \\
& \lesssim \sum_{q=n+1}^{\infty}\left\|m_{n, q+1}-m_{n, q}\right\|_{L_{t,x}^2}^{1-\frac{\beta^{\prime}}{\widetilde{\beta}}}\left\|m_{n, q+1}-m_{n, q}\right\|_{H^{\widetilde{\beta}}_{t,x}}^{\frac{\beta^{\prime}}{\widetilde{\beta}}} \\
&\quad+\left\|m-m_{n}\right\|_{L_{t,x}^2}^{1-\frac{\beta^{\prime}}{\widetilde{\beta}}}\left\|m-m_{n}\right\|_{H^{\widetilde{\beta}}_{t,x}}^{\frac{\beta^{\prime}}{\widetilde{\beta}}}\\
& \lesssim \sum_{q=n+1}^{\infty} \lambda_{q+1}^{-\beta\left(1-\beta^{\prime}\right)} \lambda_{q+1}^{8 \beta^{\prime}}+\lambda_n^{-\left(\widetilde{\beta}-\beta^{\prime}\right)}\|m\|_{H_{t, x}^{\widetilde{\beta}}} \\
& \lesssim \frac{a^{n b\left(-\beta\left(1-\beta^{\prime}\right)+8 \beta\right)}}{a^{n b\left(\beta\left(1-\beta^{\prime}\right)-8 \beta\right)}-1}+a^{-\left(\widetilde{\beta}-\beta^{\prime}\right) b n} \leq \frac{1}{n},
\end{align*}
as well as
\begin{align*}
   \left\|\mathcal{N}^{(n)}-\mathcal{N}\right\|_{H_{t,x}^{\beta^{\prime}}} \lesssim \sum_{i=1}^3 \left\|\mathcal{N}^{i,(n)}-\mathcal{N}^i\right\|_{H_{t,x}^{\beta^{\prime}}} \lesssim\frac{1}{n},
\end{align*}
and, via \eqref{eq2.11},
\begin{align*}
    \left\|\rho^{(n)}-\rho\right\|_{C_{t, x}} & \leq\left\|\rho^{(n)}-\rho_n\right\|_{C_{t, x}}+\left\|\rho-\rho_n\right\|_{C_{t, x}} \\
& \lesssim \sum_{q=n+1}^{\infty}\left\|\rho_{n, q+1}-\rho_{n, q}\right\|_{C_{t, x}}+\left\|\rho-\rho_n\right\|_{C_{t, x}} \\
& \lesssim \sum_{q=n+1}^{\infty} \delta_{q+2}^{\frac{1}{2}}+\lambda_n^{-\widetilde{\beta}}\|\rho\|_{C_{t, x}^{\widetilde{\beta}}} \\
& \lesssim \sum_{q=n+3}^{\infty} a^{-\beta b^q}+\lambda_n^{-\widetilde{\beta}} \lesssim \frac{a^{-\beta b(n+2)}}{a^{\beta b}-1}+\lambda_n^{-\widetilde{\beta}} \leq \frac{1}{n},
\end{align*}
where the last step is valid for $a$ large enough.

At last, letting $n\rightarrow+\infty$ we obtain the strong convergence \eqref{1.8-s} and complete the proof of Theorem \ref{thm-vanishing limit}.
\hfill$\square$
\subsection{Proof of Theorem \ref{thm-low mach limit}} Let $\{\phi_{\varepsilon}\}_{\varepsilon>0}$ and $\{\varphi_{\varepsilon}\}_{\varepsilon>0}$ be two families of standard compactly supported mollifiers on $\mathbb{T}^3$ and $[-T,T]$, respectively. For every $n\geq1$, set
\begin{equation}\label{eq5.51}
    \begin{aligned}     &v_n:=\left(v*_x\phi_{\lambda_n^{-1}}\right)*_t\varphi_{\lambda_n^{-1}},\quad \Pi_n:=\left(\Pi*_x\phi_{\lambda_n^{-1}}\right)*_t\varphi_{\lambda_n^{-1}},\\     &\mathbf{H}^i_n:=\left(\mathbf{H}^i*_x\phi_{\lambda_n^{-1}}\right)*_t\varphi_{\lambda_n^{-1}},\quad i=1,2,3,
    \end{aligned}
\end{equation}
restricted to $[0,T]$, where $\lambda_n:=a^{b^n}$ and $\delta_n:=\lambda^{-2\beta}_n$.

Here, we choose $a$, $b$ and $\beta$ as in \eqref{eq2.2s} and \eqref{eq2.3s}, and additionally require that 
\begin{align*}
    b^2\beta<\min \left\{\frac{1}{100},\frac{\widetilde{\beta}}{4}\right\}.
\end{align*}

Next, due to the fact that $(v,\mathbf{H})$ is a weak solution to the incompressible Oldrpyd-type model \eqref{1.5-s}, by Fourier transform and standard mollification estimates,
\begin{align}
    &\|\partial_t^M\nabla^Nv_n\|_{C_{t,x}}\lesssim\lambda_n^{M+N+2-\widetilde{\beta}}\widetilde{M},\label{5.52}\\
    &\|\partial_t^M\nabla^N\mathbf{H}^i_n\|_{C_{t,x}}\lesssim\lambda_n^{M+N+2-\widetilde{\beta}}\widetilde{M},\quad i=1,2,3\label{5.53}\\
    &\|\partial_t^M\nabla^N\Pi_n\|_{C_{t,x}}\lesssim\lambda_n^{M+N+4-2\widetilde{\beta}}\widetilde{M}^2,\label{5.54}
\end{align}
for $0\leq M,N\leq7$, where
\begin{align*}
    \widehat{M}:={\rm max}\left\{\|v\|_{H^{\widetilde{\beta}}_{t,x}},\|\mathbf{H}\|_{H^{\widetilde{\beta}}_{t,x}}\right\},
\end{align*}
and we also use the standard pressure estimates:
\begin{align*}    \|\Pi\|_{L^{\widetilde{q}}_{t,x}}\lesssim\|v\|_{H^{\widetilde{\beta}}_{t,x}}^2+\|\mathbf{H}\|_{H^{\widetilde{\beta}}_{t,x}}^2, \quad\widetilde{q}=\frac{4}{2-\widetilde{\beta}}.
\end{align*}

In order to preserve the continuous equation $\eqref{1.3s}_1$ and Piola condition $\eqref{1.3s}_4$, we choose the correctors as follows:
\begin{equation}\label{eq5.55}
  \begin{aligned}
        &\Theta_n:=\frac{P^{-1}\left(P(1)+\left(\varkappa_n{\rm Ma}\right)^2\left(\Pi_n+c_n(t)\right)\right)-1}{\left(\varkappa_n{\rm Ma}\right)^2},\\
    &\mathscr{U}_n:=-\frac{\nabla\Delta^{-1}\left(\partial_t\Theta_n+v_n\cdot\nabla\Theta_n\right)}{1+\left(\varkappa_n{\rm Ma}\right)^2\Theta_n},\\
    &\mathscr{H}^i_n:=-\frac{\Theta_n\mathbf{H}^i_n}{1+\left(\varkappa_n{\rm Ma}\right)^2\Theta_n},
  \end{aligned}
\end{equation}
where $\varkappa_n=\lambda_n^{-6}$. Furthermore, by abuse of notation, we define
\begin{equation}\label{eq5.56}
    \begin{aligned}
        &\rho_n:=1+\left(\varkappa_n{\rm Ma}\right)^2\Theta_n,\quad u_n=v_n+\left(\varkappa_n{\rm Ma}\right)^2\mathscr{U}_n,\\
        &\mathbf{F}^i_n=\mathbf{H}^i_n+\left(\varkappa_n{\rm Ma}\right)^2\mathscr{H}^i_n,\quad i=1,2,3,\\
        &m_n:=\rho_nu_n,\quad\mathcal{N}^i_n:=\rho_n\mathbf{F}^i_n,\quad i=1,2,3.
    \end{aligned}
\end{equation}
\begin{rem}
    To avoid the loss of derivatives, we perform a mollification of $v$, $\mathbf{H}$ and $\Pi$ at scale $(\varkappa_n{\rm Ma})^{\frac{1}{6}}$, which is chosen to match the well-prepared datum. Furthermore, such $\Theta_n$ can be found since the fact that the scalar function $P(\rho)$ is invertible near $P(1)$ and $\left(\varkappa_n{\rm Ma}\right)^2$ is small enough. And the time-depended function $c_n(t)$ with the decay estimates $\|c_n\|_{C^1_t}\lesssim\varkappa_n^2$, which is chosen to keep the pressure corrector $\Theta_n$ mean-free, that is, we can use $\nabla\Delta^{-1}$ to define the velocity corrector $\mathscr{U}_n$.
\end{rem}

Then, by using \eqref{5.54}, implicit function theorem and Taylor expansion, we have
\begin{align}
    \Theta_n&=(\varkappa_n{\rm Ma})^{-2}\left(\frac{(\varkappa_n{\rm Ma})^2\widetilde{\Pi}_n}{P^{\prime}(1)}-\frac{P^{\prime\prime}(1)}{2\left(P^{\prime}(1)\right)^3}(\varkappa_n{\rm Ma})^4\widetilde{\Pi}_n^2+o\left((\varkappa_n{\rm Ma})^4\widetilde{\Pi}_n^2\right)\right)\notag\\
    &=\frac{\Pi_n+c_n(t)}{P^{\prime}(1)}-\frac{P^{\prime\prime}(1)}{2\left(P^{\prime}(1)\right)^3}\lambda_n^{-12}{\rm Ma}\widetilde{\Pi}_n^2+o(\lambda_n^{-12}\widetilde{\Pi}^2_n)=\frac{\Pi_n}{P^{\prime}(1)}+O\left(\lambda_n^{-12}(\Pi_n^2)\right),\label{5.57-s}
\end{align}
where we denote $\widetilde{\Pi}_n:=\Pi_n+c_n(t)$.

Furthermore, using the interpolation inequality and \eqref{5.52}--\eqref{5.54}, we have the estimates of density correctors,
\begin{align}
    \begin{aligned}\label{5.58-s}
\quad\|\partial_t^M\nabla^N\Theta_n\|_{C_{t,x}}
       \lesssim \|\partial_t^M\nabla^N\Pi_n\|_{C_{t,x}}+O(\lambda_n^{-12})\|\partial_t^M\nabla^N(\Pi_n^2)\|_{C_{t,x}}
       \lesssim\lambda_n^{M+N+4-2\widetilde{\beta}},
   \end{aligned}
\end{align}
the estimates of velocity correctors,
\begin{align}
&\quad\|\partial_t^M\nabla^N\mathscr{U}_n\|_{C_{t,x}}
\lesssim\|\partial_t^M\nabla^N\mathscr{U}_n\|_{C_tC^{\widetilde{\beta}}_x}\notag\\
       &\lesssim\sum_{\substack{M_1+M_2=M\\N_1+N_2=N}}\left(\left(1+\lambda_n^{-12} \|\partial_t^{M_1}\nabla^{N_1}\Theta_n\|_{C_{t,x}} \right)\|\partial_t^{M_2}\nabla^{N_2-1}(\partial_t\Theta_n+v_n\cdot\nabla\Theta_n)\|_{C_{t,x}}\right)^{1-\widetilde{\beta}}\notag\\
       &\qquad\times\left(\left(1+\lambda_n^{-12} \|\partial_t^{M_1}\nabla^{N_1+1}\Theta_n\|_{C_{t,x}} \right)\|\partial_t^{M_2}\nabla^{N_2}(\partial_t\Theta_n+v_n\cdot\nabla\Theta_n)\|_{C_{t,x}}\right)^{\widetilde{\beta}}\notag\\
       &\lesssim\sum_{\substack{M_1+M_2=M\\N_1+N_2=N}}\left(1+\lambda_n^{-12} \|\partial_t^{M_1}\nabla^{N_1}\Theta_n\|_{C_{t,x}} \right)^{1-\widetilde{\beta}}\left(1+\lambda_n^{-12} \|\partial_t^{M_1}\nabla^{N_1+1}\Theta_n\|_{C_{t,x}} \right)^{\widetilde{\beta}}\notag\\
&\qquad\times\left(\|\partial_t^{M_2+1}\nabla^{N_2}\Theta_n\|_{C_{t,x}}+\sum_{\substack{M_{21}+M_{22}=M_2\\N_{21}+N_{22}=N_2}}\|\partial_t^{M_{21}}\nabla^{N_{21}-1}v_n\|_{C_{t,x}}\|\partial_t^{M_{22}}\nabla^{N_{22}}\Theta_n\|_{C_{t,x}}\right)^{1-\widetilde{\beta}}\notag\\     &\qquad\times\left(\|\partial_t^{M_2+1}\nabla^{N_2+1}\Theta_n\|_{C_{t,x}}+\sum_{\substack{M_{21}+M_{22}=M_2\\N_{21}+N_{22}=N_2}}\|\partial_t^{M_{21}}\nabla^{N_{21}}v_n\|_{C_{t,x}}\|\partial_t^{M_{22}}\nabla^{N_{22}+1}\Theta_n\|_{C_{t,x}}\right)^{\widetilde{\beta}}\notag\\
       &\lesssim\lambda_n^6\sum_{\substack{M_1+M_2=M\\N_1+N_2=N}}\left(\lambda_n^{M_2+N_2+5-2\widetilde{\beta}}+\lambda_n^{M_2+N_2+5-3\widetilde{\beta}}\right)^{1-\widetilde{\beta}}\left(\lambda_n^{M_2+N_2+6-2\widetilde{\beta}}+\lambda_n^{M_2+N_2+7-3\widetilde{\beta}}\right)^{\widetilde{\beta}}\notag\\
       &\lesssim\lambda_n^{M+N+11},\label{5.59-s}
\end{align}
and the estimates of elastic correctors,
\begin{align}
    \begin{aligned}\label{5.60-s}
\|\partial_t^M\nabla^N\mathscr{H}^i_n\|_{C_{t,x}}
       &\lesssim\sum_{\substack{M_1+M_2=M\\N_1+N_2=N}}\left(1+\lambda_n^{-12} \|\partial_t^{M_1}\nabla^{N_1}\Theta_n\|_{C_{t,x}} \right)\\     &\quad\times\left(\sum_{\substack{M_{21}+M_{22}=M_2\\N_{21}+N_{22}=N_2}}\|\partial_t^{M_{21}}\nabla^{N_{21}}\Theta_n\|_{C_{t,x}}\|\partial_t^{M_{22}}\nabla^{N_{22}}\mathbf{H}^i_n\|_{C_{t,x}}\right)\\
       &\lesssim\lambda_n^{M+N+12-3\widetilde{\beta}},\quad i=1,2,3
   \end{aligned}
\end{align}
for $1\leq M,N\leq6$. In addition, by combining \eqref{eq5.51}--\eqref{eq5.56} and \eqref{5.58-s}--\eqref{5.60-s} altogether, it follows that
\begin{align}
    &\begin{aligned}\label{5.61-s}
        \|\partial_t^M\nabla^N\rho_n\|_{C_{t,x}}&\lesssim1+\lambda_n^{-12}\lambda_n^{M+N+4-2\widetilde{\beta}},
    \end{aligned}\\
    &\begin{aligned}\label{5.62-s}
        \|\partial_t^M\nabla^Nu_n\|_{C_{t,x}}&\lesssim\lambda_n^{M+N+2-\widetilde{\beta}}+\lambda_n^{-12}\lambda_n^{M+N+5}\\
        &\lesssim\lambda_n^{M+N+2-\widetilde{\beta}},
    \end{aligned}\\
    &\begin{aligned}\label{5.63-s}
        \|\partial_t^M\nabla^N\mathbf{F}^i_n\|_{C_{t,x}}&\lesssim\lambda_n^{M+N+2-\widetilde{\beta}}+\lambda_n^{-12}\lambda_n^{M+N+6-3\widetilde{\beta}}\\
        &\lesssim\lambda_n^{M+N+2-\widetilde{\beta}},\quad i=1,2,3,
    \end{aligned}
\end{align}
for $0\leq M,N\leq6$.

Furthermore, by using \eqref{eq5.51}--\eqref{eq5.56} and the fact that $(v,\mathbf{H})$ is a weak solution to the incompressible Oldrpod-type model \eqref{1.5-s}, we obtain that $(\rho_n,m_n,\mathcal{N}_n)$ satisfies the following relaxed systems (dimensionless form):
\begin{equation}\label{5.64-s}
    \left\{
    \begin{aligned}
        &\partial_t\rho_n+{\rm div}m_n=0,\\
        &\partial_tm_n+(-\Delta)^{\alpha}(\rho^{-1}_nm_n)-\nabla{\rm div}(\rho^{-1}_nm_n)+\frac{1}{\left(\varkappa_n{\rm Ma}\right)^2}\nabla P(\rho_n)\\
        &\qquad\qquad\qquad+{\rm div}\left(\rho_n^{-1}\left(m_n\otimes m_n-\sum_{i=1}^3\mathcal{N}_n^i\otimes\mathcal{N}_n^i\right)\right)={\rm div}R^m_n,\\
        &\partial_t\mathcal{N}^i_n+{\rm div}\left(\rho_n^{-1}\left(\mathcal{N}^i_n\otimes m_n-m_n\otimes \mathcal{N}^i_n\right)\right)={\rm div}M^i_n,\quad i=1,2,3,\\
        &{\rm div}\mathcal{N}^i_n=0,\quad i=1,2,3,\\
    \end{aligned}
    \right.
\end{equation}
where $\mathcal{N}_n=\left(\mathcal{N}^1_n,\mathcal{N}^2_n,\mathcal{N}^3_n\right)$ and the Reynolds stress and elastic stress fluctuation
\begin{align}
    &\begin{aligned}\label{5.65-s}
      R^m_n:=
      &\left(v_n\otimes v_n-\sum_{i=1}^3\mathbf{H}^i_n\otimes\mathbf{H}^i_n\right)-\left(\left(v\otimes v-\sum_{i=1}^3\mathbf{H}^i\otimes\mathbf{H}^i\right)*_x\phi_{\lambda_n^{-1}}\right)*_t\varphi_{\lambda_n^{-1}}\\
      &+\left(\varkappa_n{\rm Ma}\right)^2\mathcal{R}^m\left(\partial_t\left(\Theta_nv_n\right)+\partial_tv_n+\left(\varkappa_n{\rm Ma}\right)^2\partial_t\left(\Theta_n\mathscr{U}_n\right)\right)\\
      &+\left(\varkappa_n{\rm Ma}\right)^2\left(\mathcal{R}^m(-\Delta)^{\alpha}-\mathcal{R}^m\nabla{\rm div}\right)\left(\Theta_nv_n+v_n+\left(\varkappa_n{\rm Ma}\right)^2\Theta_n\mathscr{U}_n\right)\\
      &+\left(\varkappa_n{\rm Ma}\right)^2\left(u_n\otimes u_n-v_n\otimes v_n-\sum_{i=1}^3\left(\mathbf{F}^i_n\otimes\mathbf{F}^i_n-\mathbf{H}^i_n\otimes\mathbf{H}^i_n\right)\right)\\
      &+\left(\varkappa_n{\rm Ma}\right)^2\Theta_n\left(u_n\otimes u_n-\sum_{i=1}^3\mathbf{F}^i_n\otimes\mathbf{F}^i_n\right),\\
    \end{aligned}\\
    &\begin{aligned}\label{5.66-s}
    M^i_n:=
   &\left(\mathbf{H}^i_n\otimes v_n-v_n\otimes\mathbf{H}^i_n\right)-\left(\left(\mathbf{H}^i\otimes v-v\otimes\mathbf{H}^i\right)*_x\phi_{\lambda_n^{-1}}\right)*_t\varphi_{\lambda_n^{-1}}\\
    &+\left(\varkappa_n{\rm Ma}\right)^2\mathcal{R}^n\left(\partial_t\left(\Theta_n\mathbf{H}^i_n\right)+\partial_t\mathbf{H}^i_n+\left(\varkappa_n{\rm Ma}\right)^2\partial_t\left(\Theta_n\mathscr{H}^i_n\right)\right)\\
      &+\left(\varkappa_n{\rm Ma}\right)^2\left(\left(\mathbf{F}^i_n\otimes u_n-u_n\otimes\mathbf{F}^i_n\right)-\left(\mathbf{H}^i_n\otimes v_n-v_n\otimes\mathbf{H}^i_n\right)\right)\\
      &+\left(\varkappa_n{\rm Ma}\right)^2\Theta_n\left(\mathbf{F}^i_n\otimes u_n-u_n\otimes\mathbf{F}^i_n\right),\quad i=1,2,3.\\
    \end{aligned}\\
    &\left(\mathscr{M}_n\right)_{ijk}:=\left(M^i_n\right)_{jk}.\notag
\end{align}
\begin{clm}
    For $a$ sufficiently large, $(\rho_n,m_n,\mathcal{N}_n,R^m_n,\mathscr{M}_n)$ satisfy the inductive estimates \eqref{eq2.5}--\eqref{eq2.9} at level $q=n+1$.
\end{clm}

To this end, let us start with the most delicate $L^1_{t,x}$-decay estimates \eqref{eq2.9} of $R^m_n$ and $M^i_n$. Denote that $\mathcal{K}=v\text{ or }\mathbf{H}^i$, $\mathscr{K}_n=\mathscr{U}_n\text{ or }\mathscr{H}^i_n$, $K_n=u_n\text{ or }\mathbf{F}^i_n$ and $\mathcal{R}=\mathcal{R}^m\text{ or }\mathcal{R}^n$, and then abuse the notation similar to $\mathcal{Q}\otimes\mathcal{Q}$ and $\mathcal{Q}_n\otimes\mathcal{Q}_n$. Next, by a directed calculation, we have
\begin{align}
    &\quad\|R^m_n\|_{L^1_{t,x}}+\sum_{i=1}^3\|M^i_n\|_{L^1_{t,x}}\notag\\
        &\lesssim \lambda_n^{-12}\sum\left\|\mathcal{R}\left(\partial_t\left(\Theta_n\mathcal{K}_n\right)+\partial_t\mathcal{K}_n+\lambda_n^{-12}\partial_t\left(\Theta_n\mathscr{K}_n\right)\right)\right\|_{L^1_{t,x}}\notag\\
        &\quad+\lambda_n^{-12}\left\|\left(\mathcal{R}^m(-\Delta)^{\alpha}-\mathcal{R}^m\nabla{\rm div}\right)\left(\Theta_nv_n+v_n+\lambda_n^{-12}\Theta_n\mathscr{U}_n\right)\right\|_{L^1_{t,x}}\label{5.67-s}\\
        &\quad+\lambda_n^{-12}\sum\left\|\mathrm{K}_n\otimes\mathrm{K}_n-\mathcal{K}_n\otimes\mathcal{K}_n\right\|_{L^1_{t,x}}+\lambda_n^{-12}\left\|\Theta_n^2\mathrm{K}_n\otimes\mathrm{K}_n\right\|_{L^1_{t,x}}\notag\\
        &\quad+\sum\left\|\mathcal{K}_n\otimes\mathcal{K}_n-\left(\left(\mathcal{K}\otimes \mathcal{K}\right)*_x\phi_{\lambda_n^{-1}}\right)*_t\varphi_{\lambda_n^{-1}}\right\|_{L^1_{t,x}}\notag\\
        &:=J_1+J_2+J_3+J_4+J_5.\notag
\end{align}
We start with the time-derivative terms $J_1$, by using \eqref{5.58-s}--\eqref{5.60-s} we obtain
\begin{equation}\label{5.68-s}
\begin{aligned}   J_1&\lesssim\lambda_n^{-12}\sum\left\|\mathcal{R}\left(\partial_t\left(\Theta_n\mathcal{K}_n\right)+\partial_t\mathcal{K}_n+\lambda_n^{-6}\partial_t\left(\Theta_n\mathscr{K}_n\right)\right)\right\|_{L^1_tL^2_x}\\
&\lesssim\lambda_n^{-12}\sum\left\|\partial_t\left(\Theta_n\mathcal{K}_n\right)+\partial_t\mathcal{K}_n+\lambda_n^{-12}\partial_t\left(\Theta_n\mathscr{K}_n\right)\right\|_{L^1_tL^2_x}\\
&\lesssim\lambda_n^{-12}\sum\left(\|\Theta_n\partial_t\mathcal{K}_n\|_{C_{t,x}}+\|\mathcal{K}_n\partial_t\Theta_n\|_{C_{t,x}}+\|\partial_t\mathcal{K}_n\|_{C_{t,x}}\right.\\
&\qquad\qquad\qquad+\left.\lambda_n^{-12}\|\Theta_n\partial_t\mathscr{K}_n\|_{C_{t,x}}+\lambda_n^{-12}\|\mathscr{K}_n\partial_t\Theta_n\|_{C_{t,x}}\right)\\
&\lesssim\lambda_n^{-12}\left(\lambda_n^{7-3\widetilde{\beta}}+\lambda_n^{3-\widetilde{\beta}}+\lambda_n^{-12}\lambda_n^{17-2\widetilde{\beta}}\right)\\
&\lesssim\lambda_n^{-5-2\widetilde{\beta}}.
\end{aligned}
\end{equation}
Then we consider the shear viscous and bulk viscous term $J_2$:
\begin{equation}\label{5.69-s}
    \begin{aligned}      J_2&\lesssim\lambda_n^{-12}\|\Theta_nv_n+v_n+\lambda_n^{-12}\Theta_n\mathscr{U}_n\|_{L^1_tH^2_x}\\   &\lesssim\lambda_n^{-12}\left(\sum_{N_1+N_2=2}\left(\|\Theta_n\|_{C_tC^{N_1}_x}\left(\|v_n\|_{C_tC^{N_2}_x}+\lambda_n^{-12}\|\mathscr{U}_n\|_{C_tC^{N_2}_x}\right)\right)+\|v_n\|_{C_tC^2_x}\right)\\
    &\lesssim\lambda_n^{-12}\left(\sum_{N_1+N_2=2}\lambda_n^{N_1+4-2\widetilde{\beta}}\left(\lambda_n^{N_2+2-\widetilde{\beta}}+\lambda_n^{-12}\lambda_n^{N_2+11-\widetilde{\beta}}\right)+\lambda_n^{4-\widetilde{\beta}}\right)\\
    &\lesssim\lambda_n^{-4-\widetilde{\beta}}.
    \end{aligned}
\end{equation}
Next, we concern the nonlinear terms $J_3$ and $J_4$:
\begin{equation}\label{5.70-s}
    \begin{aligned}         J_3+J_4&\lesssim\lambda_n^{-12}\left(\sum\left(\|K_n\|^2_{C_{t,x}}+\|\mathcal{K}_n\|^2_{C_{t,x}}\right)+\sum\|\Theta_n\|_{C_{t,x}}^2\|K_n\|_{C_{t,x}}^2\right)\\
    &\lesssim\lambda_n^{-12}\left(\lambda_n^{4-2\widetilde{\beta}}+\lambda_n^{12-6\widetilde{\beta}}\right)\\
    &\lesssim\lambda_n^{-2\widetilde{\beta}}.
    \end{aligned}
\end{equation}
Finally, by a similar fashion to \eqref{eq5.41}--\eqref{eq5.45} we obtain the estimates of the commutator $J_5$:
\begin{equation}\label{5.71-s}
    J_5\lesssim\lambda_n^{-2\widetilde{\beta}}.
\end{equation}

Furthermore, we plugging \eqref{5.68-s}--\eqref{5.71-s} into \eqref{5.67-s}, it follows that
\begin{align*}
    \|R^m_n\|_{L^1_{t,x}}+\sum_{i=1}^3\|M^i_n\|_{L^1_{t,x}}\lesssim\lambda_n^{-5-2\widetilde{\beta}}+\lambda_n^{-4-\widetilde{\beta}}+\lambda_n^{-\widetilde{\beta}}\leq\delta_{n+2},
\end{align*}
which means that \eqref{eq2.9} is valid at level $n+1$.

Now we turn to the inductive estimates of \eqref{eq2.5}--\eqref{eq2.8}, by \eqref{eq5.56} and \eqref{5.61-s} we get
\begin{align*}
     1-\lambda_{n+1}^{-\beta} \leq \rho_n \leq 1+\lambda_{n+1}^{-\beta}.
\end{align*}
Moreover, for $0 \leq N \leq 6$ and $M=0,1$, by using \eqref{5.61-s} again, we have
\begin{align*}
    \left\|\partial_t^M \rho_n\right\|_{C_t C_x^N} \lesssim \|\partial_t^M\nabla^N\rho_n\|_{C_tC^1_x}\lesssim1 \leq \lambda_{n+1}^{\frac{\varepsilon}{4}}.
\end{align*}
Hence, \eqref{eq2.5} and \eqref{eq2.6} are verified at level $n+1$.

Moreover, by combining \eqref{eq5.56} and \eqref{5.61-s}--\eqref{5.63-s} altogether, it leads that
\begin{align*}
    &\quad\|m_n\|_{C^N_{t,x}}+\sum_{i=1}^3\|\mathcal{N}^i_n\|_{C^N_{t,x}}\\   &\lesssim\sum_{M_1+M_2+N_1+N_2=N}\|\partial_t^{M_1}\nabla^{N_1}\rho_n\|_{C_{t,x}}\left(\|\partial_t^{M_2}\nabla^{N_2}u_n\|_{C_{t,x}}+\sum_{i=1}^3\|\partial_t^{M_2}\nabla^{N_2}\mathbf{F}^i_n\|_{C_{t,x}}\right)\\  
    &\lesssim\sum_{M_1+M_2+N_1+N_2=N}\lambda_n^{M_2+N_2+2-\widetilde{\beta}}\lesssim\lambda_n^{N+2-\widetilde{\beta}}\ll\lambda_{n+1}^{4N+3},
\end{align*}
where $0 \leq N \leq 6$ and $M=0,1$ and we also use $M+N+4\leq11$. Then we verify that \eqref{eq2.7} is valid at level $n+1$.

Finally, it remains to verify the $C^1_{t,x}$-estimates \eqref{eq2.8}, using \eqref{5.58-s}--\eqref{5.63-s}, Sobolev embedding $W^{5,1}_{t,x}\hookrightarrow L^{\infty}_{t,x}$ and H\"older inequality, similar to \eqref{5.47-s} and \eqref{5.48-s}, we obtain
\begin{align*}
    &\quad\|R^m_n\|_{C^1_{t,x}}+\sum_{i=1}^3\|M^i_n\|_{C^1_{t,x}}\\
        &\lesssim \lambda_n^{-12}\sum\left\|\mathcal{R}\left(\partial_t\left(\Theta_n\mathcal{K}_n\right)+\partial_t\mathcal{K}_n+\lambda_n^{-12}\partial_t\left(\Theta_n\mathscr{K}_n\right)\right)\right\|_{W^{6,1}_{t,x}}\\
        &\quad+\lambda_n^{-12}\left\|\left(\mathcal{R}^m(-\Delta)^{\alpha}-\mathcal{R}^m\nabla{\rm div}\right)\left(\Theta_nv_n+v_n+\lambda_n^{-12}\Theta_n\mathscr{U}_n\right)\right\|_{W^{6,1}_{t,x}}\\
        &\quad+\lambda_n^{-12}\sum\left\|\mathrm{K}_n\otimes\mathrm{K}_n-\mathcal{K}_n\otimes\mathcal{K}_n\right\|_{W^{6,1}_{t,x}}+\lambda_n^{-12}\left\|\Theta_n^2\mathrm{K}_n\otimes\mathrm{K}_n\right\|_{W^{6,1}_{t,x}}\\
        &\quad+\sum\left\|\mathcal{K}_n\otimes\mathcal{K}_n-\left(\left(\mathcal{K}\otimes \mathcal{K}\right)*_x\phi_{\lambda_n^{-1}}\right)*_t\varphi_{\lambda_n^{-1}}\right\|_{W^{6,1}_{t,x}}\\
        &\lesssim\lambda_n^{10}+\lambda_n^{4-2\widetilde{\beta}}+\lambda_n^{2-4\widetilde{\beta}}+\lambda_n^{14-4\widetilde{\beta}}\ll\lambda_{n+1}^{20}
\end{align*}
which yields \eqref{eq2.8} at level $n+1$.

Then, by the main iteration Theorem \ref{thm-main-iteration} we can obtain a sequence of relaxed solutions $\{(\rho_{n,q},m_{n,q},\mathcal{N}_{n,q})\}_{q\geq n+1}$ to \eqref{5.64-s}. Let $q\rightarrow+\infty$ we obtain a weak solution $(\rho^{(n)},m^{(n)},\mathcal{N}^{(n)})\in C_{t,x}\times (H^{\beta^{\prime}}_{t,x})^2$ to \eqref{1.4ss} with Mach number $\kappa_nMa$ for some $0<\beta^{\prime}<{\rm min}\{\widetilde{\beta},\beta/(8+\beta)\}$.

Furthermore, similar to the proof of Theorem \ref{thm-vanishing limit}, we can deduce from \eqref{eq2.7}, \eqref{eq2.12} that
\begin{align*}
    \left\|m^{(n)}-m\right\|_{H^{\beta^{\prime}}_{t,x}}
\lesssim \frac{1}{n},\quad\left\|\mathcal{N}^{(n)}-\mathcal{N}\right\|_{H_{t,x}^{\beta^{\prime}}} \lesssim\frac{1}{n},\quad\left\|\rho^{(n)}-\rho\right\|_{C_{t, x}}\lesssim\frac{1}{n}
\end{align*}
where the inequality is valid for $a$ large enough.

At last, letting $n\rightarrow+\infty$ we obtain the strong convergence \eqref{1.9-s} and complete the proof of Theorem \ref{thm-low mach limit}, and we have done.
\hfill$\square$
\\
\\
\noindent\textbf{Author Contributions.} Haobin Li, Peng Qu, Zirong Zeng and Mingxin Zhang wrote the manuscript. All authors worked equally on producing the article.
\\
\\
\noindent\textbf{Conflict of interest.} The authors declare that they have no conflict of interest.
\\
\\
\noindent\textbf{Acknowledgement.} H. Li, P. Qu, and M. Zhang are partially supported by NSFC grants (No. 62588101, 12431007). Z. Zeng is partially supported by the NSFC grants (No. 12501281) and the Natural Science Foundation of Jiangsu Province (No. SBK2024043113).
\\
\\
\appendix
\section{Oldroyd-type model in conservation-law form}\label{Appendix_A}
In this appendix, we provide a detailed derivation of \eqref{1.3s}, in conservation-law form, from \eqref{eq1.1}. In particular, we focus on deriving the conservation-law form of the elastic equations $\eqref{1.3s}_3$. 

First, we write $\eqref{eq1.1}_3$ in the following component form:
\begin{align}\label{A1}
    \partial_t\mathbf{F}^i+u\cdot\nabla\mathbf{F}^i-\mathbf{F}^i\cdot\nabla u=0,\quad i=1,2,3,
\end{align}
where $\mathbf{F}=(\mathbf{F}^1,\mathbf{F}^2,\mathbf{F}^3)$. Then by Leibniz's law, for $i=1,2,3,$ we have:
\begin{align}\label{A2}
\partial_t(\rho\mathbf{F}^i)=\rho\partial_t\mathbf{F}^i+\left(\partial_t\rho\right)\mathbf{F}^i,
\end{align}
and
\begin{equation}\label{A3}
\begin{aligned}
    &\quad{\rm div}\left(\left(\rho\mathbf{F}^i\right)\otimes u-u\otimes\left(\rho\mathbf{F}^i\right)\right)\\
    &=u\cdot\nabla\left(\rho\mathbf{F}^i\right)+{\rm div}u\left(\rho\mathbf{F}^i\right)-\left(\left(\rho\mathbf{F}^i\right)\cdot\nabla u\right)-{\rm div}\left(\rho\mathbf{F}^i\right)u\\
    &=\rho\left(u\cdot\nabla\mathbf{F}^i-\mathbf{F}^i\cdot\nabla u\right)+\left(u\cdot\nabla\rho+\rho{\rm div}u\right)\mathbf{F}^i-{\rm div}\left(\rho\mathbf{F}^i\right)u.
\end{aligned}
\end{equation}
Furthermore, by using the continuous equations $\eqref{eq1.1}_1$ and Piola condition $\eqref{eq1.1}_4$, we obtain
\begin{equation}\label{A4}
\left(u\cdot\nabla\rho+\rho{\rm div}u\right)\mathbf{F}^i-{\rm div}\left(\rho\mathbf{F}^i\right)u=-\left(\partial_t\rho\right)\mathbf{F}^i,\quad i=1,2,3.
\end{equation}
Finally, plugging \eqref{A4} into \eqref{A3} and using \eqref{A1} and \eqref{A2} we get
\begin{align*}
    \partial_t(\rho\mathbf{F}^i)+{\rm div}\left(\left(\rho\mathbf{F}^i\right)\otimes u-u\otimes\left(\rho\mathbf{F}^i\right)\right)=0,\quad i=1,2,3,
\end{align*}
which coincides with $\eqref{1.3s}_3$.
\section{Nondimensionalization of Oldroyd-type model}\label{Appendix_B}
In this appendix, we will derive the dimensionless form of \eqref{1.3s}. To begin with, we choose the following dimensionless variables:
\begin{equation}\label{B1}
    \begin{aligned}
     &x_*:=\frac{x}{L},\quad u_*:=\frac{u}{U},\quad t_*:=\frac{tU}{L},\quad\rho_*:=\frac{\rho}{\varrho_0},\\
     &\mathbf{F}_*^i:=\frac{\mathbf{F}^i}{\sqrt{G_0}},\quad i=1,2,3,
\end{aligned}
\end{equation}
where $L$, $U$, $\varrho_0$ and $G_0$ are characteristic length, characteristic velocity, characteristic density and Young's elastic modulus respectively, and $\frac{L}{U}$ is the convective time scale. Furthermore, we have
\begin{equation}\label{B2}
    \begin{aligned}
    &m_*:=\rho_*u_*=\frac{m}{U\varrho_0},\quad P_*(\rho_*):=\frac{P(\varrho_0\rho_*)-P(\varrho_0)}{K_0}:=\frac{P(\varrho_0\rho_*)-P(\varrho_0)}{\varrho_0P^{\prime}(\varrho_0)},\\
    &\mathcal{N}_*^i:=\rho_*\mathbf{F}_*^i=\frac{\varrho_0\mathbf{F}_*^i}{G_0},\quad i=1,2,3,
\end{aligned}
\end{equation}
where $K_0$ is the bulk modulus. Then we define the dimensionless numbers as follow
\begin{align}
    &{\rm Re}_{s,\alpha}:=\frac{\varrho_0UL^{2\alpha-1}}{\nu^s\ell_{\alpha}^{2\alpha-2}}\quad\text{(shear Reynolds number)},\label{B3}\\
        &{\rm Re}_b:=\frac{\varrho_0UL}{\nu^b}\quad\text{(bulk Reynolds number)},\label{B4}\\
        &{\rm Ma}:=\sqrt{\frac{\varrho_0U^2}{K_0}}\quad\text{(Mach number)},\label{B5}\\
        &{\rm Ca}:=\frac{\varrho_0U^2}{G_0}\quad\text{(Cauchy number)},\label{B6}
\end{align}
where $\ell_{\alpha}$ is the characteristic material length with $\ell_{\alpha}=1$ whenever $\alpha=1$.

Next, by a directed calculation, we have
\begin{align}\label{B7}
    \partial_t=\frac{U}{L}\partial_{t_*},\quad \nabla=\frac{1}{L}\nabla_{x_*},\quad {\rm div}=\frac{1}{L}{\rm div}_{x_*},\quad(-\Delta)^{\alpha}=L^{-2\alpha}(-\Delta_{x_*})^{\alpha}.
\end{align}
Plugging \eqref{B7} into \eqref{1.3s} and using \eqref{B1}--\eqref{B6}, we have
\begin{align*}
\left\{
    \begin{aligned}
    &\partial_{t_*}\rho_*+{\rm div}_{x_*}m_*=0,\\
        &\partial_{t_*}m_*+\frac{1}{{\rm Re}_{s,\alpha}}(-\Delta_{x_*})^{\alpha}(\rho_*^{-1}m_*)-(\frac{1}{{\rm Re}_b}+\frac{1}{3{\rm Re}_{s,\alpha}})\nabla_{x_*}{\rm div}_{x_*}(\rho_*^{-1}m_*)+\frac{1}{{\rm Ma}^2}\nabla_{x_*} P_*(\rho_*)\\
        &\qquad\quad\qquad\qquad\qquad+{\rm div}_{x_*}\left(\rho_*^{-1}\left(m_*\otimes m_*-\frac{1}{{\rm Ca}}\mathcal{N}_*\mathcal{N}_*^{{\rm T}}\right)\right)=0,\\
        &\partial_{t_*}\mathcal{N}_*^i+{\rm div}_{x_*}\left(\rho_*^{-1}\left(\mathcal{N}_*^i\otimes m_*-m_*\otimes \mathcal{N}_*^i\right)\right)=0,\quad i=1,2,3,\\
        &{\rm div}_{x_*}\mathcal{N}_*^i=0,\quad i=1,2,3.
\end{aligned}
\right.
\end{align*}
Let us abuse the notation and set ${\rm Re}_{s,\alpha}=1$, ${\rm Re}_b=\frac{3}{2}$ and ${\rm Ca}=1$, we obtain \eqref{1.3s} as dimensionless form:
\begin{align*}
    \left\{
    \begin{aligned}
        &\partial_t\rho+{\rm div}m=0,\\
        &\partial_tm+{\rm div}\left(\rho^{-1}\left(m\otimes m-\mathcal{N}\mathcal{N}^{{\rm T}}\right)\right)+(-\Delta)^{\alpha}(\rho^{-1}m)-\nabla{\rm div}(\rho^{-1}m)+\frac{1}{{\rm Ma}^2}\nabla P(\rho)=0,\\
        &\partial_t\mathcal{N}^i+{\rm div}\left(\rho^{-1}\left(\mathcal{N}^i\otimes m-m\otimes \mathcal{N}^i\right)\right)=0,\quad i=1,2,3,\\
        &{\rm div}\mathcal{N}^i=0,\quad i=1,2,3.
    \end{aligned}
    \right.
\end{align*}
\section{Proof of Lemma \ref{lem-2nd-geometric} (Second Geometric Lemma)}\label{Appendix_C}
In this appendix, we will prove Lemma \ref{lem-2nd-geometric} which provide a new canceled mechanism to overcome the difficulty caused by  the relative rigidity of pressure.

\begin{proof}
Choose $v_i\in\mathbb{S}^2\cap\mathbb{Q}^3$ as
\begin{align*}
 v_1&=(1,0,0), & v_2&=(0,1,0), & v_3&=(0,0,1),\\
 v_4&=\left(\frac35,\frac45,0\right), &
 v_5&=\left(\frac35,0,\frac45\right), &
 v_6&=\left(0,\frac35,\frac45\right),
\end{align*}
and set $V_i:=v_i\otimes v_i$. Consequently,
\begin{align*}
    \det(V_1\ V_2\ \cdots\ V_6)=\left(\frac{12}{25}\right)^3\neq0,
\end{align*}
which means that $\{V_i\}_{i=1}^6$ is a basis of $\mathbb{M}_{{\rm sym},3}$.

For any $S=(s_{ij})\in\mathbb{M}_{{\rm sym},3}$, it can be written uniquely
\begin{align*}
    S=\sum_{i=1}^6 r_i(S)V_i,
\end{align*}
where the $r_i$ are all linear functionals. More precisely,
\begin{align*}
 r_4(S)&=\frac{25}{12}s_{12}, &
 r_5(S)&=\frac{25}{12}s_{13}, &
 r_6(S)&=\frac{25}{12}s_{23},\\
 r_1(S)&=s_{11}-\frac34(s_{12}+s_{13}), &
 r_2(S)&=s_{22}-\frac43s_{12}-\frac34s_{23}, &
 r_3(S)&=s_{33}-\frac43(s_{13}+s_{23}).
\end{align*}

For each $i$, we choose rational unit vectors $a_i,b_i\in v_i^\perp$ such that $(a_i,b_i,v_i)$ is orthonormal. For example, we may choose
\begin{align*}
 (a_1,b_1)&=(e_2,e_3), &
 (a_2,b_2)&=(e_1,e_3), &
 (a_3,b_3)&=(e_1,e_2),\\
 (a_4,b_4)&=\left(\left(-\frac45,\frac35,0\right),e_3\right), &
 (a_5,b_5)&=\left(\left(-\frac45,0,\frac35\right),e_2\right), &
 (a_6,b_6)&=\left(\left(0,-\frac45,\frac35\right),e_1\right).
\end{align*}
And the we define
\begin{align*}
     \Xi_i(t):=\frac{1-t^2}{1+t^2}a_i+\frac{2t}{1+t^2}b_i,\quad t\in\mathbb{Q}.
\end{align*}
Furthermore, one can easily check that $\Xi_i(t)\in \mathbb{M}_{{\rm sym},3}^2\cap\mathbb{Q}^3\cap v_i^\perp$.  Next, we choose rational parameters $t_i^+$ and $t_i^-$ such that the $12$ vectors
\begin{align*}
    k_i^+:=\Xi_i(t_i^+),\quad \Xi_i^-:=q_i(t_i^-),\qquad i=1,\cdots,6,
\end{align*}
are pairwise distinct and do not belong to $\Lambda_{\mathcal N}$.  Set
\begin{align*}
    \Lambda_{m^1}:=\{k_i^+:1\leq i\leq6\},
  \quad
  \Lambda_{m^2}:=\{k_i^-:1\leq i\leq6\}.
\end{align*}
For $k_i^+$ choose the rational orthonormal frame
\begin{align*}
     (k_i^+,k_{i,1}^+,k_{i,2}^+)
  :=(k_i^+,v_i,k_i^+\times v_i),
\end{align*}
and for $k_i^-$ choose
\begin{align*}
    (k_i^-,k_{i,1}^-,k_{i,2}^-)
  :=(k_i^-,v_i\times k_i^-,v_i),
\end{align*}
which leads that $\Lambda_{m^1}\cap\Lambda_{m^2}=\varnothing$ and
$\Lambda_m\cap\Lambda_{\mathcal N}=\varnothing$.

Since the $r_i$ are continuous linear functionals and $r_i(0)=0$, we may choose $\varepsilon_m>0$ so small that
\begin{align*}
     |r_i(S)|<1,
  \quad i=1,\cdots,6,\quad
  S\in B_{\mathrm{sym},3}(0,\varepsilon_m).
\end{align*}
Define
\begin{align*}
 \gamma_{(k_i^+)}(S)&:=\left(1+\frac12r_i(S)\right)^{1/2},\\
 \gamma_{(k_i^-)}(S)&:=\left(1-\frac12r_i(S)\right)^{1/2}.
\end{align*}
 Then $\gamma_{(k_i^+)}$ and $\gamma_{(k_i^-)}$ are all strictly positive and smooth on
$B_{\mathrm{sym},3}(0,\varepsilon_m)$. Moreover, we have,
\begin{align*}
 &\sum_{k\in\Lambda_{m^1}}\gamma_{(k)}^2(S)\,k_1\otimes k_1
 -\sum_{k\in\Lambda_{m^2}}\gamma_{(k)}^2(S)\,k_2\otimes k_2\\
 &\quad=\sum_{i=1}^6
 \left[\left(1+\frac12r_i(S)\right)
       -\left(1-\frac12r_i(S)\right)\right]V_i\\
 &\quad=\sum_{i=1}^6r_i(S)V_i=S,
\end{align*}
which provides \eqref{eq3.3} and completes the proof.
\end{proof}
\bibliography{reference}

\bibliographystyle{acm}
\end{document}